\documentclass{article}
\usepackage{amssymb,graphicx}
\usepackage[utf8]{inputenc}
\usepackage{enumerate}
\usepackage{color}
\usepackage{amsmath,amsthm}
\usepackage{mathrsfs}
\usepackage{enumitem}
\RequirePackage[left=1in,right=1in,top=1in,bottom=1in]{geometry}

\numberwithin{equation}{section}

\theoremstyle{plain}
\newtheorem{thm}{Theorem}[section]

\newtheorem{hypoth}{Hypothesis}[section]
\newtheorem{lem}{Lemma}[section]
\newtheorem{prop}{Proposition}[section]

\newtheorem{defs}{Definition}[section]
\theoremstyle{definition}
\newtheorem{rmk}{Remark}[section]

\newcommand{\tr}{\mathbb{T}^d}
\newcommand{\rd}{\mathbb{R}^d}

\begin{document}

\title{Lipschitz continuity of solutions to drift-diffusion equations in the presence of nonlocal terms}
\date{\today}
\author{Hussain Ibdah\footnote{Department of Mathematics, Texas A\&M University, College Station, TX 77843, USA (hibdah@tamu.edu)}}
{\let\newpage\relax\maketitle}
\begin{abstract}
We analyze the propagation of Lipschitz continuity of solutions to various linear and nonlinear drift-diffusion systems, with and without incompressibility constraints. Diffusion is assumed to be either fractional or classical. Such equations model the incompressible Navier-Stokes systems, generalized viscous Burgers-Hilbert equation and various active scalars. We derive conditions that guarantee the propagation of Lipschitz regularity by the incompressible NSE in the form of a non-local, one dimensional viscous Burgers-type inequality. We show the analogous inequality is always satisfied for the generalized viscous Burgers-Hilbert equation, in any spatial dimension, leading to global regularity. We also obtain a regularity criterion for the Navier-Stokes equation with fractional dissipation $(-\Delta)^{\alpha}$, regardless of the power of the Laplacian $\alpha\in(0,1]$, in terms of H\"older-type assumptions on the solution. Such a criterion appears to be the first of its kind when $\alpha\in(0,1)$. The assumptions are critical when $\alpha\in[1/2,1]$, but sub-critical when $\alpha\in(0,1/2)$. Furthermore, we prove a partial regularity result under supercritical assumptions, which is upgraded to a regularity criterion if we consider the pressure-less drift-diffusion problem when $\alpha\in(1/2,1]$. That is, a certain H\"older super-criticality barrier is broken when considering a drift-diffusion equation without incompressibility constraints (no pressure term), which to our knowledge was never done before. Depending on the scenario, our results either improve on, generalize or provide different proofs to previously known regularity results for such models. The technique we use builds upon the evolution of moduli of continuity as introduced by Kiselev, Nazarov, Volberg and Shterenberg.
\end{abstract}

\section{Introduction}
\subsection{Qualitative summary of the results}
In this work, we analyze the Lipschitz regularity of solutions to drift-diffusion equations with and without incompressibility constraints and/or nonlocal perturbations. We will consider the case of classical and fractional dissipation, obtaining various new regularity, conditional regularity and partial regularity results, under critical and, in some cases even supercritical, H\"older-type assumptions on the drift velocity. We will provide a different proof to some known results, improve on them in certain cases as well as obtain new results in other cases. The models we analyze include the incompressible Navier-Stokes system, the supercritical surface quasi-geostrophic equation and a multidimensional, generalized Burgers-Hilbert model. The technique we use is based upon the idea of tracking the evolution of Lipschitz moduli of continuity as introduced by Kiselev, Nazarov, Volberg and Shterenberg, where they were able to control the local behavior of active scalars evolving under the critically dissipative surface quasi-geostrophic equation in \cite{KNV2007} and the critically dissipative fractional Burgers equation in \cite{KNS2008}. As a followup to a recent work of ours \cite{Ibdah2020b}, we extend such ideas from their current domain of scalar equations to the previously mentioned advection-diffusion systems in any spatial dimensions $d\geq3$ and in the absence of physical boundaries. A key ingredient used in analyzing the pressure term when studying the incompressible NSE is a subtle observation regarding its regularity made by Silvestre in an unpublished work \cite{Silvestre2010unpub}, as well as Constantin \cite{Constantin2014}, Isett \cite{isett2013regularity} (see also Isett and Oh \cite{IO2016}) and De Lellis and Sz\'ekelyhidi Jr. \cite{DS2014}. See also \cite{CDR2020} for further results in this direction. We formulate our results and motivate our work in \S\ref{secmainres}. For now, we summarize our results and compare with previous works.

Recall the classical initial-value problem formulation of the incompressible Navier-Stokes (NS) system in the absence of physical boundaries,
\begin{equation}\label{maineq}
	\begin{cases}
		\partial_tu(t,x)-\Delta u(t,x)=(u\cdot\nabla)u(t,x)+\nabla p(t,x),\\
		\nabla\cdot u(t,x)=0,\\
		u(0,x)=u_0(x),
	\end{cases}
\end{equation}
 where $u_0:\rd\rightarrow\rd$ is a given, smooth divergence-free vector field. By locally analyzing this nonlinear, nonlocal system of equations, we derive a criterion for the preservation of Lipschitz moduli of continuity in this scenario, see Hypothesis \ref{hypoth} and Theorem \ref{thm1} below.  \emph{We emphasize that this is not a solution to the global regularity problem}; Hypothesis \ref{hypoth} is best interpreted as one of the many ``regularity criteria'' available for the incompressible Navier-Stokes system. It may be worthwhile noting that it is highly unlikely that our analysis is applicable for the averaged NSE introduced by Tao \cite{Tao2016}, or related shell-models. Our analysis is localized in space, and thus we will be explicitly using the pointwise structure of the nonlinearity and incompressibility in our work, rather than relying on energy-type arguments. In particular, we will attempt to prove a ``maximum principle'', and so the pointwise structure of the equation is crucial.
 
 As will be demonstrated later on, the difficulty in dealing with the pressure term is the fact that it is related to the velocity vector field via a nonlinear order one operator, namely 
 \[
 \nabla p=\sum_{i,j=1}^{d}\nabla R_iR_j(u_iu_j),
 \]
 where $\{R_i\}_{i=1}^d$ are the Riesz transforms. Thus, to demonstrate the applicability of the ideas presented herein, we will consider various simplifications of the pressure term. We start by replacing the pressure term and incompressibility constraint in \eqref{maineq} by a simpler, singular integral operator of order zero, $\mathcal{N}$, and consider the resulting model in the periodic setting
\begin{equation}\label{maineq2}
\begin{cases}
	\partial_t u(t,x)-\Delta u(t,x)=(u\cdot\nabla)u(t,x)+\mathcal{N}u(t,x),\\
	u(0,x)=u_0(x).
\end{cases}
\end{equation}
The difficulty in analyzing this model in dimensions larger than one is the lack of any a-priori estimates, not even $L^2$ due to the lack of any incompressibility constraints. This model is a viscous generalization of the Burgers-Hilbert model, where the latter is the following one-dimensional equation
\begin{equation}\label{BH}
\partial_tu=u\partial_xu+H[u],
\end{equation}
with $H$ being the Hilbert-transform. Equation \eqref{BH} was first introduced by Marsden and Weinstein \cite{MW1983} to model a quadratic approximation for the motion of the boundary of a vortex patch. It was also derived by Biello and Hunter \cite{BH2010} as a model for waves with constant nonzero linearized frequency. Since then it has attracted a lot of attention in the literature, for instance \cite{BN2014, BZ2017, CCG2010, HI2012, HITW2015, KV2020}. In particular, just like the inviscid Burgers equation, model \eqref{BH} is known to develop a singularity in finite time, as was shown in \cite{CCG2010} and more recently in the preprint \cite{Yang2020}. We were also informed via private communication that Pasqualotto and Oh \cite{PO2021} recently obtained a different proof of this fact, even when adding fractional dissipation of the form $(-\Delta)^{\alpha}$, $\alpha\in(0,1/2)$. We show that adding classical dissipation would result in a globally well-posed problem that would never develop any singularities in finite time provided the initial data is smooth, see Theorem \ref{thm2} below for precise statements. Of course this result is much more interesting in high dimensions, it is trivial when the dimension is one. To our knowledge, this result in new, and as mentioned previously, not at all obvious due to the lack of any a-priori bounds in high dimensions.

Next, we will analyze linear drift-diffusion systems with and without incompressibility constraints, with classical and fractional diffusion. Namely, for $\alpha\in(0,1]$, we will consider
\begin{equation}\label{classicalddintro}
\begin{cases}
	\partial_t u(t,x)+(-\Delta)^{\alpha} u(t,x)=(b\cdot\nabla)u(t,x), \\
	u(0,x)=u_0(x),
\end{cases}
\end{equation}
as well as 
\begin{equation}\label{maineq3intro}
\begin{cases}
	\partial_t u(t,x)+(-\Delta)^{\alpha} u(t,x)=(b\cdot\nabla)u(t,x)+\nabla p(t,x), \\
	\nabla\cdot u(t,x)=0, \\
	u(0,x)=u_0(x),
\end{cases}
\end{equation}
where $(-\Delta)^{\alpha}$ is the fractional Laplacian; the nonlocal operator of order $2\alpha$ whose Fourier symbol is $|\zeta|^{2\alpha}$ and having the pointwise representation
\begin{equation}\label{fraclap}
(-\Delta)^{\alpha}\theta (x):=C_{d,\alpha}P.V.\int_{\rd}\frac{\theta(x)-\theta(x-z)}{|z|^{d+2\alpha}}\ dz.
\end{equation}
Here, $u_0$ is assumed to be a given, smooth vector-field and $b:[0,T]\times\rd\rightarrow\rd$ is also a given vector field, \emph{which may or may not be incompressible when considering \eqref{classicalddintro}}, but is assumed to be incompressible in the case of \eqref{maineq3intro}. The question that we will be interested in is the following: under what assumptions on $b$ can we get some degree of regularity on the solution $u$? Notice that both \eqref{classicalddintro} and \eqref{maineq3intro} have a natural scale invariance: given a vector field $b$, an $\epsilon>0$ and a solution to either, if we define $u_{\epsilon}(t,x):=\epsilon^{2\alpha-1}u(\epsilon^{2\alpha}t,\epsilon x)$, $b_{\epsilon}(t,x):=\epsilon^{2\alpha-1}b(\epsilon^{2\alpha}t,\epsilon x)$ and $p_{\epsilon}(t,x):=\epsilon^{4\alpha-2}p(\epsilon^{2\alpha}t,\epsilon x)$, then the rescaled functions solve the same equations (with a rescaled initial data). A general rule of thumb is that if the drift-velocity $b$ is assumed to lie in a critical or subcritical space, meaning that it lies in some normed space $X$ such that $\|b_{\epsilon}\|_X=\epsilon^{r}\|b\|_{X}$ for some $r\geq0$, then one should expect some degree of regularity on the solution $u$. Throughout the years this has been demonstrated by various authors and in various spaces, the list is far too long to include here. In this work, we will say that a solution to either \eqref{classicalddintro} or \eqref{maineq3intro} is regular on $[0,T]$ provided 
\begin{equation}\label{defregintro}
\int_0^T\|\nabla u(t,\cdot)\|_{L^{\infty}}dt<\infty.
\end{equation}
The above quantity is critical with respect to the previously mentioned scaling invariance, regardless of the value of $\alpha\in(0,1]$. Our motivation behind characterizing regularity via \eqref{defregintro} stems from the fact that controlling such a quantity implies control of higher order norms when considering nonlinear equations such as Burgers equation, surface quasi-geostrophic equation, Euler, NSE and related models. This regularity criterion is usually referred to as the ``weak'' Beale-Kato-Majda criterion, where the latter involves controlling the weaker quantity
\[
\int_0^T\|\nabla \times u(t,\cdot)\|_{L^{\infty}}dt<\infty.
\]
We refer the reader to the classical paper \cite{BKM1984} and the more recent ones \cite{CCW2011, CCW2012, Dai2017} for more details regarding this remark. Though the fractional NSE may not be a physically relevant model in and off itself, we still find it to be mathematically interesting, as does various members of the mathematical community. Understanding the fractional dissipative case may aid in understanding the classical one, and this model is indeed an active area of research, see for instance \cite{Chae2007, LMZ2019, Wu2006, YZ2012} and the references therein.

Controlling $\|\nabla u(t,\cdot)\|_{L^{\infty}}$ will be achieved by propagating Lipschitz moduli of continuity, and as the analysis that we perform is localized, we would need to assume local bounds on $b(t,\cdot)$. To be more specific, our main assumption will be the existence of a function $g\in L^{p}(0,T)$, some  $p>0$ (not to be confused with the pressure $p$ in NSE, the distinction will be clear from the context), and a $\beta\in[0,1)$ such that 
\begin{equation}\label{condbintro}
\sup_{x\neq  y}\frac{|b(t,x)-b(t,y)|}{|x-y|^{\beta}}\leq g(t),\quad a.e.\ t\in[0,T].
\end{equation}
When $\beta=0$, \eqref{condbintro} is replaced with $\|b(t,\cdot)\|_{L^{\infty}}\leq g(t)$. Due to the fact that Riesz transforms are embedded into the pressure term, when analyzing \eqref{maineq3intro}, we will always have to require $\beta>0$ in this case. That being said, and following the scale invariance, one would expect the quantity \eqref{defregintro} to be under control whenever $b$ lies in a critical or subcritical space. That is, whenever $b\in L_t^{p}C_x^{0,\beta}([0,T]\times\rd)$, for $p\geq p^*$, with 
\begin{equation}\label{defpstarintro}
p^*:=\frac{2\alpha}{2\alpha+\beta-1},
\end{equation}
being the critical exponent. We would need to further assume that the bounding function $g$ in  \eqref{condbintro} is non-decreasing. This would imply that  for almost every $t\in[0,T]$, we have 
\[
\sup_{s\in[0,t]}[b(s,\cdot)]_{C_x^{0,\beta}}\leq g(t),
\]
which is a slightly stronger assumption than simply requiring $b\in L_t^pC_x^{0,\beta}([0,T]\times\rd)$. Nevertheless, \emph{we emphasize that we do not assume boundedness of the function $g$: it merely is a non-decreasing $L^p$ function}. For instance, bounds of the form $(T-t)^{-\kappa}$ are allowed, for appropriately chosen $\kappa>0$. Thus, with our requirements, it seems more natural to replace the finiteness of the standard $L_t^pC_x^{0,\beta}$ (semi)-norm with the requirement that 
\begin{equation}\label{altcondbintro}
\left(\int_0^T\sup_{s\in[0,t]}[b(s,\cdot)]_{C_x^{0,\beta}}^pdt\right)^{1/p}<\infty.
\end{equation}
We invite the reader to readily verify that the above quantity \emph{scales exactly like the standard $L_t^pC_x^{0,\beta}$ (semi)-norm.} With those remarks in mind, we show that if 
\begin{equation}\label{defgammaintro}
\gamma:=\frac{1}{2\alpha+\beta-1}=\frac{p^*}{2\alpha},
\end{equation}
then solutions to \eqref{classicalddintro} satisfy 
\begin{equation}\label{bduddintro}
\|\nabla u(t,\cdot)\|_{L^{\infty}}\lesssim g^{\gamma}(t),\quad a.e.\ t\in[0,T],
\end{equation}
while solutions to \eqref{maineq3intro} satisfy 
\begin{equation}\label{bduNSEintro}
\|\nabla u(t,\cdot)\|_{L^{\infty}}\lesssim g^{\gamma}(t)\exp\left(\int_0^tg^{p^*}(t)dt\right),\quad a.e.\ t\in[0,T],
\end{equation}
see Theorems \ref{thm3} and \ref{thm4} below. Notice that when $\alpha\in(1/2,1]$, the parameter $\gamma$ is below the critical level $p^*$, and so to get regularity (satisfy condition \eqref{defregintro}) of solutions to \eqref{classicalddintro}, it is sufficient to make a \emph{supercritical assumption} on the drift velocity, which is rather surprising. Unfortunately at this point, we cannot break the analogous criticality barrier for solutions to \eqref{maineq3intro}, and so in this case we have to stick to the critical assumption (the limitations come from the exponential term). The best we were able to achieve, so far, is a partial regularity result of the form 
\[
\|\nabla u(t,\cdot)\|_{L^{\infty}}\lesssim \lambda(t)\log(\lambda(t)),\quad a.e.\  t\in[0,T],
\]
where 
\[
\log(\lambda(t))\approx 
\begin{cases}
g^{\gamma}(t)\exp\left(\int_0^tg^{1-\gamma\beta}(s)ds\right), &\alpha\in[1/2,1],\\
g^{(1-\beta)\gamma}\exp\left(\int_0^tg(s)ds\right),&\alpha\in(0,1/2),
\end{cases}
\]
meaning that 
\begin{equation}\label{partregintro}
\int_0^T\log\left(\|\nabla u(t,\cdot)\|_{L^{\infty}}\right)dt<\infty.
\end{equation}
See discussion towards the end of \S\ref{secimplication} below for the importance of this last remark as it relates to the global regularity problem of the classical incompressible Navier-Stokes system.
\subsection{Previous work and implications of our results}\label{secimplication}
The analysis of \eqref{maineq3intro} with drifts in $L_t^{p^*}C_x^{0,\beta}$ and classical diffusion ($\alpha=1$) was considered previously by Silvestre and Vicol in \cite{SV2012}, where they showed that H\"older regularity of initial data does persist on $[0,T]$ provided $b\in L_t^{p^*}C_x^{0,\beta}$ (actually they consider more general spaces than $C_x^{0,\beta}$, see discussion following Theorem \ref{thm4} below). However, it was explicitly mentioned in \cite{SV2012} that their technique fails in controlling Lipschitz norms, and so our results complement theirs in that regard. Further, to our knowledge, our results are new for the case of fractional dissipation ($\alpha\in(0,1)$), even for the drift-diffusion problem \eqref{classicalddintro}, let alone \eqref{maineq3intro}. It is also worth noting that Zhang \cite{Zhang2006} also analyzed the Lipschitz regularity of solutions to \eqref{classicalddintro} and \eqref{maineq3intro} (with $\alpha=1$) when the drift velocity $b$ lies in a \emph{sub-critical} Kato class. Nevertheless, the result obtained in \cite{Zhang2006} is ``local'' boundedness of the Lipschitz norm, as opposed to the global result we obtain here. 

Before precisely formulating our results, let us make a few remarks which we find noteworthy. For starters, the requirement that $\beta+2\alpha-1\geq0$ is known to be sharp \cite{SVZ2013}, with the borderline case $\beta=1-2\alpha$ being delicate. Such a fact is also reiterated in a forthcoming work by Pasqualotto and Oh \cite{PO2021}. To our knowledge, it first appeared in the context of showing that solutions to \eqref{classicalddintro} gain a small degree of H\"older regularity by Constantin and Wu \cite{CW2008b} when the drift velocity $b\in L_t^{\infty}C_x^{0,1-2\alpha}$ and is divergence free. The divergence free structure of the drift was later relaxed by Silvestre \cite{Silvestre2012a}. Silvestre's result was later on improved by himself in a second paper \cite{Silvestre2012b}, where he showed that under a  further mild smallness assumption on $b\in L_t^{\infty}C_x^{0,1-2\alpha}$, one can prove that the solution is H\"older continuous with any exponent $\kappa\in(0,1)$ (as opposed to some $\kappa\in(0,1)$ in \cite{CW2008b, Silvestre2012a}). We emphasize that all previously mentioned results only prove H\"older continuity of the solution, but nothing about differentiability (when $\beta=1-2\alpha$ and $p=p^*=\infty$). When it comes to obtaining regularity criterion for the supercritical SQG (equation \eqref{classicalddintro} in two dimensions with $\alpha\in(0,1/2)$ and $b=R^{\perp}u$, where $R$ is the Riesz transform), Dong and Pavlovic were able to make do with $b\in C_tC_x^{0,1-2\alpha}$, but not $L_t^{\infty}C_x^{0,1-2\alpha}$. However, this result explicitly uses the structure of the SQG (in particular, the divergence-free nature of $b$), and it is not clear whether one can prove differentiability of solutions to \eqref{classicalddintro} with $\alpha\in(0,1/2)$ and an abstract drift $b\in L_t^{\infty}C_x^{0,1-2\alpha}$. This problem is still open (to our knowledge), and will remain so even after our work.

The situation becomes much better under the assumption that $b\in L_t^{\infty}C_x^{0,\beta}$ when $\beta+2\alpha-1>0$. For the case of the SQG, Constantin and Wu \cite{CW2008a} were able to show regularity under such an assumption using Besov space techniques. This was later on improved by Dong and Pavlovic to the critical H\"older spaces $L_t^{p^*}C_x^{0,\beta}$ \cite{DP2009a} and the associated critical Besov spaces in \cite{DP2009b}. Similarly, Silvestre \cite{Silvestre2012b} was able to show that for the abstract equation \eqref{classicalddintro} (without any divergence-free assumptions on $b$), the solution has a H\"older continuous derivative, i.e. $u\in L_t^{\infty}C_x^{1,\kappa}$, for some small $\kappa\in(0,1)$, and is therefore classical (when $b\in L_t^{\infty}C_x^{0,\beta}$). Our results are more in line with Silvestre's, namely we do not make any assumption on the divergence of the drift velocity, nor on its relation to the solution $u$. They allow for slightly weaker assumptions, namely $L_t^{\gamma}C_x^{0,\beta}$ (with the non-decreasing assumption on the bound) rather than $L_t^{\infty}C_x^{0,\beta}$, at the expense of proving integrability of the Lipschitz constant, rather than its uniform boundedness and H\"older continuity. Our results also improve on those in \cite{Silvestre2012a} when $\alpha\in(1/2,1]$, namely we prove integrability of the Lipschitz constant under a supercritical assumption on the drift velocity. To our knowledge, in the context of drift-diffusion type problems, so far ``slightly supercritical'' barriers were broken, where ``slightly'' usually means one is below the critical scaling level by a logarithmic factor. See for instance \cite{DKSV2014, Ignatova2014, Tao2009}, the preprint \cite{BC2021} and the references therein, though none of these results relate to ours. It is also worth noting that unlike \cite{CW2008b, Silvestre2012a, Silvestre2012b}, we do not study any smoothing effect of the transport-diffusion operator. We focus on the propagation of regularity instead, obtaining a-priori bounds under the assumption that we do have a classical solution.

The reader may be wondering why we are unable to hit the critical threshold $p^*$ when $\alpha\in(0,1/2)$ (recall that in this regime, $\gamma>p^*$). Let us remark that for the specific critical norm $\beta=1-2\alpha$ and $p^*=\infty$, Silvestre's work \cite{Silvestre2012a, Silvestre2012b} as well as Constantin and Wu \cite{CW2008b}, imply H\"older continuity, with a small H\"older exponent, but not differentiability. For the specific SQG case, Dong and Pavlovic \cite{DP2009a} were able to get regularity when $b\in C_tC_x^{0,1-2\alpha}$ or $L_t^{p^*}C_x^{0,\beta}$ if $\beta+2\alpha-1>0$. This tells us that the critical case $p=p^*=\infty$ when $\alpha\in(0,1/2)$ is quite delicate, and unless one specializes to specific drifts, one may not get differentiability. Thus, before trying to extend our results to the other critical cases when $\alpha\in(0,1/2)$, it might be a good idea to obtain (sharp) explicit blowup rates on certain H\"older norms for solutions to the fractional Burgers equation (which are known to lose regularity in finite time), and make sure that it is indeed possible to hit the critical threshold without making further assumptions on $b$ (such as divergence-free). The situation when $\alpha\in[1/2,1]$ is better. For instance, it was shown in \cite{SVZ2013} that when $\alpha=1$ and the spatial dimension is two, then there are distributional solutions to \eqref{classicalddintro} which obey a logarithmic modulus of continuity for positive time, provided the drift velocity is independent on time and locally integrable in space, which is a supercritical assumption.
 
 The advantage of using the approach presented herein is that, one, we are able to recover most of the previously mentioned results in a unified fashion that is arguably much simpler than what has been used, and in some cases, relaxing some of the conditions on the drift velocity. Moreover, when working in the periodic setting, we do not use the classical maximum principle that solutions of \eqref{classicalddintro} satisfy, as opposed to the work of Silvestre \cite{Silvestre2012a,Silvestre2012b}. It is only implicitly used when proving the same estimates in the whole space, and this potentially is a mere technicality that could be overcome. As a consequence, our approach is robust enough that it extends, without much difficulty, to the much more interesting ``linear'' Navier-Stokes system with fractional and classical dissipation \eqref{maineq3intro}, since it is not known whether one has a maximum principle in this case. Actually proving a maximum principle, or even $L^{\infty}$ estimate, for \eqref{maineq3intro} would immediately imply regularity for the nonlinear problem when $\alpha\in(1/2,1]$ (as in the case with Burgers). One of the goals of this work is to provide an alternative route to proving regularity, without relying on conserved quantities (to be specific, Theorems \ref{thm1} and \ref{thm2} below). Two, it allows us to reproduce part of the result in \cite{SV2012} when dissipation is classical ($\alpha=1$), as well as extend this criterion to the fractional dissipation case $\alpha\in(0,1)$. Three, it also reveals a supercritical regularity result for \eqref{classicalddintro} when $\alpha\in(1/2,1]$. 
 
 The importance of this last remark stems from a classical result obtained by  Foias, Guillop\'e and Temam \cite{FGT1981}. In that paper, they showed that Leray-Hopf weak solutions to the incompressible Navier-Stokes system \eqref{maineq} in the periodic three dimensional setting emanating from smooth initial data satisfy the following partial regularity: for any  $T\in(0,\infty)$, and any integer $m\geq1$, we have
 \begin{equation}\label{FGTintro}
 \int_0^T\|u(t,\cdot)\|^{\zeta_m}_{\dot{H}^m}dt<\infty,
 \end{equation}
 where $\|\cdot\|_{\dot{H}^m}$ is the Sobolev semi-norm and the exponent $\zeta_m$ is given by 
 \[
 \zeta_m:=\frac{2}{2m-1}.
 \]
Recall that in three-dimensions, the Sobolev embedding theorem tells us whenever $m>3/2$,
\[
\sup_{x\neq y}\frac{|f(x)-f(y)|}{|x-y|^{\beta}}\lesssim\|f\|_{H^m},\quad  \forall\beta\in(0,m-3/2].
\]
Choosing $m=2$, we get that 
\[
\int_0^T[u(t,\cdot)]_{C_x^{0,\beta}}^{2/3}dt<\infty, \quad \forall \beta\in(0,1/2].
\]
Recalling estimate \eqref{bduddintro} when $\alpha=1$ (and assuming for the moment we are working with the standard $L_t^pC_x^{0,\beta}$ semi-norm and not \eqref{altcondbintro}),
\[
\|\nabla u(t,\cdot)\|_{L^{\infty}}\lesssim [u(t,\cdot)]_{C_x^{0,\beta}}^{1/(1+\beta)},\quad \forall\beta\in(0,1),
\]
and so we are able to apply the Foias-Guillop\'e-Temam a-priori estimate to get regularity by choosing $\beta=1/2$. The case when $\beta=0$ is special, and one has 
\[
\int_0^T\|u(t,\cdot)\|_{L^{\infty}}dt<\infty,
\]
see also \cite{Constantin2001, Constantin2014, Tao2013}. Other supercritical a-priori bounds can be found in \cite{CV2014, Constantin1990} and the references tehrein. Of course, as mentioned previously, \emph{estimate \eqref{bduddintro} is only guaranteed without the pressure term}. We do not know whether this can be done for the NSE as well, the best we have in this case is the partial regularity result \eqref{partregintro}. The other issue is that we are working with the slightly stronger norm \eqref{altcondbintro}, so this is another technical difficulty that needs to be bypassed either by improving on the bound \eqref{FGTintro} or relaxing this condition in our result before hoping to apply this for the NSE. Those issues are currently being investigated by the author, and any interesting  results will be reported in a forthcoming manuscript. In particular, we make no claim that the estimates we obtain herein are sharp. Indeed, the idea of tracking the evolution of moduli of continuity is by no means a perturbative technique, and thus there is no ``systematic'' way of constructing those objects. The results reported in this work are a consequence of choosing moduli of continuity of the form
\begin{equation}\label{transtmodintro}
\Omega(t,\xi):=\lambda(t)\omega(\mu(t)\xi),
\end{equation}
where $\lambda, \mu$ and $\omega$ are chosen depending on the model at hand. Our particular choice of $\Omega$ in \eqref{transtmodintro} will be motivated in \S\ref{secpfmainres}. Such a construction is not unique, and a more careful construction could lead to sharper bounds. Apart from the above, another direction where our results could be improved is upgrading the $L_t^1$ control over the Lipschitz constant to $L_t^{\infty}$ whenever the H\"older semi-norm of $b$ is not uniformly bounded. The reader will soon realize that this is not possible if we consider moduli of continuity of the form \eqref{transtmodintro}, and that a different construction is needed. Our aim here is to lay down some of the fundamental ideas that hopefully could be used later on to get stronger results, or possibly generalized to related systems and models.

This paper is organized as follows. In \S\ref{secmainres} we formulate our results and motivate the current work, while in \S\ref{secprelim} we list some preliminary results that will be used. Section \ref{seccontest} concerns itself in obtaining some continuity estimates with regard to the nonlocal operators that we will be working with, before proving our main results in \S\ref{secpfmainres}.
\section{Main results}\label{secmainres}
\subsection{Formulation and precise statements}\label{secform}
Let us now precisely formulate our results. Throughout this work, whenever $X$ is a functional space, we abuse notation and say a vector field $u\in X$ to mean that every component $u_j\in X$. With regard to systems \eqref{maineq} and \eqref{maineq2}, we are mainly interested in smooth (classical, pointwise), periodic solutions arising from smooth, periodic vector-fields. For simplicity, we restrict ourselves to solutions with zero averages over the fundamental periodic domain. To be more precise, we define the space 
\begin{equation}\label{defcinftyper}
\dot{C}^{\infty}_{per}:=\left\{\theta\in C^{\infty}(\rd):\theta(x+L e_k)=\theta(x)\  \forall x\in\rd,\ k=1,\cdots,d,\ \int_{\tr}\theta(x)\ dx=0\right\},
\end{equation}
where $L>0$ is arbitrary, $\{e_k\}_{k=1}^d$ is the standard unit basis of $\rd$ and $\tr:=[0,L]^d$, and assume that the initial vector-field $u_o\in\dot{C}^{\infty}_{per}$ (of course, we also impose the condition $\nabla\cdot u_0=0$ in the case of NSE). This can be done without any loss in generality for the case of \eqref{maineq} due to its Galilean invariance, and we prefer to avoid dealing with unnecessary complications in the case of \eqref{maineq2}. Recall that in the case of the NSE the pressure term is recovered from the vector-field $u$ via solving
\[
-\Delta p=\text{div}\left[(u\cdot\nabla)u\right],
\]
 and so one can think of the PDE in \eqref{maineq} in terms of $u$ only. That being said, by a classical solution to \eqref{maineq} or \eqref{maineq2} we mean 
\begin{defs}\label{classsol}
A vector-field $u:[0,T)\times\rd\rightarrow\rd$ is said to be a classical solution to \eqref{maineq} or \eqref{maineq2} on a time interval $[0,T)$ if $u\in C^{\infty}([0,T)\times \rd)$, $u(t,\cdot)\in \dot{C}^{\infty}_{per}$ for every $t\in[0,T)$, satisfies \eqref{maineq} or \eqref{maineq2} in the pointwise sense for every $(t,x)\in[0,T)\times\rd$ and for which
\[
\lim_{t\rightarrow0^+}\left|\partial^{\alpha}_xu(t,x)-\partial^{\alpha}_xu_0(x)\right|=0,
\]
holds true for every $x\in\rd$ and every multi-index $\alpha\in \mathbb{N}^d$.
\end{defs}

\begin{defs}\label{defmod}
 We say a function $\omega:[0,\infty)\rightarrow[0,\infty)$ is a modulus of continuity if $\omega\in C([0,\infty))$ is non-decreasing and is piecewise $C^2$ on $(0,\infty)$ with finite one-sided derivatives such that $\omega'(\xi^+)\leq\omega'(\xi^-)$ for every $\xi\in(0,\infty)$. A modulus of continuity $\omega$ is said to be strong if in addition $\omega(0)=0$, $0<\omega'(0)<\infty$ and $\displaystyle{\lim_{\xi\rightarrow0^+}\omega''(\xi)=-\infty}$.
\end{defs}
\begin{defs}\label{timdepmoddef}
Let $T>0$ be given. A function $\Omega\in C([0,T]\times[0,\infty))$ is said to be a time-dependent strong modulus of continuity on $[0,T]$ if $\Omega(t,\cdot)$ is a strong modulus of continuity for each $t\in[0,T]$, $\Omega(\cdot,\xi)$ is piecewise $C^1$ on $[0,T]$ for each fixed $\xi\in[0,\infty)$ (with finite one-sided derivatives) and satisfies at least one of the following conditions:
\begin{enumerate}
\item $\Omega(\cdot,\xi)$ is nondecreasing as a function of time for each fixed $\xi$,
\item $\partial_\xi\Omega(\cdot,0)$ is continuous as a function of time. 
\end{enumerate}
\end{defs}
\begin{hypoth}\label{hypoth}
Let $d\geq3$, and suppose $C_d\geq1$ is a given absolute constant depending only on the dimension $d$. Then there exists a strong modulus of continuity $\Omega_0(\xi)$ such that for any given $T>0$, one can construct a time-dependent strong modulus of continuity $\Omega(t,\xi)$ on $[0,T]$ such that $\Omega(0,\xi)=\Omega_0(\xi)$ and for which 
\begin{equation}\label{modevolintro}
	\partial_t\Omega(t,\xi)-4\partial^2_\xi\Omega(t,\xi)-\Omega(t,\xi)\partial_\xi\Omega(t,\xi)-C_d\left[\int_0^{\xi}\frac{\Omega^2(t,\eta)}{\eta^2}\ d\eta+\Omega(t,\xi)\int_{\xi}^{\infty}\frac{\Omega(t,\eta)}{\eta^2}\ d\eta\right]\geq0,
\end{equation}
holds true for every $(t,\xi)\in(0,T]\times(0,\infty)$.
\end{hypoth}
\begin{rmk}
	Note that we slightly modify the definition of a modulus of continuity from \cite{KNS2008,KNV2007} by replacing the concavity assumption with the weaker condition $\omega'(\xi^+)\leq\omega'(\xi^-)$, as well as allow for bounded moduli of continuity. To get interesting results, the initial modulus of continuity $\Omega(0,\cdot)$ need to be independent on $T$, as will be explained in \S\ref{pfthm1} (ideally, one would like $T=\infty$). Moreover, strictly speaking, the terms $\partial_t\Omega$ and $\partial^2_\xi\Omega$ should be interpreted as left derivatives, since $\Omega$ is assumed to be piecewise $C^1$ in time and piecewise $C^2$ in space.
\end{rmk}
\begin{thm}\label{thm1}
	Let $d\geq3$ and suppose $u_0$ is a smooth divergence free vector-field such that $u_0\in \dot{C}^{\infty}_{per}$. Let $T_*>0$ be the maximal time of existence of the classical solution to \eqref{maineq} (according to Definition \ref{classsol}). Let $T>0$ be given, and assume that Hypothesis \ref{hypoth} is true for our choice of $d$ and $T$. It follows that $T_*>TB^{-2}$ and 
	\begin{equation}\label{mainbd1}
	\|\nabla u(t,\cdot)\|_{L^{\infty}}<B^2\partial_\xi\Omega(t,0),\quad \forall t\in\left[0,\frac{T}{B^2}\right],
	\end{equation}
where $B>0$ depends only on the $W^{1,\infty}$ norm of the initial data and the initial modulus of continuity $\Omega(0,\xi)$. For instance, one can take any 
\[
B\geq\frac{2\|u_0\|_{L^{\infty}}}{\Omega_0(\delta)}+\left(\frac{\delta}{\Omega_0(\delta)}\|\nabla u_0\|_{L^{\infty}}\right)^{1/2},
\]
where $\delta\in(0,1]$ is a small parameter depending on the structure of $\Omega_0(\xi)$; see Lemma \ref{buildmod}. In particular, if Hypothesis \ref{hypoth} is true with $T=\infty$, then we get that the solution $u$ is globally regular and satisfies
\begin{equation}\label{uniformlipbd}
\|\nabla u(t,\cdot)\|_{L^{\infty}}< B^2\partial_\xi\Omega(t,0), \quad \forall t\geq0.
\end{equation}
\end{thm}
\begin{rmk}
Although we normalized the viscosity coefficient in \eqref{maineq} to 1, the length of the period and size of initial data were both kept arbitrary large, so the same result is true for \eqref{maineq} with arbitrary viscosity coefficient $\nu>0$. Indeed, one can use the scale invariance of \eqref{maineq} to normalize the viscosity at the expense of increasing (or decreasing) the size of initial data and/or the length of the period. This will be reflected in the constant $B$ in the above Theorem (it depends on the size of initial data, and by extension, viscosity after normalizing it). One could prove a similar result in the whole space setting, see the discussion towards the end of \S\ref{pfthm1} below.
\end{rmk}
The above Theorem and Hypothesis seem to be quite technical. Perhaps the easiest way to heuristically describe this result is as follows: the lifespan of solutions to \eqref{modevolintro} give a lower bound for the lifespan of solutions to the NSE. In particular, if one can construct a solution to \eqref{modevolintro} for all time which happens to be a modulus of continuity, then one can guarantee that regular solutions to the incompressible NSE never breakdown. Alternatively, one can use \eqref{modevolintro} to get a lower bound on the time of existence of solutions to the NSE in terms of the dimension $d$ and the size of the initial data $\|u_0\|_{W^{1,\infty}}$, see \S\ref{pfthm1} below.

The reader may be tempted to argue that the quadratic terms appearing in inequality \eqref{modevolintro} will kill any hope of constructing a global solution to this inequality. We prefer to be more conservative in this regard, for a number of reasons. One, the quadratic term is ``averaged'', in particular, an integration by parts would transform the quadratic term into a ``nonlocal'' transport:

\[
\int_0^{\xi}\frac{\Omega^2(t,\eta)}{\eta^2}\ d\eta+\Omega(t,\xi)\int_{\xi}^{\infty}\frac{\Omega(t,\eta)}{\eta^2}\ d\eta=2\int_0^{\xi}\frac{\Omega(t,\eta)\partial_\eta\Omega(t,\eta)}{\eta}d\eta+\Omega(t,\xi)\int_{\xi}^{\infty}\frac{\partial_\eta\Omega(t,\eta)}{\eta}\ d\eta.
\]
We know classical dissipation always prevails over the standard advective nonlinearity, but it is not clear to us whether the same can be said about the above nonlocal advective nonlinearity. Two, in theory, to get regularity, from bound \eqref{mainbd1} one only needs to make sure that 
\[
\int_0^T\partial_\xi\Omega(t,0)dt<\infty,
\]
for any $T\in(0,\infty)$, so that some degree of singularity formation is allowable. Of course, this will introduce various other technical difficulties, namely, how can we make sense of a ``singular'' modulus of continuity, which we do not address here. Finally, recall the condition 
\[
\lim_{\xi\rightarrow0^+}\partial^2_\xi\Omega(t,\xi)=-\infty,
\]
so that near the  boundary $\xi=0$, the viscous term is very strong. And it is exactly the behavior of $\Omega$ near $\xi=0$ that we care about the most. Unfortunately, so far we are not able to turn such remarks into a rigorous global regularity result. Nevertheless, we are able to provide a short-time existence result for large data, see \S\ref{pfthm1} below. As for the constant $C_d$, it stems from the fundamental solution to Laplace's equation in $\rd$, there does not seem to be any particular advantage when applying this argument when $d=2$ as opposed to $d\geq3$. 

Before trying to tackle the nonlinear integrals with a high degree of singularity near the origin, it is best to see what can be achieved if we consider a simpler model. Namely, instead of dealing with a nonlinear, order one operator (the term $\nabla p$), let us see what happens if we have a linear, zero order term, for instance a Riesz transform. We let $K$ be a kernel of the form
\begin{equation}\label{defkern}
K(z):=
\begin{cases}
	\displaystyle{\frac{\Phi\left(\large{z/|z|}\right)}{|z|^d}}, &z\neq0,\\
	0, &z=0,
\end{cases}
\end{equation}
where $\Phi:\mathbb{S}^{d-1}\rightarrow\mathbb{R}$ is a H\"older continuous function that satisfies the following zero average condition
\begin{equation}\label{zeroavg}
\int_{\mathbb{S}^{d-1}}\Phi(y)\ dy=0,
\end{equation}
and we consider the (nonlocal) singular integral operator $\mathcal{N}$ defined by
\begin{equation}\label{ptwiseN}
\mathcal{N}\theta(x):=P.V.\int_{\rd}K(x-z)\theta(z)\ dz,
\end{equation}
where $\theta\in L^{p}(\rd)$ with $p\in(1,\infty)$. Alternatively, one can define $\mathcal{N}$ as a Fourier multiplier with symbol
\begin{equation}\label{Khat}
\widehat K(\zeta):=\int_{\mathbb{S}^{d-1}}\left[\frac{\pi i}{2}\text{sgn}\left(\frac{\zeta}{|\zeta|}\cdot y\right)-\log\left(\frac{\zeta}{|\zeta|}\cdot y\right) \right]\Phi(y)\ dy,\quad \zeta\in\mathbb{R}^d\backslash \{0\},
\end{equation}
see Theorem 3 in \cite[Chapter~2]{Stein1970book} for more details and rigorous justification of the above. 

Since we will be working with periodic functions, we would need to make sense of periodic analogues to such operators. This can be done by recalling Theorem 3.8 (and Corollary 3.16) of \cite[Chapter~7]{SWbook1971}, which guarantees that one can ``periodize'' the operator $\mathcal{N}$ in a unique fashion via utilizing the symbol \eqref{Khat} in an obvious way as a Fourier multiplier (now over the Torus). However, there is the issue of defining $\widehat K(0)$, but this will not be of concern to us as we will only be working with periodic functions having zero averages, and so if $\theta$ is $L-$ periodic with zero average, the periodization of $\mathcal{N}$ is understood as 
\begin{equation}\label{defNp}
\mathcal{N}_p\theta(x):=C_{d,L}\sum_{\substack{m\in\mathbb{Z}^d \\ m\neq0}}\widehat K(m)\widehat \theta(m)e^{2\pi im\cdot x/L},
\end{equation}
where $C_{d,L}$ is a normalizing constant depending on the dimension $d$ and (possibly) the length of the period. The issue of periodizing such operators is addressed in more details in \cite{CZ1954}, where it was shown that one also has the pointwise definition
\begin{equation}\label{ptwiseNp}
	\mathcal{N}_p\theta(x):=C_{d,L}\sum_{m\in\mathbb{Z}^d}\int_{\tr}\left[K(z+m)-K(m)	\right]\theta(x-z)\ dz.
\end{equation}
Note that the sum \eqref{ptwiseNp} converges absolutely and uniformly on compact sets owing to the fact that $\Phi$ is assumed to be in the class $C^{0,\rho}(\mathbb{S}^{d-1})$, some $\rho\in(0,1]$ (see the proof of Lemma \ref{sioholder}, below, where one can use a similar argument to prove convergence of the series). Moreover, it was shown in \cite{CZ1954} that the operator defined by \eqref{ptwiseNp} can be realized as a Fourier multiplier whose symbol agrees with the restriction of the Fourier transform of the original kernel $K$ on $\mathbb{Z}^d$. It follows from the uniqueness result of Theorem 3.8 in \cite[Chapter~7]{SWbook1971} that $\mathcal{N}_p$ is well defined, and so we may drop the subscript ``$p$'' from \eqref{defNp} without ambiguity.
\begin{thm}\label{thm2}
	Let $d\geq2$, suppose $u_0$ is a smooth vector-field such that $u_0\in \dot{C}^{\infty}_{per}$ and let $\mathcal{N}$ be as described above. Then there exists a unique classical solution to the IVP \eqref{maineq2} on $[0,\infty)$. Moreover, we have the bound
	\begin{equation}\label{mainbd2}
	\|\nabla u(t,\cdot)\|_{L^{\infty}}\leq 2\exp\left[2\log(\lambda_0)\exp\left(C_0t\right)\right], \quad \forall t\geq0,
	\end{equation}
	where 
	$C_0\geq1$ depends on the dimension $d$, the period $L$ and the kernel $K$ while $\lambda_0$ depends only on $\|u_0\|_{W^{1,\infty}}$ but not on the dimension $d$, the period or kernel. If $\mathcal{N}\equiv 0$ (that is, we analyze the viscous Burgers equation), then we get 
	\begin{equation}\label{bdviscburgers}
	\|\nabla u(t,\cdot)\|_{L^{\infty}}\leq \lambda_0^2, \quad\forall t\geq0,
	\end{equation}
	where $\lambda_0$ depends only on $\|u_0\|_{W^{1,\infty}}$, but not on the period $L$ nor the dimension $d$.
\end{thm}
\begin{rmk}
We point out that periodicity is used in an essential way in the proof of Theorem \ref{thm2} and there seems to be a significant technical obstacle in the way of obtaining a similar result in the whole space scenario, as will be demonstrated in \S\ref{pfthm2}. Such a difficulty could potentially be bypassed if we either have an $L^{\infty}$ estimate on the solution, or explicit rates of decay. In particular, bound \eqref{bdviscburgers} is true in the whole space. Our purpose here is merely a proof of concept (obtain  regularity without relying on conserved quantities as energy), with our aim being to keep the presentation as simple as possible, so we do not pursue this direction.
\end{rmk}
With regard to \eqref{classicalddintro}  and \eqref{maineq3intro}, we think of them as being posed either in the periodic setting (where the periods of $b$ and $u$ are the same) or in the whole space with sufficient decay at infinity. The first result we prove is with regard to the drift-diffusion problem \eqref{classicalddintro}.
\begin{thm}\label{thm3}
	Given $\alpha\in(0,1]$, let $\beta\in[0,1)$ be chosen such that $\beta+2\alpha-1>0$ and define
\begin{equation}\label{defgamma}
\gamma:=\frac{1}{\beta+2\alpha-1}.
\end{equation}
Assume that $b:[0,T]\times\rd\rightarrow\rd$ is a given vector field (not necessarily divergence-free) that satisfies
\begin{equation}\label{condb}
\left[b(t,\cdot)\right]_{C^{0,\beta}}:=\sup_{x\neq y}\frac{|b(t,x)-b(t,y)|}{|x-y|^{\beta}}\leq g(t),\quad \forall\  t\in[0,T],
\end{equation} 
	where $g\geq 1$ is a non-decreasing function defined on $[0,T]$. If $\beta=0$, we assume that $\|b(t,\cdot)\|_{L^{\infty}}\leq  g(t)$. Let $u$ be a solution to
\begin{equation}\label{classicaldd}
\begin{cases}
	\partial_t u(t,x)+(-\Delta)^{\alpha} u(t,x)=(b\cdot\nabla)u(t,x), \\
	u(0,x)=u_0(x),
\end{cases}
\end{equation}
with regularity 
\begin{equation}\label{regu}
	u\in C_t^1C_x^2((0,T]\times\rd)\cap C([0,T];W^{1,\infty}(\rd)),
\end{equation}
Assume further that one of the following hypotheses is true:
	\begin{enumerate}[label=\Alph*)]
	\item $u(t,\cdot)$ is periodic in space with arbitrary period $L>0$ or
	\item $u(t,\cdot)\in W^{2,\infty}(\rd)$ for every $t\in[0,T]$ and the gradient vanishes at infinity, i.e.
	\[
	\lim_{|x|\rightarrow\infty}|\nabla u(t,x)|=0,\quad \forall t\in[0,T].
	\] 
	\end{enumerate}
It follows that there exists a constant $C_{0,\alpha}\geq1$ depending only on $\|u_0\|_{W^{1,\infty}}$ and $\alpha$ such that the following estimate is true 
\begin{equation}\label{supcritdd}
\|\nabla u(t,\cdot)\|_{L^{\infty}}\leq C_{0,\alpha} g^{\gamma}(t),\quad \forall t\in[0,T].
\end{equation}
\end{thm}
Here, we are assuming the existence of a classical solution, and obtaining bounds based on that. In other words, we are obtaining a-priori estimates, and unlike \cite{CW2008b, Silvestre2012a, Silvestre2012b}, we do not study any smoothing effect of the operator $\partial_t+(-\Delta)^{\alpha}-b\cdot\nabla$. Such a-priori estimates could be used to show regularity of weak solutions to \eqref{classicalddintro}, or obtain a regularity criterion for a wide class of (non-linear) active scalars, such as the dissipative SQG and related models. We leave it up to the reader to decide how they would like to make use of the result. To say the least, it reproduces the previously mentioned regularity criteria for the supercritical SQG equation; we refer the reader back to \S\ref{secimplication} for a discussion on how our results compare to the previous ones. 

The condition $\beta+2\alpha-1\geq0$ is sharp, even if we impose a divergence-free constraint on the drift. It was shown in \cite{SVZ2013} that when the dimension $d=2$, $\alpha\in(0,1/2)$ and $\beta\in(0,1-2\alpha)$, one can construct a time-independent drift $b\in C_x^{0,\beta}$ that is divergence free, and a solution that violates any modulus of continuity in finite time. For $\alpha\in(1/2,1]$, the condition $\beta+2\alpha-1\geq0$ is satisfied even when $\beta=0$. Further, it was also shown in \cite{SVZ2013} that in two dimensions, if the drift velocity is divergence-free and locally integrable in space, then one can construct a solution to \eqref{classicaldd} that satisfies a logarithmic modulus of continuity, a supercritical assumption. Here, we show that a supercritical H\"older assumption on the drift velocity gives us estimates on the Lipschitz constant, and in particular, control of a critical quantity of the solution. For instance when $\alpha=1$, we need only assume that $g$ is non-decreasing and
\[
\int_0^Tg^{1/(1+\beta)}(s)ds<\infty.
\]

We find such an estimate interesting due to the following. As a consequence of the Foias-Guillop\'e-Temam partial regularity result \cite{FGT1981} (see also \cite{Constantin2001, Constantin2014, Tao2013} for when $\beta=0$) in the three-dimensional periodic NSE \eqref{maineq3intro}, one has 
\begin{align*}
&\int_0^T[u(t,\cdot)]_{C_x^{0,\beta}}^{2/3}dt<\infty,\quad \forall\beta\in(0,1/2],\\
&\int_0^T\|u(t,\cdot)\|_{L^{\infty}}dt<\infty.
\end{align*}
Thus, choosing any $\beta=1/2$ tells us that 
\[
\int_0^T[u(t,\cdot)]_{C_x^{0,\beta}}^{1/(1+\beta)}dt=\int_0^T[u(t,\cdot)]_{C_x^{0,1/2}}^{2/3}dt<\infty,
\]
and so proving estimate \eqref{supcritdd} for the NSE \eqref{maineq3intro} would be a significant step towards resolving the global regularity problem (one still has to take care of the non-decreasing assumption we have on $g$). We emphasize again that we do not use the classical maximum principle that is readily available for \eqref{classicalddintro} to obtain estimate \eqref{supcritdd} in the periodic setting, only in the whole space do we rely on it. That could be relaxed if one has explicit rates on the decay of the solution at spatial infinity, though we do not pursue that direction here.

Unfortunately, as the reader will soon realize in \S\ref{pfthm4a} below, it is just not possible to obtain an estimate analogous to \eqref{supcritdd} for solutions to \eqref{maineq3intro}, even with $\alpha=1$, when working with moduli of continuity of the form $\Omega(t,\xi):=\lambda(t)\omega(\mu(t)\xi)$. Whether one can achieve this bound by constructing moduli of continuity of different form is currently under investigation by the author, and any interesting results will be reported in a separate manuscript. We refer the reader to \S\ref{secmotivation}, below for a heuristic explanation of the difficulty encountered when dealing with the nonlocal pressure term, as well as why one sees double-exponential growth in Theorem \ref{thm2} above. 

Before stating our next Theorem, let us reiterate the main drawback of the non-decreasing assumption on the bounding function $g$. It would lead to 
\[
\sup_{s\in[0,t]}[b(s,\cdot)]_{C_x^{0,\beta}}\leq g(t),
\]
which is a slightly stronger assumption than simply requiring $b\in L_t^pC_x^{0,\beta}([0,T]\times\rd)$. However, \emph{we emphasize that we do not assume $g\in L^{\infty}$}. In particular, it could be of the form $(T-t)^{-\kappa}$, for appropriately chosen $\kappa>0$. Thus, in our analysis, it seems more natural to replace the standard $L_t^pC_x^{0,\beta}$ (semi)-norm with 
\begin{equation}\label{altcondbform}
\left(\int_0^T\sup_{s\in[0,t]}[b(s,\cdot)]_{C_x^{0,\beta}}^pdt\right)^{1/p}<\infty,
\end{equation}
which scales exactly like the standard $L_t^pC_x^{0,\beta}$ (semi)-norm.
\begin{thm}\label{thm4}
	Given $\alpha\in(0,1]$, let $\beta\in(0,1)$ be chosen such that $\beta+2\alpha-1>0$ and define
\begin{equation}\label{defgammathm4}
\gamma:=\frac{1}{\beta+2\alpha-1}.
\end{equation}
Assume that $b:[0,T]\times\rd\rightarrow\rd$ is a given, divergence-free vector field that satisfies
\begin{equation}\label{condb}
\left[b(t,\cdot)\right]_{C^{0,\beta}}:=\sup_{x\neq y}\frac{|b(t,x)-b(t,y)|}{|x-y|^{\beta}}\leq g(t),\quad a.e.\  t\in[0,T],
\end{equation} 
	where $g\geq 1$ is a non-decreasing function defined on $[0,T]$. Let $(u,p)$ be a solution to
\begin{equation}\label{maineq3thm}
\begin{cases}
	\partial_t u(t,x)+(-\Delta)^{\alpha} u(t,x)=(b\cdot\nabla)u(t,x)+\nabla p(t,x), \\
	\nabla\cdot u(t,x)=0, \\
	u(0,x)=u_0(x),
\end{cases}
\end{equation}
with regularity 
\begin{equation}\label{regu}
	u\in C_t^1C_x^2((0,T]\times\rd)\cap C([0,T];W^{1,\infty}(\rd)),
\end{equation}
Assume further that one of the following hypotheses is true:
	\begin{enumerate}[label=\Alph*)]
	\item $b(t,\cdot)$ and $u(t,\cdot)$ are both periodic in space with arbitrary period $L>0$ or
	\item $b(t,\cdot)\in L^{r}(\rd)$ some $r\in(1,\infty)$, the velocity vector field $u(t,\cdot)\in W^{2,\infty}(\rd)$ for every $t\in[0,T]$, and its gradient vanishes at infinity, i.e.
	\[
	\lim_{|x|\rightarrow\infty}|\nabla u(t,x)|=0,\quad \forall t\in[0,T].
	\] 
	\end{enumerate}
It follows that there exists two constants $B\geq1$ and $C_{d,\beta,\alpha}\geq1$ where $B$ depends only on the $\|u_0\|_{W^{1,\infty}}$ norm, while $C_{d,\alpha,\beta}$ depends only on the dimension $d$, $\alpha$ and $\beta$, with $C_{d,\alpha,\beta}\rightarrow\infty$ as $\beta\rightarrow 0$ or $\alpha\rightarrow 0$, such that the following estimates are true 
\begin{align}
&\|\nabla u(t,\cdot)\|_{L^{\infty}}\leq B g^{\gamma}(t)\exp\left(C_{d,\alpha,\beta}B^{1-\beta}\int_0^tg^{2\alpha\gamma}(s)ds\right),\quad \forall t\in[0,T],\label{crtibdnsethm}\\
&\|\nabla u(t,\cdot)\|_{L^{\infty}}\leq \lambda(t)\log(\lambda(t)),\quad \forall t\in[0,T]\label{supercritbdnsethm},
\end{align}
where 
\[
\log(\lambda(t)):= 
\begin{cases}
g^{\gamma}(t)\exp\left(C_{d,\alpha,\beta}B^{1-\beta}\int_0^tg^{1-\gamma\beta}(s)ds\right), &\alpha\in[1/2,1],\\
g^{\gamma(1-\beta)}(t)\exp\left(C_{d,\alpha,\beta}B^{1-\beta}\int_0^tg(s)ds\right),&\alpha\in(0,1/2),
\end{cases}
\]
\end{thm}
\begin{rmk}
The difference between \eqref{crtibdnsethm} and the corresponding bound \eqref{supcritdd} is the presence of the exponential term. Notice that $p^*=2\alpha\gamma$, and so in order to guarantee finiteness of the exponential term in \eqref{crtibdnsethm}, one always has to make a critical assumption on drift velocity. Thus, when $\alpha\in(1/2,1]$, the requirement on the drift is always $g\in L^{p^*}(0,T)$ (critical assumption), while for $\alpha\in(0,1/2)$, one has to require $g\in L^{\gamma}(0,T)$ (sub-critical assumption). As for the reason why we need $\beta>0$ even when $\alpha\in(1/2,1]$, it is required because unlike \eqref{classicaldd}, here we need to locally estimate Riesz transforms, and it is known that $L^{\infty}$ is a bad space for those operators \cite{Stein1970book}.
\end{rmk}
The regularity criterion for the classical NSE (system \eqref{maineq3thm} with $\alpha=1$) that we get as a consequence applying the above Theorem with $b=u$ together with the weak Beale-Kato-Majda criterion is not new. It was obtained by Silvestre and Vicol in \cite{SV2012}, where they show that $u\in L_t^{\infty}C_x^{0,\kappa}$ for any $\kappa\in(0,1)$ privided $u_0\in C_x^{0,\kappa}$ and $b$ is at the critical level. However, estimates \eqref{crtibdnsethm} and \eqref{supercritbdnsethm} \emph{are new}, even in the classical $\alpha=1$ case. It was explicitly mentioned in \cite[Remark~3.3]{SV2012} that their approach only measures the propagation of H\"older regularity, and that it fails when trying to measure the propagation of Lipschitz regularity. Moreover, the regularity criterion we get as a consequence of \eqref{crtibdnsethm} when $\alpha\in(0,1)$, is to our knowledge, the first of its kind. 

Let us conclude this section by briefly comparing our results to the ones obtained in \cite{SV2012} in the classical diffusion case. In contrast to \cite{SV2012} where the authors measured H\"older continuity via local integral (average) characterization and utilized $L^1$ based Morrey-Campanato spaces, we perform our analysis in a purely local pointwise medium. This puts our approach at a slight disadvantage to theirs in the sense that they were able to cover a wider range of critical norms. In particular they allow for $\beta\in(-1,1]$, where for $\beta\in(-1,0]$, condition \eqref{condb} is replaced with the requirement
\begin{equation}\label{condgsv}
\sup_{x\in\rd}\sup_{0<r<1}r^{-\beta}\int_{|y|\leq1}|b(t,x+ry)-\bar b(t,x,r)|dy\leq g(t),\quad g\in L^{2/(1+\beta)}([0,T]),
\end{equation}
with $\bar b(t,x,r)=0$ if $\beta\in(-1,0)$, and $\bar b(t,x,r)$ is the average of $b(t,\cdot)$ in a ball centered at $x$ with radius $r$ when $\beta=0$. On the other end of the spectrum, the reason why we omit the case $\beta=1$ is, very loosely speaking, a result of the fact that singular integral operators in general prefer H\"older over Lipschitz functions, see Lemma \ref{pressest} for more details. Nevertheless, Theorem \ref{thm3} complements the results of Silvestre-Vicol, upgrading the estimates from H\"older to Lipschitz. To be specific, they showed that if $b$ satisfies \eqref{condgsv} and if the initial data $u_0\in C_x^{0,\kappa}$, any $\kappa\in(0,1)$, then  the solution lies in $L_t^{\infty}C_x^{0,\kappa}([0,T]\times\rd)$. Our results fill this gap and says something when $\kappa=1$, at least in the range $\beta\in(0,1)$. Moreover, with regards to controlling the Lipschitz constant of a solution to the incompressible NSE when $\beta\in(-1,0]$, we can combine Theorem \ref{thm4} along with their results to obtain an estimate in terms of only the $W^{1,\infty}(\rd)$ norm of the initial data and the quantity \eqref{condgsv}.

\subsection{Motivation and heuristics}\label{secmotivation}
Let us start by recalling that sufficiently regular solutions to the NSE satisfy the following a-priori estimate
\begin{equation}\label{energy}
\|u(t,\cdot)\|_{L^2}^2+2\int_0^t\|\nabla u(s,\cdot)\|^2_{L^2}\ ds\leq \|u(0,\cdot)\|_{L^2}^2,\quad \forall t\geq0,
\end{equation}
and that such a control is sufficient to prevent blowup in two dimensions, but not when $d\geq3$. On the other hand, although solutions to the viscous Burgers equation, 
\begin{equation}\label{Burgers}
\partial_tu-\Delta u=(u\cdot\nabla)u,
\end{equation}
do not necessarily satisfy \eqref{energy}, we know that they do not develop singularities in finite time. The latter fact is due to the maximum principle $\|u(t,\cdot)\|_{L^{\infty}}\leq \|u(0,\cdot)\|_{L^{\infty}}$ \cite{LSUbook1968}. Similarly, such a maximum principle, if available for the NSE, would allow us to prove a global regularity result.

Unfortunately, due (in part) to the nonlocal dependence of the pressure term on the solution vector-field, such a-priori bound is not readily available. To understand the nonlocal structure of the equation, let us recall that the pressure is recovered from the velocity vector-field via
\begin{equation}\label{defpintro}
p:=\sum_{i,j=1}^d\partial_i\partial_j(-\Delta)^{-1}(u_iu_j)=\sum_{i,j=1}^dR_iR_j(u_iu_j)
\end{equation}
with $\{R_j\}_{j=1}^d$ denoting the standard Riesz transforms, the singular integral operators corresponding to the kernel \eqref{defkern} with $\Phi_j(z)=z_j/|z|$. Due to the lack of any ``obvious" useful pointwise upper bounds on $\partial_k p$ (in terms of $u$), it is not clear whether one can obtain a maximum principle in the same spirit as that of \eqref{Burgers}.

That being said, let us now recall the critically dissipative surface quasi-geostrophic (SQG) equation. In $d=2$, this equation reads
\begin{equation}\label{SQG}
\begin{cases}
	&\partial_t\theta+(-\Delta)^{1/2}\theta=(u\cdot\nabla)\theta,\\
	&u=(u_1,u_2)=(-R_2\theta,R_1\theta),\\
\end{cases}
\end{equation}
where $\theta$ is a scalar, and $R_1$, $R_2$ are the two dimensional Riesz transforms. This equation was first introduced in \cite{CMT1994}, and it was advocated as a toy model for the NS system, see also \cite{KNV2007} and the references therein. Equation \eqref{SQG} conserves all $L^p$ norms of the initial data, $p\in[1,\infty]$, however such control is not strong enough (in general) to deduce global regularity. The fact that evolution under equation \eqref{SQG} does not develop any singularities in finite time was proven by Kiselev, Nazarov and Volberg in \cite{KNV2007} (see also \cite{KNS2008} where the same tools were applied to the fractal Burgers equation simultaneously). A different proof was obtained by Caffarelli and Vasseur \cite{CaffarelliVasseur2010} almost at the same time. Other proofs surfaced in \cite{PCVC2012, KN2009, MM2013}. In this work, we build upon the ideas introduced in \cite{KNS2008, KNV2007}, which we now outline.

The elegant ideas introduced in \cite{KNS2008, KNV2007} are based on local continuity estimates, and allow for control of the Lipschitz constant of the solution \emph{by studying the evolution of the solution itself}, that is proving a ``maximum principle'' for $\nabla\theta$ by studying the evolution of $\theta$. This was achieved by constructing a family of strong moduli of continuity that must be preserved by the flow, and for which any arbitrary smooth enough initial data has a modulus of continuity belonging to such a family. Such moduli of continuity have a finite derivative at 0, thus providing an upper bound for the Lipschitz constant of the solution (see Lemma \eqref{gradbd}, below), meaning that their preservation would prevent a gradient blowup scenario, from which higher regularity follows by bootstrap. The central idea is to compare the power of dissipation, nonlinearity and nonlocality at the local level, carefully constructing the modulus of continuity in order to make sure dissipation will always prevail. 

This was later expanded upon in \cite{AH2008,DKSV2014,DD2008,Kiselev2011, MX2012} and the references therein, to name but a few. In particular, time dependent moduli of continuity were introduced in \cite{Kiselev2011} mainly in order to obtain certain eventual regularity results for a class of active scalar evolution equations. We recently extended this technique in \cite{Ibdah2020b} to a modified Michelson-Sivashinsky equation (see discussion below), where we were able to obtain a global regularity result. To our knowledge, this program has never been tried for the NS system. However, it was noted in \cite{KNV2007} that their original argument does not apply to the incompressible NS system ``due to the different structure of nonlinearity'', and we believe this work provides a partial remedy to that situation; we were able to at least initiate the study. 

  Tracking the evolution of moduli of continuity is in particular very useful for analyzing certain nonlinear parabolic equations in the presence of nonlocal terms, as it provides a ``medium'' capable of comparing stabilizing (dissipative) terms against destabilizing (certain nonlinear and/or nonlocal) terms at the same level. To more concretely demonstrate the last point, we consider the following problem analyzed in \cite{Ibdah2020b}, which essentially is what motivated this paper:
\begin{equation}\label{modMSE}
\partial_t\theta-\nu\Delta\theta+\frac{1}{2}|\nabla\theta|^2-(-\Delta)^{\alpha}\theta=0,\quad \alpha\in(0,1/2),
\end{equation}
If we drop the nonlocal term, one can prove a classical maximum principle \cite{LSUbook1968}, while on the other hand if we drop the nonlinearity, one can easily handle the instabilities arising from the nonlocal part by direct energy estimates, which may cause at most growth in time of higher order norms, but not blowup. The presence of both at the same time complicates things. Nevertheless, the singular integral operator $(-\Delta)^{\alpha}$ is known to map $C^{0,\beta}$ H\"older continuous functions to $C^{0,\beta-2\alpha}$, provided $0<2\alpha<\beta\leq1$ \cite{Silvestre2007}. Similarly, one can show that while it doesn't quite preserve abstract moduli of continuity, it doesn't distort them too much either \cite{Ibdah2020b}. Therefore one can successfully employ the above mentioned strategy and construct a certain modulus of continuity that must be obeyed by the solution for all time. 

Equation \eqref{modMSE} is a slight modification to the so called Michelson-Sivashinsky (MS) model, where the latter corresponds to the case when $\alpha=1/2$. The MS model is a refined combustion model based on the Darrieus–Landau flame stability analysis \cite{Sivashinsky1977}. The reason for restricting $\alpha\in(0,1/2)$ in \cite{Ibdah2020b} is because of the lack of any ``obvious'' continuity (or even $L^{\infty}$) estimate for the square root of the Laplacian of a function in terms of those known of function itself. To see this, recall the representation
\begin{equation}\label{sqrootlap}
(-\Delta)^{1/2}\theta=\sum_{j=1}^dR_j\partial_j\theta,
\end{equation}
and so even when $\theta$ has a strong modulus of continuity, all what we know is that $\nabla\theta\in L^{\infty}$, which is known to be a bad space for the Riesz transforms $R_j$ \cite{Stein1970book}. 

The relationship between equation \eqref{modMSE} (when $\alpha=1/2$) and the NS system \eqref{maineq} is based on the representations \eqref{defpintro} and \eqref{sqrootlap} for the nonlocal part of the corresponding evolution equation. From \eqref{defpintro}, we see that the ``local regularity'' of $\partial_kp$ is linked to that of the solution $u$ in a fashion similar to the link between the local regularity of $\theta$ and $(-\Delta)^{1/2}\theta$. One way of looking at it is by noting that both nonlocal operators, $(-\Delta)^{1/2}$  and $\nabla R_iR_j$, are of order one. In general, one should not expect to obtain continuity estimates on $\mathcal{N}\theta$ from those known on $\theta$ if $\mathcal{N}$ is an operator of order one or higher. Of course, the situation in NSE is more complicated due to the fact that the pressure depends not only nonlocally on the evolving entity (the velocity-field), but also nonlinearly.

 With the previous remarks in mind, let us focus purely on the incompressible Navier-Stokes system. A key observation that we will use in the proof of Theorems \ref{thm1} and \ref{thm4} is one that was made by Silvestre in an unpublished work \cite{Silvestre2010unpub}. By performing more careful local analysis of the pressure term and using incompressibility in a vital way, it was argued that if $u\in C^{0,\beta}$ then $p\in C^{0,2\beta}$. When $\beta\in(1/2,1)$, this translates to a H\"older condition on $\nabla p$. It is also important to further highlight the fact that those H\"older estimates are homogenous: \emph{they do not depend on the $L^{\infty}$ norm of $u$}. In particular, the generalized continuity estimate we get for $\nabla p$ in \eqref{pressestintro} below is purely local. This is exactly the kind of estimate required to track the evolution of moduli of continuity. Similar results were also obtained by Constantin \cite{Constantin2014}, Isett \cite{isett2013regularity} (also Isett and Oh \cite{IO2016}) and were observed by De Lellis and  Sz\'ekelyhidi Jr. \cite{DS2014} in the context of convex integration. We generalize Silvestre's argument to an abstract modulus of continuity in Lemma \ref{pressest}, below, and show that if $\omega$ is a modulus of continuity according to Definition \ref{defmod} and if
\[
|u(x)-u(z)|\leq\omega(|x-z|),\quad \forall (x,z)\in\rd\times\rd,
\]
then
\begin{equation}\label{pressestintro}
|\nabla p(x)-\nabla p(z)|\leq C_d\left[\int_{0}^{|x-z|}\frac{\omega^2(\eta)}{\eta^2}\ d\eta+\omega(|x-z|)\int_{|x-z|}^{\infty}\frac{\omega(\eta)}{\eta^2}\ d\eta\right],
\end{equation}
where $C_d$ is a positive, absolute universal constant depending only on the dimension $d$ and not on any norm of $u$ or $p$. The previous estimate is exactly the origin of the integrals appearing in Hypothesis \ref{hypoth}, while the transport part translates to the Burgers nonlinearity appearing in inequality \eqref{modevolintro}. Such an estimate allows us to at least initiate the study of propagation of moduli of continuity by the NSE, as opposed to \eqref{modMSE} with $\alpha=1/2$.

Notice that the singularity appearing in \eqref{pressestintro} corresponds to that of an operator of order one, while nonlinearity translates to the appearance of the quadratic term $\omega^2$. It is unclear to us what effect those integrals have on the Dirichlet to Neumann map of the solution to the one dimensional Burgers boundary value problem described in Hypothesis \ref{hypoth}. To be more specific, we need to ensure that $\partial_\xi\Omega(t,0)<\infty$, or at least $\partial_\xi\Omega(\cdot,0)\in L^1(0,T)$, for any $T>0$. Thus, in an attempt to better understand this, we considered the two model problems \eqref{maineq2} and \eqref{maineq3intro}. The first corresponds to an equation where part of the full nonlinear structure of the NSE is preserved; we assumed that the nonlocal part depends linearly on the solution, while retaining $(u\cdot\nabla)u$. Moreover, we assumed that the nonlocal operator is of order zero, going from that to higher order seems nontrivial and out of reach at this point (in the presence of an advective nonlinearity). In particular, we are unable to show that solutions to 
\[
\partial_t u-\Delta u=(u\cdot\nabla)u+(-\Delta)^{\alpha}u,
\]
are regular, no matter how small $\alpha>0$ is. The special case when the initial data  (and hence the solution) is conservative, $u_0=\nabla\theta_0$ for some scalar $\theta_0$, was addressed in our previous work \cite{Ibdah2020b}. The difficulty is that one experiences a $2\alpha$ loss in regularity when propagating Lipschitz moduli of continuity, a fact that will be translated into a singular zero order term at the continuity level, which we were unable to control (when coupled with a Burgers nonlinearity). On the other hand, model \eqref{maineq3intro} is linear with a nonlocal term that mimics an operator of order $1-\beta$, as will be shown in Lemma \ref{pressest}, below. In this case, what we can guarantee under a supercritical assumption on the drift velocity is that the solution is differentiable almost everywhere, with a logarithmically integrable Lipschitz constant, leading to a partial regularity result for the NSE. This can be strengthened further under a critical assumption on the advecting quantity, leading to global regularity in the case of NSE.

Before starting to prove our results, let us explain at least heuristically the origin behind the double-exponential growth seen in Theorem \ref{thm2}, as well as the reason behind the exponential term (which requires a critical assumption) in Theorem \ref{thm4}. Very loosely speaking, this is reminiscent of the fact that ``lower order terms'' in parabolic operators in general would lead to exponential growth when trying to obtain maximum principles. This is best demonstrated via example: solutions to the viscous Burgers equation
\[
\partial_t u-\Delta u=(u\cdot\nabla)u(t,x),
\]
satisfy the classical maximum principle $\|u(t,\cdot)\|_{L^{\infty}}\leq \|u_0\|_{L^{\infty}}$. We show later on in \S\ref{pfthm1} that one also has $\|\nabla u(t,\cdot)\|_{L^{\infty}}\lesssim \|u_0\|_{W^{1,\infty}}$. However, adding a zero order term (with a bad sign)
\[
\partial_t u-\Delta u=(u\cdot\nabla)u(t,x)+u,
\]
would lead to a bound of the form $\|u(t,\cdot)\|_{L^{\infty}}\leq e^t\|u(0,\cdot)\|_{L^{\infty}}$, an exponential growth. It is not too far-fetched to expect that replacing $u$ with a zero-order non-local term $\mathcal{N}u$ would also lead to exponential growth (unless one specializes to certain operators that may be dispersive in nature). Recall that singular integral operators in general do prefer H\"older over Lipschitz functions. This potentially could explain why we see double-exponential growth in the Lipschitz bound in Theorem \ref{thm2}: one losses a ``logarithmic'' degree of regularity due to the fact that in most cases $\mathcal{N}u$ is at most H\"older if $u$ is only Lipschitz (and bounded). At the technical level, we would need our modulus of continuity to satisfy 
\[
\partial_t\Omega(t,\xi)-4\partial^2_{\xi}\Omega(t,\xi)-\Omega(t,\xi)\partial_\xi\Omega(t,\xi)-C_{K,d}\left[\int_{0}^{3\xi}\frac{\Omega(t,\eta)}{\eta}d\eta+\xi^{\rho}\int_{3\xi}^{\infty}\frac{\Omega(t,\eta)}{\eta^{1+\rho}}d\eta\right]\geq0,
\]
with the integral terms corresponding to a ``non-local'' zero order term. 

Similarly, the reason behind why estimate \eqref{supcritdd} is hard to obtain for the NSE is because the (non-local) pressure term would lead us to having to roughly deal with
\[
\partial_t\Omega(t,\xi)-\partial^2_\xi\Omega(t,\xi)-g(t)\xi^{\beta}\partial_\xi\Omega(t,\xi)-g(t)\xi^{1-\beta}\Omega(t,\xi)\geq0,
\]
together with certain integrals that  are of the same order as the last term in the above inequality. The transport term is due to the advective term $b\cdot\nabla$, while the singular zero order term is due to the pressure, and is not present when analyzing \eqref{classicalddintro}. So we  should expect at least some exponential growth when trying to apply this technique to the NSE. The fact that the coefficient of the zero order term is singular, while we are trying to measure Lipschitz continuity, is the main reason why we had to downgrade our assumption from a supercritical to a critical one.

\section{Preliminaries}\label{secprelim}
In \S\ref{classtheory}, we review the classical local well-posedness results and regularity criteria of smooth solutions. We proceed in \S\ref{modulusofcontinuity} to recall the various tools and results used in studying the evolution of moduli of continuity as introduced in \cite{Kiselev2011,KNS2008,KNV2007}.
\subsection{Local Well-Posedness and Regularity}\label{classtheory}
Let us start by recalling the definition of the homogenous Sobolev space of (integer) order $k\geq0$ over the torus $\tr$,
\begin{equation}\label{defHk}
\dot{H}^k:=\left\{\theta(x)=\sum_{n\in\mathbb{Z}^d}\hat\theta(n)e^{2\pi in\cdot x/L}:\hat\theta(0)=0,\ \hat\theta(-n)=\bar{\hat\theta}(n),\  \sum_{n\in\mathbb{Z}^d}|n|^{2k}|\hat\theta(n)|^2<\infty\right\},
\end{equation}
where 
\[
\hat\theta(n):=\frac{1}{L^d}\int_{\tr}\theta(x)e^{-2\pi in\cdot x/L}\ dx,
\]
and the corresponding space of divergence-free, periodic vector-fields with components belonging to $\dot{H}^k$,
\begin{equation}\label{defVk}
\mathbb{V}^k:=\left\{u:\rd\rightarrow\rd:u_j\in \dot{H}^k,j\in\{1,\cdots, d\},\ n\cdot\hat u(n)=0\ \forall n\in\mathbb{Z}^d\right\},
\end{equation}
where $\hat{u}(n)\in\mathbb{C}^d$ is the vector with components $\hat{u}_j(n)$.
Recall from the Sobolev imbedding theorem that if $\theta\in\dot{H}^k$ and $k>d/2$, then $\theta$ can be identified with a continuous function that is defined on all of $\rd$ and which is $L-$ periodic in all directions. In this work, we will always assume $k>d/2+2$, so that all of our functions can be realized as $C^2_{per}(\rd)$ functions. 

It is a well-known fact to experts and the mathematical fluid mechanics community that the pressure is recovered from the velocity vector-field by solving the elliptic problem (in space) for every $t\geq0$
\[
-\Delta p(t,x)=\text{div}\left[(u\cdot\nabla)u\right](t,x),
\]
see for instance \cite{Lemarie2016book,RRSBook2016} for details. In particular, using the fact that $u$ is divergence free, we see that
\[
-\Delta p(t,x)=\sum_{i,j=1}^d\partial_j u_i(t,x)\partial_i u_j(t,x),
\]
so that by Sobolev embedding and elliptic regularity results, if $u(t,\cdot)\in\mathbb{V}^k$ with $k>d/2+1$, we must have $p(t,\cdot)\in \dot{H}^{k+1}$. That is to say, the regularity of $u$ determines that of $p$. We then have the following version of short-time existence, uniqueness and regularity criterion for classical solutions to \eqref{maineq}. A proof (as well as analogous results on bounded domains with Dirichlet conditions or whole space scenarios) can be found in any of the classical textbooks on fluid mechanics, for instance \cite{CFbook1998, Ladyzhenskayabook1969, Lemarie2016book,RRSBook2016}.
\begin{thm}\label{shortime}
Let $d\geq 3$, $k>d/2+2$ and suppose $u^0\in\mathbb{V}^k$. Then there is a $T_0=T_0(\|u^0\|_{\mathbb{V}^k},L)>0$ and a solution $(u,p)$ to \eqref{maineq} with regularity
\begin{equation}\label{regclass}
u\in C([0,T_0];\mathbb{V}^k)\cap C^1([0,T_0];\mathbb{V}^{k-2}), \quad p\in C([0,T_0];\dot{H}^{k+1})\cap C^1([0,T_0];\dot{H}^{k-2}),
\end{equation}
for which 
\[
\partial_tu(t,x)-\Delta u(t,x)=(u\cdot\nabla)u(t,x)+\nabla p(t,x),
\]
holds true in $C([0,T_0];\mathbb{V}^{k-2})$; in particular, since $k>d/2+2$, it holds true in the pointwise sense for every $(t,x)\in [0,T_0]\times\rd$. Furthermore, $(u,p)$ is unique in the regularity class \eqref{regclass} (in fact, depends continuously on initial data), and if
\[
T_k:=\sup\left\{T>0:\|u(t,\cdot)\|_{\mathbb{V}^k}<\infty,\ \forall t\in[0,T] \right\},
\]
is the maximal time of existence of the solution in $\mathbb{V}^k$, then $T_k=T_*$, where
\[
T_*:=\sup\left\{T>0:\|\nabla u(t,\cdot)\|_{L^{\infty}}<\infty,\ \forall t\in[0,T]\right\}.
\]
In particular, if $u^0\in \mathbb{V}^k$ for every $k\geq0$, then the pair $(u,p)$ is a classical solution to \eqref{maineq} on $[0,T_*)$.
\end{thm}
The fact that the solution is regular on $[0,T_*)$ follows from, for instance, the famous Beale-Kato-Majda \cite{BKM1984} criterion, which asserts that smooth solutions to the Navier-Stokes and Euler equations do not develop singularities on $[0,T]$ provided 
\[
\int_0^T\|\nabla\times u(t,\cdot)\|_{L^{\infty}}\ dt<\infty.
\] 
One can also control higher order norms of the solution on $[0,T]$ in terms of 
\[
\sup_{t\in[0,T)}\|\nabla u(t,\cdot)\|_{L^{\infty}}
\]
by direct energy estimates via utilizing the following calculus inequality valid for smooth enough functions $f$ and $g$ and proven in \cite[Lemma A.1]{KM1981},
\begin{equation}\label{prodest}
\left\|\partial^{\alpha}(fg)\right\|_{L^2}\leq C_{d,\alpha}\left(\|f\|_{L^{\infty}}\|\partial^{\alpha}g\|_{L^2}+\|g\|_{L^{\infty}}\|\partial^{\alpha}f\|_{L^2}\right).
\end{equation}

We also need an analogous result for equation \eqref{maineq2}, which we now state. Note that in this case, we no longer require incompressibility, and so we look for solutions in the space
\[
\dot{\mathbb{H}}^k:=\left\{u:\rd\rightarrow\rd:u_j\in \dot{H}^k,j\in\{1,\cdots, d\}\right\}.
\]
\begin{thm}\label{shortime2}
Let $d\geq 2$, $k>d/2+2$ and suppose $u^0\in\dot{\mathbb{H}}^k$. Then there is a $T_0=T_0(\|u^0\|_{\dot{\mathbb{H}}^k},L)>0$ and a solution $v$ to \eqref{maineq2} with regularity
\begin{equation}\label{regclass2}
\begin{aligned}
&u\in C([0,T_0];\dot{\mathbb{H}}^k)\cap C^1([0,T_0];\dot{\mathbb{H}}^{k-2}),
\end{aligned}
\end{equation}
for which 
\[
\partial_t u(t,x)-\Delta u(t,x)=(u\cdot\nabla)u(t,x)+\mathcal{N}_pu(t,x),
\]
holds true in $C([0,T_0];\dot{\mathbb{H}}^{k-2})$; in particular, since $k>d/2+2$, it holds true in the pointwise sense for every $(t,x)\in [0,T_0]\times\rd$. Furthermore, $u$ is unique in the regularity class \eqref{regclass2} (in fact, depends continuously on initial data), and if
\[
T_k:=\sup\left\{T>0:\|u(t,\cdot)\|_{\dot{\mathbb{H}}^k}<\infty,\ \forall t\in[0,T] \right\},
\]
is the maximal time of existence of the solution in $\dot{\mathbb{H}}^k$, then $T_k=T_*$, where
\[
T_*:=\sup\left\{T>0:\|\nabla u(t,\cdot)\|_{L^{\infty}}<\infty,\ \forall t\in[0,T]\right\}.
\]
In particular, if $u^0\in \dot{\mathbb{H}}^k$ for every $k\geq0$, then $u$ is a classical solution to \eqref{maineq2} on $[0,T_*)$.	
\end{thm}
The proof of the above theorem is pretty much similar to that of Theorem \ref{shortime}; one can construct a solution via Galerkin approximations and utilize the energy bound $\|(-\Delta)^{k/2}\mathcal{N}_p \theta\|_{L^2}\leq C\|\theta\|_{\dot{H}^k}$, while the regularity criterion can be obtained by utilizing the product estimate \eqref{prodest}.

\subsection{Moduli of Continuity}\label{modulusofcontinuity}
Let us review some of the terminology and basic results used when studying the evolution of moduli of continuity. Such results were collectively obtained in \cite{Kiselev2011, KNS2008, KNV2007} and were summarized in \cite{Ibdah2020b}. Nevertheless, we provide proofs here for the sake of convenience and completeness. Moreover, Lemma \ref{buildmod} allows for moduli of continuity that may be bounded, which we can consider in this case since the scaling invariance of the Navier-Stokes system is different from that of the critical SQG and critical Burgers equation. Throughout this work, for any $z\in\rd$, by $z_j$ we mean the $j^{th}$ coordinate of $z$, and by $|z|$ we mean the standard Euclidean norm. Given any vector field $u:\rd\rightarrow\mathbb{R}^d$, we remind the reader that if $X$ is a space consisting of scalar functions defined on $\rd$, we abuse notation and say $u\in X$ provided each component $u_j\in X$. Further, given any vector field $u\in W^{1,\infty}(\rd)$, we define its Lipschitz constant as 
\begin{equation}\label{deflipvf}
\|\nabla u\|_{L^\infty}:=\sup_{x\in \rd}\sup_{\substack {e\in \rd\\ |e|=1}}\left|J_u(x)e\right|,
\end{equation}
where $J_u(x)$ is the Jacobian matrix of $u$, the matrix whose row vectors are $\nabla u_j$, evaluated at a point $x\in\rd$. We invite the reader to verify that 
\begin{equation}\label{normequiv}
\frac{1}{\sqrt{d}}\left(\sum_{j=1}^d\|\nabla u_j\|_{L^{\infty}}^2\right)^{1/2}\leq \|\nabla u\|_{L^{\infty}}\leq \left(\sum_{j=1}^d\|\nabla u_j\|_{L^{\infty}}^2\right)^{1/2},
\end{equation}
so that $u\in W^{1,\infty}(\rd)$ if and only if $u$ is bounded and $\|\nabla u\|_{L^{\infty}}<\infty$. Furthermore, from the fundamental theorem of calculus, we see that $|u(x)-u(y)|\leq \|\nabla u\|_{L^{\infty}}|x-y|$.
\begin{defs}\label{defobeymod}
Let $\omega$ be a modulus of continuity as in Definition \ref{defmod} and let $u:\rd\rightarrow\mathbb{R}^d$ be a vector field. We say $u$ has modulus of continuity $\omega(\xi)$ if $|u(x)-u(y)|\leq\omega(|x-y|)$ for every $(x,y)\in\rd\times\rd$. We say $u$ strictly obeys $\omega$ if $|u(x)-u(y)|<\omega(|x-y|)$ whenever $x\neq y$.	
\end{defs}
To avoid cumbersome notation, we drop the $L^{\infty}$ subscript from $\|\cdot\|_{L^{\infty}}$ in the proofs of the following two Lemmas. 
\begin{lem}\label{buildmod}
Let $u\in W^{1,\infty}\left(\mathbb{R}^d\right)$ be a bounded and Lipschitz vector field, and suppose $\omega$ is a modulus of continuity. Suppose further that $\omega$ is concave on a small interval $(0,\delta]$ some $\delta>0$. Then there exists $B_{0}>0$ depending only on $\|u\|_{L^{\infty}}$, $\|\nabla u\|_{L^{\infty}}$, and $\omega$ such that $u$ strictly obeys $B\omega(B\xi)$ whenever $B\geq B_0$. If $\omega$ is unbounded and concave on $(0,\infty)$, then $B_0$ can be chosen such that $u$ strictly obeys $\omega(B\xi)$ whenever $B\geq B_0$.
\end{lem}
\begin{rmk}
In particular, strong moduli of continuity according to Definition \ref{defmod} satisfy the hypothesis of the above Lemma.
\end{rmk}

\begin{proof}
We ignore the trivial case when $u$ is a constant vector. Set 
\begin{equation}\label{bnod}
B_0:=\frac{2}{\omega(\delta)}\|u\|+\left(\frac{\delta}{\omega(\delta)}\|\nabla u\|\right)^{1/2},
\end{equation}
and let $\xi:=|x-y|>0$. Notice that as $\omega$ is concave on $(0,\delta]$, the function
\[
h(\xi):=\|\nabla u\|-\frac{B_0\omega(B_0\xi)}{\xi}
\]
is increasing on $(0,\delta B_0^{-1}]$. This means that 
\[
h(\xi)\leq h(\delta B_0^{-1})=\|\nabla u\|-\frac{B_0^2\omega(\delta)}{\delta}<0,
\]
by choice of $B_0$ \eqref{bnod}. Hence, for $|x-y|\in(0,\delta B_0^{-1}]$, we must have 
\[
|u(x)-u(y)|\leq \|\nabla u\||x-y|<B_0\omega(B_0|x-y|),
\]	
where the first inequality follows from definition \eqref{deflipvf} and the fundamental theorem of calculus. On the other hand, for $\xi\in[\delta B_0^{-1},\infty)$, since $\omega$ is nondecreasing, we have again by choice of $B_0$,
\[
B_0\omega(B_0\xi)\geq B_0\omega(\delta)>2\|u\|\geq|u(x)-u(y)|.
\]
Therefore, we get that for any $|x-y|\in(0,\infty)$,
\[
|u(x)-u(y)|<B_0\omega(B_0|x-y|),
\]
from which we conclude that since $\omega$ is nondecreasing, $u$ has strict modulus of continuity $B\omega(B\xi)$ for any $B\geq B_0$. If $\omega$ is unbounded and concave on $(0,\infty)$, then we can replace $B_0$ from \eqref{bnod} by a large $B_0$ depending on $\omega$ such that $\omega(B_0)\geq \max\{\|\nabla u\|+1,2\|u\|+1\}$ and then repeat the previous argument with $\delta=1$.
\end{proof}
\begin{lem}\label{gradbd}
Suppose a vector field $u\in C^2(\mathbb{R}^d)\cap W^{2,\infty}(\rd)$ has a strong modulus of continuity $\omega$. It then follows that $u$ is Lipschitz and $\|\nabla u\|_{L^{\infty}}<\omega'(0)$.
\end{lem}
\begin{rmk}\label{rmklips}
	The bound $\|\nabla u\|_{L^{\infty}}\leq \omega'(0)$ follows from Definition \ref{defmod} and the Fundamental Theorem of Calculus. The important part is the strict inequality, for which we need $\omega''(0)=-\infty$, and $u\in C^2\cap W^{2,\infty}$.
\end{rmk}
\begin{proof}
Let $x\in\mathbb{R}^d$ be arbitrary, and for any $\xi\in(0,1]$, let $\displaystyle{y=x+\xi e}$, where $e$ is any unit vector. From the first order Taylor expansion of $u$ about $x$ we see that 
\[
|u(y)-u(x)|\geq |J_u(x)e|
\xi-C\xi^2\|\nabla^2u\|,
\]
here $\|\nabla^2u\|$ is just the maximum of all second order derivatives and $C$ is a combinatorial constant. The left hand side is at most $\omega(\xi)$, and so after rearranging we get for any $x\in\rd$, any $\xi \in(0,1]$, and any unit vector $e\in\mathbb{S}^{d-1}$,
\[
		|J_u(x)e|\leq\frac{\omega(\xi)}{\xi}+\frac{C\xi}{2}\|\nabla^2u\|.
\]
As the right hand side is independent on $x$ and $e$ we have
\begin{equation}\label{comp3}
\|\nabla u\|\leq \frac{\omega(\xi)}{\xi}+\frac{C\xi}{2}\|\nabla^2u\|.
\end{equation}
Since $\omega$ is $C^2$ on $(0,\infty)$, and $\displaystyle{\lim_{\xi\rightarrow0^+}\omega''(\xi)=-\infty}$, it follows that
\[
\omega(\xi)=\omega(\xi/2)+\frac{\omega'(\xi/2)}{2}\xi-\rho(\xi)\xi^2,
\]
where $\displaystyle{\lim_{\xi\rightarrow0^+}\rho(\xi)=\infty}$.
Plugging this into \eqref{comp3} we get
\[
\|\nabla u\|\leq\frac{\omega(\xi/2)}{\xi}+\frac{\omega'(\xi/2)}{2}+\xi\left(C\|\nabla^2u\|-\rho(\xi)\right),
\]
The result now follows by choosing $\xi\in(0,1]$ small enough such that $C\|\nabla^2\theta\|-\rho(\xi)<0$ and noting that 
\[
\frac{\omega(\xi/2)}{\xi}+\frac{\omega'(\xi/2)}{2}<\frac{\omega'(0)}{2}+\frac{\omega'(0)}{2}=\omega'(0),
\]
where in the last inequality we used concavity of $\omega$ on some small interval $(0,\delta_0]$, since $\omega''(0^+)=-\infty$ and $\omega$ is piecewise $C^2(0,\infty)$.
\end{proof}
The following Lemma is nothing but a rigorous justification of the consequences of the function
\[
\theta(x)-\theta(y)-\omega(|x-y|)
\]
having a maximum at some $x_0\neq y_0$: one expects first order derivatives to vanish, while the Laplacian to be non-positive. Since $\omega$ is only assumed to be piecewise $C^2$, we have to do this carefully. See also  \cite[Proposition~2.4]{Kiselev2011} and \cite[Lemma~2.3]{Ibdah2020b}. We could prove an analogous vector version of this results, but that won't be necessary. For our purposes, this will do.
\begin{lem}\label{dervh}
	Suppose $\theta$ is $C^2(\mathbb{R}^d)$ is a scalar and has modulus of continuity $\omega$. If $\theta(x_0)-\theta(y_0)=\omega(|x_0-y_0|)$ for some $x_0\neq y_0$ with $x_0-y_0=\xi e_m$, some $m\in\{1,\cdots,d\}$ and $\xi>0$. Then 
	\begin{equation}\label{derv}
	\begin{cases}
		\omega'(\xi^-)\leq\partial_m\theta(x_0)=\partial_m\theta(y_0)\leq\omega'(\xi^+),\\
		\partial_j\theta(x_0)=\partial_j\theta(y_0)=0, \quad j\neq m,
	\end{cases}
	\end{equation}
and 
\begin{equation}\label{lap}
\Delta\theta(x_0)-\Delta\theta(y_0)\leq 4\omega''(\xi^-).
\end{equation}
Furthermore, for any  $\alpha\in(0,1)$, there exists a $C_{\alpha}>0$ such that 
\begin{equation}\label{fraclap}
-\left[(-\Delta)^{\alpha}\theta(x_0)-(-\Delta)^{\alpha}\theta(y_0)\right]\leq D_{\alpha}[\omega](\xi),
\end{equation}
where
\begin{align}\label{defdalpha}
D_{\alpha}[\omega](\xi)=&C_{\alpha}\int_0^{\xi/2}\frac{\omega(\xi+2\eta)+\omega(\xi-2\eta)-2\omega(\xi)}{\eta^{1+2\alpha}}d\eta\nonumber\\
&+C_{\alpha}\int_{\xi/2}^{\infty}\frac{\omega(\xi+2\eta)-\omega(2\eta-\xi)-2\omega(\xi)}{\eta^{1+2\alpha}}d\eta,
\end{align}
\end{lem}
\begin{proof}
We will only prove \eqref{derv} and \eqref{lap}.  We refer the reader to \cite{Kiselev2011} for the  proof of \eqref{fraclap} with \eqref{defdalpha}. Since moduli of continuity are not sensitive to distance preserving maps, we may assume that $m=1$. We start by showing $\partial_j\theta(x_0)=\partial_j\theta(y_0)$ and $\partial^2_j\theta(x_0)-\partial^2_j\theta(y_0)\leq0$ any $j\in\{1,\cdots,d\}$. Let $\epsilon>0$ and define
\begin{align*}
&d_\epsilon^+:=\theta(x_0+\epsilon e_j)-\theta(y_0+\epsilon e_j)-\left[\theta(x_0)-\theta(y_0)\right],\\
&d_\epsilon^-:=\theta(x_0)-\theta(y_0)+\left[\theta(y_0-\epsilon e_j)-\theta(x_0-\epsilon e_j)\right]	,\\
&d_\epsilon:=\left[\theta(x_0+\epsilon e_j)-2\theta(x_0)+\theta(x_0-\epsilon e_j)\right]-\left[\theta(y_0+\epsilon e_j)-2\theta(y_0)+\theta(y_0-\epsilon e_j)\right],
\end{align*}
	where $\{e_j\}_{j=1}^d$ is the standard unit basis of $\mathbb{R}^d$. It is sufficient to show $d_\epsilon^+\leq0$, $d_\epsilon^-\geq0$ and $d_\epsilon\leq0$. But this follows immediately from the fact that $\theta(x_0)-\theta(y_0)=\omega(\xi)$ and $|\theta(x)-\theta(y)|\leq\omega(|x-y|)$ for any $x,y$. 
	
Next, we show that $\partial_j\theta(x_0)=0$ for $j\neq 1$ and $\omega'(\xi^-)\leq\partial_1\theta(x_0)\leq\omega'(\xi^+)$. For $j\in\{1,\cdots,d\}$, we define
\begin{align*}
&d_{\epsilon,j}^+:=\theta(x_0+\epsilon e_j)-\theta(x_0)=\theta(x_0+\epsilon e_j)-\theta(y_0)-\omega(\xi),\\
&d_{\epsilon,j}^-:=\theta(x_0)-\theta(x_0-\epsilon e_j)=\omega(\xi)+\theta(y_0)-\theta(x_0-\epsilon e_j).
\end{align*}
Notice that for $j=1$, we have $|x_0+\epsilon e_1-y_0|=\xi+\epsilon$, and $|y_0-x_0+\epsilon e_1|=\xi-\epsilon$ whenever $\epsilon\in(0,\xi/2)$, while for $j>1$, $|x_0+\epsilon e_j-y_0|=|y_0-x_0+\epsilon e_j|=\sqrt{\xi^2+\epsilon^2}$. Hence,
\begin{align}
&d_{\epsilon,j}^+\leq 
\begin{cases}
	\omega(\xi+\epsilon)-\omega(\xi), &j=1,\\
	\omega(\sqrt{\xi^2+\epsilon^2})-\omega(\xi), &j>1
\end{cases},\label{depsplus}\\
&d_{\epsilon,j}^-\geq 
\begin{cases}
	\omega(\xi)-\omega(\xi-\epsilon), &j=1,\\
	\omega(\xi)-\omega(\sqrt{\xi^2+\epsilon^2}), &j>1
\end{cases}\label{depsminus}	,
\end{align}
from which \eqref{derv} follows immediately upon dividing \eqref{depsplus} and \eqref{depsminus} by $\epsilon>0$ and letting $\epsilon\rightarrow 0^+$, since $\omega$ is continuous and have one-sided derivatives. Finally, let $x_0=(s_0,x')$, $y_0=(s_0',x')$ where $s_0,s_0'\in\mathbb{R}$ are the first coordinates, and $x'\in\mathbb{R}^{d-1}$ are the other coordinates, and define
 \[
 h(s):=\theta(s,x')-\theta(s_0+s_0'-s,x')-\omega(2s-s_0-s_0'),\ s>\frac{s_0+s_0'}{2}.
 \]
Suppose for the sake of contradiction that $\partial^2_1\theta(x_0)-\partial^2_1\theta(y_0)>4\omega''(\xi^-)$. As $\omega$ is piecewise $C^2$, it follows that there exists some small enough $\epsilon>0$ such that $h$ is $C^2$ on $[s_0-\epsilon,s_0]$ and $-h''(s)<0$ on that interval. On the one hand, a Lemma of Hopf (or simple calculus) tells us that we must have $h'(s_0^{-})>0$. On the other hand, owing to \eqref{derv}, we must have
\[
h'(s_0^{-})=2\left(\partial_1\theta(x_0)-\omega'(\xi^-)\right)\leq 2\left(\omega'(\xi^+)-\omega'(\xi^-)\right),
\]
 which leads to a contradiction under the assumption $\omega'(\xi^+)\leq\omega'(\xi^-)$.
\end{proof}
\begin{rmk}\label{imprmk}
We emphasize that under the assumption that $\omega'(\xi^+)\leq\omega'(\xi^-)$, we see from \eqref{derv} that the equality $\theta(x_0)-\theta(y_0)=\omega(\xi)$ cannot happen if $\xi$ is a point of jump discontinuity of $\omega'$, since $\theta$ is smooth.
\end{rmk}

\section{Local Continuity Estimates}\label{seccontest}
As mentioned earlier, it was observed by Silvestre \cite{Silvestre2010unpub} that $C^{\beta}$ incompressible vector-fields have $C^{2\beta}$ scalar pressures. When $\beta\in(1/2,1)$, this translates to a H\"older condition on the gradient of the pressure. Similar results were reported by Constantin \cite{Constantin2014}, Isett \cite{isett2013regularity} (Isett and Oh \cite{IO2016}) as well as De Lellis and Sz\'ekelyhidi Jr. \cite{DS2014}. See also Colombo and De Rosa \cite{CDR2020} for further results among those lines. Constantin's proof is based on some local formulae regarding spatial averages of the pressure term, and is somewhat related to Silvestre's argument. Isett's proof is based on Littlewood-Paley decompositions, while De Lellis and Szek\'ekelyhidi Jr. observed this naturally in the context of convex integration. The key observation made by Silvestre to obtain such an estimate is the following identity valid for a smooth enough, divergence-free vector field $u$:
\begin{equation}\label{incomptrick}
\sum_{i,j=1}^d\partial_{y_i}\partial_{y_j}\left[u_i(x-y)u_j(x-y)\right]=\sum_{i,j=1}^d\partial_{y_i}\partial_{y_j}\left[\left(u_i(x-y)-u_i(x)\right)\left(u_j(x-y)-u_j(x)\right)\right].
\end{equation} 
Here $x\in\rd$ is fixed, and we used the notation $\partial_{z_k}[\theta(z)]$ to mean the derivative of $\theta:\rd\rightarrow\mathbb{R}$ with respect to the variable $z_k$, i.e. $\partial_{z_k}[\theta(z)]=\partial_k\theta(z)$, where $\partial_k\theta(z)$ is understood as $\nabla\theta\cdot e_k$ evaluated at  $z$. This notation will be used throughout the proof of the following Lemma, a generalization of Silvestre's argument and an extension to the periodic setting. Identity \eqref{incomptrick} is utilized as follows: from the fact that the pressure solves
\begin{equation}\label{elliptequppf}
-\Delta p=\sum_{i,j}\partial_i\partial_j(u_iu_j),
\end{equation}
a pointwise representation in terms of the fundamental solution to Laplace's equation in $\rd$ is obtained, and then the H\"older seminorms of $p$ and $\nabla p$ are controlled in terms of the H\"older seminorms of $u$ via utilizing the incompressibility trick \eqref{incomptrick}. We remind the reader that we are identifying functions belonging to $\mathbb{V}^k$ by functions that are defined on $\rd$ and are $L-$ periodic in every direction.
\begin{lem}\label{pressest}
	Let $u$ and $b$ be continuous divergence-free vector fields. Suppose further that either $u,b\in C_{per}(\rd)$ or $b\in L^{q}(\rd)$ and $u\in L^{\infty}(\rd)$ for some $q\in(1,\infty)$. Assume $u$ and $b$ have moduli of  continuity $\omega_{u}$ and $\omega_{b}$ respectively. If 
	\begin{equation}\label{defp}
	p:=\sum_{i,j=1}^dR_iR_j(b_iu_j),
	\end{equation}
where $\{R_j\}_{j=1}^d$ are the Riesz transforms, then $p\in C^1(\rd)$ and $\nabla p$ has modulus of continuity
\begin{equation}\label{modgradp}
\widetilde\omega(\xi):=C_d\left[\int_0^\xi\frac{\omega_{b}(\eta)\omega_{u}(\eta)}{\eta^2}d\eta+\omega_{b}(\xi)\int_{\xi}^\infty\frac{\omega_{u}(\eta)}{\eta^2}\ d\eta+\omega_{u}(\xi)\int_{\xi}^\infty\frac{\omega_{b}(\eta)}{\eta^2}\ d\eta\right],
\end{equation}
where $C_d$ is a positive, absolute universal constant depending only on the spatial dimension $d\geq3$ but not on any norm of $u$, $b$ or $p$ (provided the integrals converge).
\end{lem}
\begin{rmk}
If $b=u$ was H\"older continuous with exponent $\beta\in(1/2,1)$, then we recover Silvestre's estimate \cite{Silvestre2010unpub} by choosing $\omega_{b}(\xi)=\omega_{u}(\xi)=[u]_{C^{0,\beta}}\xi^{\beta}$, where $[u]_{C^{0,\beta}}$ is the H\"older semi-norm of $u$. Further, one may allow different components of the vector fields $u$ and $b$ to have different moduli of continuity, but this doesn't seem to be helpful here so we just assume the whole vector field obeys one modulus of continuity.
\end{rmk}
\begin{rmk}
The above Lemma is applicable for Leray-Hopf weak solutions that obey moduli of continuity, \emph{provided the integrals converge}. It was shown in \cite{Constantin2001, Constantin2014, FGT1981} that Leray-Hopf solutions reside in $L^1([0,T];L^{\infty}(\rd))$ meaning that we can apply this Lemma with $b(t,\cdot)=u(t,\cdot)\in L^2(\rd)\cap L^{\infty}(\rd)$ for almost every $t\in[0,T]$. One can also certainly allow for $u$ and $b$ to be in different $L^{q}$ spaces, but we prefer to keep the presentation simple without sacrificing rigour. 
\end{rmk}

\begin{proof}
In what follows, $C_d$ is nothing but an absolute universal constant depending only on the dimension $d\geq3$ and whose value may change from line to line. As Riesz transforms are bounded on $L^q$ spaces, $q\in(1,\infty)$, see for instance \cite{Stein1970book} for the whole space setting, \cite{CZ1954} and \cite{SWbook1971} for periodic functions, $p$ in \eqref{defp} is well defined under our hypotheses. The first step is to show that $p\in C^1(\rd)$ and we have the representation  
\begin{equation}\label{gradp}
\partial_k p(x)=\sum_{i,j=1}^d\int_{\rd}\partial_k\partial_i\partial_j\Phi(y)\varphi_{i,j}(x,y)\ dy, \quad \forall k\in\{1,\cdots,d\},
\end{equation}
where $\Phi(y):=C_d|y|^{2-d}$ is the fundamental solution to Laplace equation in $d\geq3$ and 
\begin{equation}\label{varphi}
\varphi_{i,j}(x,y):=\left(b_i(x-y)-b_i(x)\right)\left(u_j(x-y)-u_j(x)\right),\quad (x,y)\in\rd\times\rd.
\end{equation}
Note that since $|\partial_k\partial_i\partial_j\Phi(y)|\leq C_d|y|^{-d-1}$, the integral \eqref{gradp} always converges under the hypotheses of Lemma \ref{pressest}. To see this, note that if the first integral appearing in \eqref{modgradp} converges, then the singularity at 0 in \eqref{gradp} is integrable. On the other hand, in the periodic case, $\varphi_{i,j}\in L^{\infty}(\rd\times\rd)$, so that the tail end of integral \eqref{gradp} is finite, while under the assumption that $b\in L^{q}(\rd)$, $u\in L^{\infty}(\rd)$ for some $q\in(1,\infty)$, the tail end of the integral is convergent since the product $b_iu_j\in L^q(\rd)$ and $\partial_k\partial_i\partial_j\Phi\in L^q(A)$, any $q\in[1,\infty]$ where $A:=\{y\in\rd:|y|\geq1\}$. 

We get that $p\in C^1(\rd)$ immediately provided we show that the distributional derivative of $p$ has the representation \eqref{gradp}, since the right-hand side of \eqref{gradp} is continuous. To that extent, we start with the case when $b\in L^{q}(\rd)$ and $u\in L^{\infty}(\rd)$ for some $q\in(1,\infty)$. It follows that $p$ as defined in \eqref{defp} is in $L^{q}(\rd)$. We now let $\epsilon>0$ be small, obtain smooth, divergence-free vector-fields $b^{\epsilon}$ and $u^{\epsilon}$ such that $b^{\epsilon}\rightarrow b$ in $L^{q}(\rd)$, $u^{\epsilon}(x)\rightarrow u(x)$ pointwise, assume $b^{\epsilon}$ is compactly supported while $u^{\epsilon}$ is uniformly bounded in $\epsilon>0$ and we define
\[
p^{\epsilon}:=\sum_{i,j=1}^dR_iR_j(b_i^{\epsilon}u_j^{\epsilon}).
\] 
Moving on, we drop the summation for convenience. It follows that $p^{\epsilon}$ is smooth and solves \eqref{elliptequppf} (with one of the factors chosen as $u=u_\epsilon$ and the other factor being $b_\epsilon$). Since the product $b^{\epsilon}_iu^{\epsilon}_j$ is compactly supported, we therefore have the representation 
\[
p^{\epsilon}(x)=\int_{\rd}\Phi(y)\partial_{y_i}\partial_{y_j}(b^{\epsilon}_i(x-y) u^{\epsilon}_j(x-y))\ dy=\int_{\rd}\Phi(y)\partial_{y_i}\partial_{y_j}\left[\varphi^{\epsilon}_{i,j}(x,y)\right]\ dy,
\]
where $\varphi^{\epsilon}_{i,j}$ is as in \eqref{varphi} and we used the incompressibility trick \eqref{incomptrick} (with one of the $u$ factors replaced by $b$). Thus, the representation \eqref{gradp} follows immediately for $\partial_k p^{\epsilon}$ via differentiating under the integral and integrating by parts. Finally, one can readily verify that up to a subsequence, $\partial_kp^{\epsilon}$ converges pointwise almost everywhere to the representation \eqref{gradp} as $\epsilon\rightarrow 0$, and since $p^{\epsilon}$ converges to $p$ in $L^{q}(\rd)$, we get that the distributional derivative can be represented by \eqref{gradp}, meaning $p\in C^1(\rd)$.

Now on to the periodic case. The approximating sequence is simpler: we just assume that $u,b$ are divergence-free trigonometric polynomials. The complication now is that we cannot justify convolving with $\Phi$ directly. However, we are interested in the pointwise representation \eqref{gradp}, and as discussed above, this integral is absolutely convergent. Therefore, we proceed by cutting off the part at infinity by a smooth function, obtain the required representation and then pass to the limit. To that end, let $\chi:\rd\rightarrow[0,1]$ be a smooth radially symmetric function such that $\chi(z)=1$ for $|z|\leq 1$ and  $\chi(z)=0$ for $|z|\geq 2$, and for $R>0$, set $\chi_R(y):=\chi(y/R)$. Fix $x\in\mathbb{R}^d$ and chose $R\geq2|x|+1$. Noting that since $|\nabla\Phi(\lambda)|\leq C_d |\lambda|^{1-d}$ and $|\nabla^2\Phi(\lambda)|\leq C_d|\lambda|^{-d}$ for $\lambda\neq 0$, while all derivatives of $\chi$ are supported in the shell $1\leq |y|\leq2$, we can bound, for any $f\in L^{\infty}(\rd)$, $m,n,k\in\{1,\cdots,d\}$
\begin{align}
&\left|\int_{\mathbb{R}^d}\partial_{k}\Phi(x-y)\partial_m\partial_n\chi(y/R)f(y)\ dy \right|\leq C_d \|f\|_{L^{\infty}}\int_{R\leq|y|\leq2R}|x-y|^{1-d}\ dy\leq C_d\|f\|_{L^{\infty}}R,\label{phi'}\\
&\left|\int_{\rd}\partial_k\partial_n\Phi(x-y)\partial_m\chi(y/R)f(y)\ dy\right|\leq C_d\|f\|_{L^{\infty}}\int_{R\leq|y|\leq2R}|x-y|^{-d}\ dy\leq C_d\|f\|_{L^{\infty}}\label{phi''}.
\end{align}Next, let $p_R(y):=p(y)\chi_R(y)$, and note that
\[
-\Delta p_R(y)=\partial_i\partial_j(b_iu_j)\chi_R(y)-2\nabla p\cdot\nabla\chi_R(y)-p\Delta\chi_R(y),
\]
so that upon multiplying the above equation by $\Phi(x-y)$ and integrating the left-hand side by parts, we get (since $p_R\in C^2(\rd)$ and is compactly supported)
\begin{align*}
p_R(x)=&\int_{\mathbb{R}^d}\Phi(x-y)\partial_{i}\partial_{j}(b_iu_j)(y)\chi_R(y)\ dy-2R^{-1}\sum_{l=1}^d\int_{\mathbb{R}^d}\Phi(x-y)\partial_{l}p(y)\partial_{l}\chi(y/R)\ dy\\
&-R^{-2}\int_{\mathbb{R}^d}\Phi(x-y)p(y)\Delta\chi(y/R)\ dy,
\end{align*}
meaning 
\begin{align*}
\partial_kp_R(x)=&\int_{\mathbb{R}^d}\partial_k\Phi(x-y)\chi_R(y)\partial_{i}\partial_{j}(b_iu_j)(y)\ dy
-2R^{-1}\sum_{l=1}^d\int_{\mathbb{R}^d}\partial_{k}\Phi(x-y)\partial_{l}\chi(y/R)\partial_{y_l}[p(y)]\ dy\\
&-R^{-2}\int_{\mathbb{R}^d}\partial_{k}\Phi(x-y)\Delta\chi(y/R)p(y)\ dy.
\end{align*}
Let us start with integrating the second integral by parts, transferring the derivative from the pressure to the other terms to get
\[
\int_{\mathbb{R}^d}\partial_{k}\Phi(x-y)\partial_{l}\chi(y/R)\partial_{y_l}[p(y)]\ dy=\int_{\mathbb{R}^d}\left[\partial_{l}\partial_{k}\Phi(x-y)\partial_{l}\chi(y/R)-R^{-1}\partial_{k}\Phi(x-y)\partial_{l}^2\chi(y/R)\right]p(y)\ dy,
\]
making 
\begin{align}\label{gradp1}
\partial_kp_R(x)=&\int_{\mathbb{R}^d}\partial_k\Phi(x-y)\chi_R(y)\partial_i\partial_j(b_iu_j)(y)\ dy
-2R^{-1}\sum_{l=1}^d\int_{\mathbb{R}^d}\partial_l\partial_{k}\Phi(x-y)\partial_{l}\chi(y/R)p(y)\ dy\nonumber\\
&+R^{-2}\int_{\mathbb{R}^d}\partial_{k}\Phi(x-y)\Delta\chi(y/R)p(y)\ dy=I_1(x)+I_2(x)+I_3(x).
\end{align}
For fixed $x\in\rd$, we say $\theta(x)=O(1/R)$ if there  exists a constant $C$ such that $|\theta(x)|\leq CR^{-1}$ for any $R\geq2|x|$. That being said, by virtue of \eqref{phi'}-\eqref{phi''} and the fact that $p\in C^2_{per}(\rd)\subset L^{\infty}(\rd)$, we see that $I_2(x)$ and $I_3(x)$ are both $O(1/R)$. For $I_1(x)$, we first invoke a change of variables, use the incompressibility trick \eqref{incomptrick}, and then integrate by parts twice to get 
\[
I_1(x)=\int_{\mathbb{R}^d}\partial_k\Phi(y)\chi_R(x-y)\partial_{y_i}\partial_{y_j}[\varphi_{i,j}(x,y)]\ dy=\int_{\mathbb{R}^d}\partial_i\partial_j\partial_k\Phi(y)\varphi_{i,j}(x,y)\chi_R(x-y)\ dy+O(1/R)
\]
where $\varphi_{i,j}(x,y)$ is as defined in \eqref{varphi} and we again used calculations similar to \eqref{phi'}-\eqref{phi''} to conclude that the leftover terms from integration by parts are $O(1/R)$. Therefore, we get for any $x\in\mathbb{R}^d$,
\[
\partial_kp_R(x)=\int_{\mathbb{R}^d}\partial_k\partial_i\partial_j\Phi(y)\varphi_{i,j}(x,y)\chi_R(x-y)\ dy+O(1/R),
\]
from which \eqref{gradp} follows by passing to the limit $R\rightarrow\infty$. 

Now that we proved the representation \eqref{gradp} holds true in both the periodic and whole space scenario, we proceed to obtaining estimate \eqref{modgradp}. We let $(x,z)\in\rd\times\rd$, $\xi:=|x-z|$, and write 
\begin{align*}
|\partial_k p(x)-\partial_k p(z)|&=\sum_{i,j=1}^d\left|\int_{\rd}\partial_k\partial_i\partial_j \Phi(y)\left(\varphi_{i,j}(x,y)-\varphi_{i,j}(z,y)\right)dy \right|\\
&=\sum_{i,j=1}^d\left|\int_{|y|\leq\xi}\partial_k\partial_i\partial_j \Phi(y)\left(\varphi_{i,j}(x,y)-\varphi_{i,j}(z,y)\right)dy\right.\\
&\quad \quad \quad+\left.\int_{|y|>\xi}\partial_k\partial_i\partial_j \Phi(y)\left(\varphi_{i,j}(x,y)-\varphi_{i,j}(z,y)\right)dy\right|.
\end{align*}
Now, for $|y|\leq\xi$, we bound 
\[
\sum_{i,j=1}^d\left|\partial_k\partial_i\partial_j \Phi(y)\left(\varphi_{i,j}(x,y)-\varphi_{i,j}(z,y)\right)\right|\leq C_d|y|^{-d-1}\omega_b(|y|)\omega_u(|y|),
\]
and so the first integral can be estimated from above by 
\[
C_d\int_0^\xi\frac{\omega_{b}(\eta)\omega_{u}(\eta)}{\eta^2}\ d\eta.
\]
For the second integral, we first write
\begin{align*}
\varphi_{i,j}(x,y)=&\left[b_i(x-y)-b_i(z-y)+b_i(z)-b_i(x)\right]\left[u_j(x-y)-u_j(x)\right]\\
&+\left[b_i(z-y)-b_i(z)\right]\left[u_j(x-y)-u_j(z-y)+u_j(z)-u_j(x)\right]+\varphi_{i,j}(z,y),
\end{align*}
making
\[
\frac{1}{2}\sum_{i,j=1}^d|\varphi_{i,j}(x,y)-\varphi_{i,j}(z,y)|\leq d\omega_{b}(|x-z|)\omega_{u}(|y|)+d\omega_{b}(|y|)\omega_{u}(|x-z|),
\]
and so the second integral is dominated by
\[
C_d\omega_{b}(\xi)\int_{\xi}^\infty\frac{\omega_{u}(\eta)}{\eta^2}\ d\eta+C_d\omega_{u}(\xi)\int_{\xi}^\infty\frac{\omega_{b}(\eta)}{\eta^2}\ d\eta,
\]
giving us precisely \eqref{modgradp}.
\end{proof}
The following Lemma should not be surprising and is a slight generalization of the argument presented in the Appendix of \cite{KNV2007}, based on the fact that singular integral operators of order zero preserve H\"older continuity.
\begin{lem}\label{sioholder}
Suppose that $K$ is a kernel as in \eqref{defkern} with $\Phi\in C^{0,\rho}(\mathbb{S}^{d-1})$ for some $\rho\in(0,1]$ satisfying the zero average condition \eqref{zeroavg}. Let $\mathcal{N}$ be the singular integral operator given by \eqref{ptwiseN} and let $\mathcal{N}_p$ be its periodization \eqref{defNp} or \eqref{ptwiseNp}. Suppose that $\theta:\rd\rightarrow\mathbb{R}$ has modulus of continuity $\omega$. If $\theta\in L^p(\rd)$ some $p\in(1,\infty)$ (resp. periodic with zero average) then $\mathcal{N}\theta$ (resp. $\mathcal{N}_p\theta$) has modulus of continuity
\begin{equation}\label{modNL}
	\widetilde\omega(\xi):=C_{K,d}\left[\int_0^{3\xi}\frac{\omega(\eta)}{\eta}\ d\eta+\xi^{\rho}\int_{3\xi}^{\infty}\frac{\omega(\eta)}{\eta^{1+\rho}}\ d\eta\right],
\end{equation}
where $C_{K,d}>0$ is some constant that depends only on the kernel $K$ and dimension $d$.
\end{lem}
\begin{proof}
	That $\mathcal{N}\theta$ and $\mathcal{N}_p\theta$ are both well defined under the respective assumptions on $\theta$ is known, see \cite{CZ1954, Stein1970book, SWbook1971}. The proof is pretty much a standard trick in the analysis of singular integral operators, where smoothness of the function is used to control the singularity of the kernel near 0, while we use the smoothness and decay of the kernel in order to control the tail end of the integral. To that extent, we let $(x,y)\in\rd$ be arbitrary, set $\xi:=|x-y|$, $\widetilde x:=(x+y)/2$ and $\bar x:=(x-y)/2$. From \eqref{ptwiseN} and since $\Phi$ satisfies the zero average condition \eqref{zeroavg}, we get that 
	\[
	\mathcal{N}\theta(x)=\int_{|x-z|\leq2\xi}K(x-z)(\theta(z)-\theta(x))\ dz+\int_{|x-z|\geq2\xi}K(x-z)(\theta(z)-\theta(\tilde x))\ dz,
	\]
	from which we can bound
	\begin{equation}\label{bd1}
	\left|\int_{|x-z|\leq2\xi}K(x-z)(\theta(z)-\theta(x))\ dz\right|\leq C_d\|\Phi\|_{L^{\infty}}\int_0^{2\xi} \frac{\omega(\eta)}{\eta}d\eta,
	\end{equation}
	and similarly for $\mathcal{N}\theta(y)$. Therefore, we get that 
	\begin{equation}
	|\mathcal{N}\theta(x)-\mathcal{N}\theta(y)|\leq2C_d\|\Phi\|_{L^{\infty}}\int_0^{2\xi} \frac{\omega(\eta)}{\eta}d\eta+|I(x)-I(y)|,
	\end{equation}
	where
	\[
	I(x):=\int_{|x-z|\geq2\xi}K(x-z)(\theta(z)-\theta(\tilde x))\ dz,\quad I(y):=\int_{|y-z|\geq2\xi}K(y-z)(\theta(z)-\theta(\tilde x))\ dz.
	\]
	Since we have the inclusion 
	\[
	\{z\in\rd:|z-x|\geq 2\xi\}\cup\{z\in\rd:|z-y|\geq 2\xi\}\subset\{z\in\rd:|z-\widetilde x|\geq3\xi/2\},
	\]
	we can bound
	\begin{align*}
	|I(x)-I(y)|\leq& \int_{|z-\widetilde x|\geq3\xi}|K(x-z)-K(y-z)||\theta(z)-\theta(\widetilde x)|\ dz\nonumber\\
	&+\int_{3\xi/2\leq|z-\widetilde x|\leq3\xi}\left|K(x-z)+K(y-z)\right|\left|\theta(z)-\theta(\widetilde x)\right|\ dz\\
	&\leq\int_{|z-\widetilde x|\geq3\xi}|K(x-z)-K(y-z)||\theta(z)-\theta(\widetilde x)|\ dz+ C_d\|\Phi\|_{L^{\infty}}\int_0^{3\xi} \frac{\omega(\eta)}{\eta}d\eta,
	\end{align*}
	where in the last inequality we used $2|x-z|\geq|z-\widetilde{x}|$ (with the same for $y$) along with a similar calculation as in \eqref{bd1}. It is more convenient to invoke the change of variable $t:=\widetilde x-z$ in the first integral above, and hence to obtain \eqref{modNL}, we have to prove 
	\begin{equation}\label{I2}
	\int_{|t|\geq3\xi}|K(t+\bar x)-K(t-\bar x)||\theta(\widetilde x-t)-\theta(\widetilde x)|\ dt\leq C_{K,d}\xi^{\rho}\int_{3\xi}^{\infty}\frac{\omega(\eta)}{\eta^{1+\rho}},
	\end{equation}
	from some constant $C_{K,d}$. Now, from \eqref{defkern}, and the fact that $\Phi\in C^{0,\rho}(\mathbb{S}^{d-1})$, we get 
	\begin{align}\label{diffK}
	|K(t+\bar x)-K(t-\bar x)|=&\left|\frac{\Phi\left(\frac{t+\bar x}{|t+\bar x|}\right)-\Phi\left(\frac{t-\bar x}{|t-\bar x|}\right)}{|t+\bar x|^d}+\Phi\left(\frac{t-\bar x}{|t-\bar x|}\right)\left[\frac{1}{|t+\bar x|^d}-\frac{1}{|t-\bar x|^d}\right]\right|\nonumber\\
	&\leq \frac{[\Phi]_{C^{0,\rho}}}{|t+\bar x|^{d}}\left|\frac{t+\bar x}{|t+\bar x|}-\frac{t-\bar x}{|t-\bar x|}\right|^{\rho}+\|\Phi\|_{L^{\infty}}\left|\frac{1}{|t+\bar x|^d}-\frac{1}{|t-\bar x|^d}\right|.
	\end{align}
	Let us now note that since $\left||t-\bar x|-|t+\bar x|\right|\leq 2|\bar x|=\xi$, we must have  
	\begin{equation}\label{bd2}
	\left|\frac{t+\bar x}{|t+\bar x|}-\frac{t-\bar x}{|t-\bar x|}\right|\leq \frac{|t|\xi}{|t+\bar x||t-\bar x|}+\frac{\xi}{2|t+\bar x|}+\frac{\xi}{2|t-\bar x|}\leq 6\frac{\xi}{|t|},
	\end{equation}
	where in the last inequality we also used $|t\pm\bar x|\geq |t|/2$ for $|t|\geq 3\xi$. As for the second difference in \eqref{diffK}, we first note that $\left|a^d-b^d\right|\leq d\left|a-b\right|\left[a^{d-1}+b^{d-1}\right]$, whenever $a,b\geq0$ to get 
	\begin{equation}\label{bd3}
		\frac{\left||t-\bar x|^d-|t+\bar x|^d\right|}{|t-\bar x|^d|t+\bar x|^d}\leq d\xi\left[\frac{1}{|t+\bar x||t-\bar x|^d}+\frac{1}{|t-\bar x||t+\bar x|^d}\right].
	\end{equation}
	Finally, plugging \eqref{bd2}-\eqref{bd3} into \eqref{diffK}, using $|\theta(\widetilde x-t)-\theta(\widetilde x)|\leq \omega(|t|)$ and utilizing one more time $|t\pm\bar x|\geq |t|/2$ for $|t|\geq 3\xi$ we get
	\begin{equation}\label{bd4}
		\int_{|t|\geq3\xi}|K(t+\bar x)-K(t-\bar x)||\theta(\widetilde x-t)-\theta(\widetilde x)|\ dt\leq C_{K,d}\xi^{\rho}\int_{3\xi}^{\infty}\frac{\omega(\eta)}{\eta^{1+\rho}}\ d\eta+C_{K,d}\xi \int_{3\xi}^{\infty}\frac{\omega(\eta)}{\eta^{2}}\ d\eta,
	\end{equation}
	from which \eqref{modNL} would follow immediately after noting that $\eta^2\geq (3\xi)^{1-\rho}\eta^{1+\rho}$ in the second integral above.
	
	Now for the case of $\mathcal{N}_p\theta$ where $\theta$ is periodic with zero average. For a sufficiently large $M>>1$, consider the function defined by 
	\[
	\varphi_M(x):=\int_{|z|\leq M}K(z)\theta(x-z)\ dz=\int_{|x-z|\leq M}K(x-z)\theta(z)\ dz.
	\]
	We can repeat the previous argument, word by word, to conclude that 
	\[
	|\varphi_M(x)-\varphi_M(y)|\leq C_{K,d}\int_0^{3\xi}\frac{\omega(\eta)}{\eta}\ d\eta+C_{K,d}\xi^{\rho}\int_{3\xi}^{\infty}\frac{\omega(\eta)}{\eta^{1+\rho}}\ d\eta.
	\]
	On the other hand, without loss of generality, we may assume the Fourier series of $\theta$ converges absolutely and uniformly on compact sets. Using the Fourier multiplier definition of $\mathcal{N}_p$ from \eqref{defNp}, it follows that 
	\[
	|\varphi_M(x)-\mathcal{N}_p\theta(x)|\lesssim_{d,L} \sum_{\substack{m\in\mathbb{Z}^d\\m\neq0}}|\widehat \theta(m)|\left|\int_{|z|\leq M}K(z)e^{-2i\pi m\cdot z}\ dz-\widehat K(m)\right|=\sum_{\substack{m\in\mathbb{Z}^d\\m\neq0}}|\widehat \theta(m)|\left|\int_{|z|> M}K(z)e^{-2i\pi m\cdot z}\ dz\right|.
	\]
	According to the proof of Theorem 3 in \cite[Chapter 2]{Stein1970book}, the integral is bounded uniformly in $M$ and $m\in\mathbb{Z}^d \backslash\{0\}$ and goes to 0 as $M$ goes to infinity. This allows us to conclude our proof.
\end{proof}

\section{Proofs of Main Results}\label{secpfmainres}
In \S\ref{strat}, following \cite{Ibdah2020b, Kiselev2011,KNS2008, KNV2007}, we outline the strategy of the proof of all our results. Theorem \ref{thm1} is then proven in subsection \S\ref{pfthm1}, Theorem \ref{thm2} is proven in \S\ref{pfthm2}, and we conclude the section with the proof of Theorems \ref{thm3} and \ref{thm4} in \S\ref{pfthm3}, \S\ref{pfthm4a} and \S\ref{pfthm4b}. 
\subsection{The Breakthrough Scenario}\label{strat}
Let $u$ be a solution to any of \eqref{maineq}, \eqref{maineq2}, \eqref{classicalddintro} or \eqref{maineq3intro} with regularity $C_t^1C_x^2$. Recall that in the case of \eqref{maineq} and \eqref{maineq2}, if 
\begin{equation}\label{defmaxT}
T_*:=\sup\left\{t>0:\|\nabla u(s,\cdot)\|_{L^{\infty}}<\infty, \forall s\in[0,t]\right\},
\end{equation}
then $u\in C_t^1C_x^2([0,T_*)\times\rd)$ (and in fact, $u(t,\cdot)\in C^{\infty}(\rd)$ for $t\in(0,T_*)$). Let $T>0$ be arbitrary, and suppose $\Omega$ is a time-dependent strong modulus of continuity on $[0,T]$ (according to Definition \ref{timdepmoddef}).  Furthermore, suppose $u_0$ strictly obeys $\Omega(0,\cdot)$ as in Definition \ref{defobeymod}. Let us define 
\begin{equation}\label{deftau1}
	\tau:=\sup\left\{t\in(0,T]:|u(s,x)-u(s,y)|<\Omega(s,|x-y|)\ \forall s\in[0,t],x\neq y,\right\},\\
\end{equation}
and suppose for the moment that $\tau>0$. Clearly we must have $\tau<T_*$ regardless of how $T_*$ compares to $T$, otherwise since the solution is $C^2$ in space on $[0,T_*)$ and obeys the modulus of continuity $\Omega(t,\cdot)$ for every $t\in[0,T_*)$, Lemma \ref{gradbd} is applicable on $[0,T_*)$. Since $\partial_\xi\Omega(t,0)<\infty$ for $t\in[0,T]$ and $\Omega$ satisfies one of the time-regularity conditions listed in Definition \ref{timdepmoddef}, this would lead to the uniform bound 
\[
\sup_{t\in[0,T_*)}\|\nabla u(t,\cdot)\|_{L^{\infty}}\leq \max_{t\in[0,T]}\partial_\xi\Omega(t,0)<\infty,
\]
clearly contradicting the definition of $T_*$. The idea behind the proofs of Theorems \ref{thm1}, \ref{thm2}, \ref{thm3} and \ref{thm4} is to construct an $\Omega$ such that we can guarantee $\tau=T$, and hence $T_*>T$ (in the case of \eqref{maineq} and \eqref{maineq2}). As a byproduct, we obtain the bounds \eqref{mainbd1}, \eqref{mainbd2}, \eqref{supcritdd}, \eqref{crtibdnsethm} and \eqref{supercritbdnsethm}. This will be done by contradiction, that is from now on, we assume $T_*\leq T$ (when considering Theorems \ref{thm1} and \ref{thm2}), and so by the previous discussion, we must have $\tau<T_*\leq T$. Point is, the solution has not exhibited blowup at time $\tau$ and still obeys $\Omega$, albeit not necessarily strictly. More so, by short time existence, $u$ is still smooth for a short time beyond $\tau$. In the case of Theorems \ref{thm3} and \ref{thm4}, we assumed $u$ is at the required regularity level on $[0,T]$ (with $T$ now being dictated by the given drift velocity $b$) to apply Lemma \ref{gradbd}, and so what we do is assume $\tau<T$, and construct a time-dependent strong modulus of continuity $\Omega$ on $[0,T]$ depending on the drift velocity $b$. We start by identifying the only possible scenario at time $\tau$, one that is depicted by the solution vector-field violating its strict modulus of continuity away from the diagonal $x=y$, the so called ``breakthrough scenario''. See also \cite[Lemma~2.3]{Kiselev2011} and \cite[Proposition~4.1]{Ibdah2020b}. We emphasize that when working in the whole space with vanishing properties at spatial infinity, we need to make further technical assumptions regarding the growth of $\Omega(t,\cdot)$ with respect to $\|u(t,\cdot)\|_{L^{\infty}}$, see Proposition \ref{breakthroughwholespace} and the remarks following it.
 \begin{prop}\label{breakthrough}
 Let $u\in C([0,T_*)\times\rd)\cap C([0,T_*);W^{1,\infty}(\rd))$ be a vector field such that $u(t,\cdot)\in C_{per}^2(\rd)$ for every $t\in[0,T_*)$. Let $T>0$ be arbitrary, and suppose that $\Omega$ is a time-dependent strong modulus of continuity on $[0,T]$ such that $u(0,\cdot)$ strictly obeys $\Omega(0,\cdot)$. Let us define
 \begin{equation}\label{deftau}
	\tau:=\sup\left\{t\in(0,T]:|u(s,x)-u(s,y)|<\Omega(s,|x-y|)\ \forall s\in[0,t],x\neq y\right\}.\\
\end{equation}
It follows that $\tau$ is positive and if $\tau<\min\{T,T_*\}$, then we must have 
\[
|u(\tau,x_0)-u(\tau,y_0)|=\Omega(\tau,|x_0-y_0|),
\]
for some $x_0\neq y_0$.
\end{prop}
\begin{proof}
	Here, we will assume condition 1 of Definition \ref{timdepmoddef}, and we refer the reader to \cite[Lemma~2.3]{Kiselev2011} for the other case. According to Lemma \ref{gradbd} (since $C^2_{per}(\rd)\subset W^{2,\infty}(\rd)$), we have $\|\nabla u(0,\cdot)\|_{L^{\infty}}<\partial_\xi\Omega(0,0)$, and by continuity of the function $\left\|\nabla u(t,\cdot)\right\|_{L^{\infty}}$, this remains true for $t\in[0,\epsilon_0]$, some $\epsilon_0>0$. Set
	\[
	M_0:=\max_{t\in[0,\epsilon_0]}\left\|\nabla u(t,\cdot)\right\|_{L^{\infty}}<\partial_\xi\Omega(0,0),
	\]
	and consider the function
	\[
	h(\xi):=M_0-\frac{\Omega(0,\xi)}{\xi},\quad \xi>0.
	\]
	Clearly, $h(\xi)<0$ for $\xi\in(0,\delta)$, some $\delta>0$. It follows that whenever $t\in[0,\epsilon_0]$ and $|x-y|\in(0,\delta)$, since $\Omega(\cdot,\xi)$ is nondecreasing as a function of time for each fixed $\xi\geq0$, we must have
	\[
	|u(t,x)-u(t,y)|\leq M_0|x-y|<\Omega(t,|x-y|).
	\]
	We now define
	\[
	\mathcal{A}:=\left\{(x,y)\in[0,L]^d\times\mathbb{R}^d:|x-y|\in[\delta,2L\sqrt{d}]\right\},
	\]
	where $L>0$ is the period of $\theta(t,\cdot)$, and note that since the set $[0,\epsilon_0]\times \mathcal{A}$ is compact, the function
	\[
	R(t,x,y):=|u(t,x)-u(t,y)|-\Omega(t,|x-y|),
	\]
	is uniformly continuous on it, and as $R(0,x,y)<0$, the same must be true on $[0,\epsilon]\times \mathcal{A}$, some $\epsilon\in(0,\epsilon_0]$. As $u(t,\cdot)$ is $L$ periodic and $\Omega(t,\cdot)$ is non-decreasing, we must have $\tau\geq\epsilon>0$.
	
	The second part of the proposition follows by similar arguments. Suppose for the sake of contradiction that $\tau<\min\{T,T_*\}$ and that $u(\tau,\cdot)$ strictly obeys $\Omega(\tau,\cdot)$. Since $u$ has sufficient smoothness in a small neighborhood of $\tau$, we can repeat the previous argument word by word to extend preservation of the strict modulus of continuity for a small time beyond $\tau$, which contradicts the definition of $\tau$.
\end{proof}
\begin{prop}\label{breakthroughwholespace}
 Let $u\in C([0,T_*)\times\rd)\cap C([0,T_*);W^{1,\infty}(\rd))$ be a vector field such that $u(t,\cdot)\in C^2(\rd)\cap W^{2,\infty}(\rd)$ for every $t\in[0,T_*)$ and for which 
 \begin{equation}\label{decay}
 \lim_{|x|\rightarrow\infty}|\nabla u(t,x)|=0,\quad \forall t\in[0,T_*).
 \end{equation}
 Let $T>0$ be arbitrary, and suppose that $\Omega$ is a time-dependent strong modulus of continuity on $[0,T]$ such that $u(0,\cdot)$ strictly obeys $\Omega(0,\cdot)$. Assume further that for every compact subinterval $\mathcal{K}=[t_1,t_2]\subset[0,T]$, there exists some $K=K(\mathcal{K})\geq1$ such that $\Omega(t,\xi)\geq 3\|u(t,\cdot)\|_{L^{\infty}}$ whenever $(t,\xi)\in \mathcal{K}\times[K,\infty)$ and suppose that $\Omega$ satisfies condition 1 of Definition \ref{timdepmoddef}. Let $\tau$ be as defined in \eqref{deftau}. It follows that $\tau$ is positive and if $\tau<\min\{T,T_*\}$, then we must have 
\[
|u(\tau,x_0)-u(\tau,y_0)|=\Omega(\tau,|x_0-y_0|),
\]
for some $x_0\neq y_0$.
\end{prop}
\begin{proof}
As in the proof of Proposition \ref{breakthrough}, we have 
\[
	|u(t,x)-u(t,y)|<\Omega(t,|x-y|),
\]
whenever $|x-y|\in (0,\delta]$ and $t\in[0,\epsilon_0]$ some $\delta,\epsilon_0>0$, since periodicity was not used in this part of the proof. Our assumption on $\Omega$ guarantees the existence of a sufficiently large $K_0=K_0(\epsilon_0)$ such that $\Omega(t,\xi)\geq 3\|u(t,\cdot)\|_{L^{\infty}}$ whenever $(t,\xi)\in[0,\epsilon_0]\times[K_0,\infty)$. It follows that for $|x-y|\in[K_0,\infty)$ and $t\in[0,\epsilon_0]$, we must have 
\[
|u(t,x)-u(t,y)|<\Omega(t,|x-y|),
\]
and so it remains to handle the case $|x-y|\in[\delta,K_0]$. To that extent, we first note that from \eqref{decay}, for any given $\mu>0$, we can find a large enough $K_1$ such that if we define the set $\mathcal{A}:=\{x\in\rd:|x|\geq K_1\}$, then $\|\nabla u(0,\cdot)\|_{L^{\infty}(\mathcal{A})}<\mu$. Since we assumed that $u\in C([0,T_*);W^{1,\infty}(\rd))$, the function $\gamma(t):=\|\nabla\theta(t,\cdot)\|_{L^{\infty}(\mathcal{A})}$ is continuous as well. Hence, we must have $\|\nabla u(t,\cdot)\|_{L^{\infty}(\mathcal{A})}<\mu$ for  $t\in[0,\epsilon_1]$, some $\epsilon_1\in(0,\epsilon_0]$. Thus, we now chose $K_1$ large enough such that 
\begin{equation}\label{small}
|\nabla u(t,x)|<K_0^{-1}\Omega(0,\delta),\quad \forall|x|\geq K_1, \quad t\in[0,\epsilon_1].
\end{equation}
Next, we split the set $\mathcal{B}:=\left\{(x,y)\in\rd\times\rd:|x-y|\in[\delta,K_0]\right\}$, into $\mathcal{B}_1\cup\mathcal{B}_2$ where $\mathcal{B}_2$ is the complement of $\mathcal{B}_1:=\left\{(x,y)\in\mathcal{B}:\min\{|x|,|y|\}>K_0+K_1\right\}$.  By the mean value theorem, whenever $(t,x,y)\in[0,\epsilon_1]\times\mathcal{B}_1$, we must have, for some $\sigma\in(0,1)$
\[
|u(t,x)-u(t,y)|\leq |x-y||\nabla u(t,\sigma(x-y)+y)|< \Omega(0,\delta)\leq \Omega(t,|x-y|),
\]
where we used $|x-y|\in[\delta,K_0]$ and \eqref{small} in the second inequality, while we used the fact that $\Omega$ is nondecreasing in both variables in the third one. Finally, since $\mathcal{B}_2$ is compact, we can certainly obtain a small enough $\epsilon>0$ such that the strict modulus of continuity is obeyed for $(t,x,y)\in[0,\epsilon]\times\mathcal{B}_2$. We may repeat the exact same argument to show that at time $\tau$ we must have 
\[
|u(\tau,x_0)-u(\tau,y_0)|=\Omega(\tau,|x_0-y_0|),
\]
for some $x_0\neq y_0$.
\end{proof}
\begin{rmk}
The assumption that $\Omega(\cdot,\xi)$ is non-decreasing in time could also be replaced with continuity of $\partial_\xi \Omega(\cdot,0)$, but the argument will be slightly more involved. We prefer to keep things simple, since in any case the moduli of continuity that we will be working with are non-decreasing in time.
\end{rmk}
\begin{rmk}
The lower bound on $\Omega$ is satisfied if $\Omega(t,\cdot)$ is unbounded in space. Alternatively, if one has a-priori bounds on $\|u(t,\cdot)\|_{L^{\infty}}$, the one can easily satisfy this condition. It could also be dropped if one can guarantee that the solution $u$ achieves its maximum and minimum at some $x_0$ and $x_1$ for each $t$, as this would imply that $\text{osc}\ u(t,\cdot)<\|\Omega(t,\cdot)\|_{L^{\infty}}$.
\end{rmk}
\subsection{Proof of Theorem \ref{thm1}}\label{pfthm1}
Let us now assume that Hypothesis \ref{hypoth} is true for some $T>0$, and let us suppose further that a given divergence-free vector field $u_0\in \mathbb{V}^k$ ($k>d/2+2$) strictly obeys $\Omega(0,\cdot)$. Let $T_*$ and $\tau$ be as defined in \eqref{defmaxT} and \eqref{deftau1} respectively. Notice that by Proposition \ref{breakthrough}, we must have $\tau>0$, and that for every $t\in[0,\tau]$, $u$ has $\Omega(t,\cdot)$ as a modulus of continuity, and in fact strictly obeys $\Omega(t,\cdot)$ for $t<\tau$. Now assume we are at the breakthrough scenario depicted by Proposition \ref{breakthrough}. That is, we suppose $\tau<T_*\leq T$, and so for some $x_0\neq y_0$, velocity vector-field $u$ violates its strict modulus of continuity. The analysis will be easier to carry out if we work with scalar functions and if $x_0-y_0=\xi e_1$, some $\xi>0$. Hence, we start by choosing a rotational matrix $\mathcal{R}$ such that $\mathcal{R}(x_0-y_0)=\xi e_1$, where $\xi:=|x_0-y_0|$ and $e_1$ is the standard unit vector in the first coordinate. Let us define $\tilde x_0:=\mathcal{R}x_0$, $\tilde y_0:=\mathcal{R}y_0$, $\widetilde u(t,x):=\mathcal{R} u(t,\mathcal{R}^{-1}x)$, and $\widetilde p(t,x):=p(t,\mathcal{R}^{-1}x)$. It follows that there exists some unit vector $e\in\mathbb{S}^{d-1}$, possibly depending on the time $\tau$, $x_0$ and $y_0$, but is otherwise a constant on $[0,\tau]\times\rd$ such that 
\[
\left[\widetilde u(\tau,\tilde x_0)-\widetilde u(\tau,\tilde y_0)\right]\cdot e=\Omega(\tau,\xi).
\]
Since moduli of continuity are not sensitive to distance preserving maps, if we define the scalar function $\theta(t,x):=e\cdot\widetilde u(t,x)$ and study its evolution on $[0,\tau+\epsilon]$ we get that
\begin{align}
	&\partial_t\theta(t,x)-\Delta\theta(t,x)=(\widetilde u\cdot\nabla)\theta(t,x)+e\cdot\nabla \widetilde p(t,x),&&\forall (t,x)\in[0,\tau+\epsilon]\times\rd,\label{evoltheta}\\
	&|\widetilde u(t,x)-\widetilde u(t,y)|\leq \Omega(t,|x-y|), &&\forall (t,x,y)\in [0,\tau]\times\rd\times\rd,\label{tmod2}\\
	&|\theta(t,x)-\theta(t,y)|<\Omega(t,|x-y|), &&\forall t\in [0,\tau),\quad x\neq y,\label{modtheta}\\
	&\theta(\tau,\tilde x_0)-\theta(\tau,\tilde y_0)=\Omega(\tau,\xi), &&(\tilde x_0-\tilde y_0)=\xi e_1,\label{violation}
\end{align}
We now consider the function
\[
\Gamma(t):=\theta(t,\tilde x_0)-\theta(t,\tilde y_0)-\Omega(t,\xi),
\]
which is piecewise $C^1$ on the interval $[0,\tau+\epsilon]$. Notice that from \eqref{modtheta} and \eqref{violation}, we have $\Gamma(t)<0$ for $t\in[0,\tau)$, so that $\Gamma$ attains its maximum at $\tau$. Thus, in order to rule out the breakthrough scenario \eqref{violation}, it suffices to construct $\Omega$ such that $\Gamma'(\tau)<0$, with $\partial_t\Omega(\tau,\xi)$ being understood as the left time derivative of $\Omega$ evaluated at the point $(\tau,\xi)$. Using \eqref{evoltheta} (which holds true in the pointwise sense for $t\in[0,\tau+\epsilon]$), we see that 
\begin{align}\label{gprime}
\Gamma'(\tau)=&\Delta\theta(\tau,\tilde x_0)-\Delta\theta(\tau,\tilde y_0)-\partial_t\Omega(\tau^{-},\xi)\nonumber\\
&+\left(\widetilde u\cdot\nabla\right)\theta(\tau,\tilde x_0)-\left(\widetilde u\cdot\nabla\right)\theta(\tau,\tilde y_0)+e\cdot\nabla \widetilde p(\tau, x_0)-e\cdot\nabla \widetilde p(\tau,y_0).
\end{align}
Let us now estimate each term in \eqref{gprime}. We extract local dissipation from the Laplacian via \eqref{lap} to get
\begin{equation}\label{locdissp}
\Delta\theta(\tau,\tilde x_0)-\Delta\theta(\tau,\tilde y_0)\leq 4\partial^2_{\xi}\Omega(\tau,\xi^-).
\end{equation}
Next, to handle the nonlinearity, we use \eqref{derv} to get that all derivatives of $\theta$ evaluated at $\tilde x_0$ or $\tilde y_0$ in directions other than $e_1$ vanish. On the other hand, following Remark \ref{imprmk}, we see that $\xi$ cannot be a point of jump discontinuity of $\partial_\xi\Omega$. Thus, the derivative of $\theta$ (evaluated at $\tilde x_0$ or $\tilde y_0$) in the $e_1$ direction is precisely $\partial_\xi\Omega$ and so the nonlinear term reduces to
\begin{equation}\label{nonlin}
\begin{aligned}
(\widetilde u\cdot\nabla)\theta(\tau,\tilde x_0)-(u\cdot\nabla)\theta(\tau,\tilde y_0)&=\widetilde u_1(\tau,\tilde x_0)\partial_1\theta(\tau,\tilde x_0)-\widetilde u_1(\tau,\tilde y_0)\partial_1\theta(\tau,\tilde y_0)\\
&=\left(\widetilde u_1(\tau,\tilde x_0)-\widetilde u_1(\tau,\tilde y_0)\right)\partial_\xi\Omega(\tau,\xi)\leq \Omega(\tau,\xi)\partial_\xi\Omega(\tau,\xi),
\end{aligned}
\end{equation}
where in the last inequality we used \eqref{tmod2}. Finally, the gradient of the pressure is estimated from above via Lemma \ref{pressest} (with $b=u$) to get, since $|e|=1$ and $\mathcal{R}$ is an orthogonal matrix,
\begin{equation}\label{press}
e\cdot\nabla \widetilde p(\tau, x_0)-e\cdot\nabla \widetilde p(\tau,y_0)\leq C_d\int_{0}^\xi\left(\frac{\Omega(\tau,\eta)}{\eta}\right)^2d\eta+C_d\Omega(\tau,\xi)\int_{\xi}^{\infty}\frac{\Omega(\tau,\eta)}{\eta^2}d\eta.  
\end{equation}
Plugging estimates \eqref{locdissp}-\eqref{press} into \eqref{gprime} we obtain
\begin{equation}\label{ineqgprime}
\Gamma'(\tau)\leq 4\partial^2_{\xi}\Omega(\tau,\xi)-\partial_t\Omega(\tau,\xi)+\Omega(\tau,\xi)\partial_\xi\Omega(\tau,\xi)+C_d\int_{0}^\xi\left(\frac{\Omega(\tau,\eta)}{\eta}\right)^2d\eta+C_d\Omega(\tau,\xi)\int_{\xi}^{\infty}\frac{\Omega(\tau,\eta)}{\eta^2}d\eta.
\end{equation}
Therefore, in order to rule out the only possible ``breakthrough'' scenario at time $\tau\in(0,T]$ as depicted by Proposition \ref{breakthrough}, we need $\Omega$ to be a time-dependent strong modulus of continuity on $[0,T]$ that makes the right-hand side of \eqref{gprime} negative. That is, if we let 
\begin{equation}\label{contnonloc}
I(t,\xi):=2C_d\left[\int_{0}^\xi\frac{\Omega^2(t,\eta)}{\eta^2}d\eta+\Omega(t,\xi)\int_{\xi}^{\infty}\frac{\Omega(t,\eta)}{\eta^2}d\eta\right],
\end{equation}
then we seek an $\Omega$ that satisfies
\begin{equation}\label{modevol}
\partial_t\Omega(t,\xi)-4\partial^2_{\xi}\Omega(t,\xi)-\Omega(t,\xi)\partial_\xi\Omega(t,\xi)-I(t,\xi)\geq0,\quad \forall (t,\xi)\in(0,T]\times(0,\infty),
\end{equation}
together with the boundary conditions
\begin{equation}\label{BC}
\begin{aligned}
&\Omega(t,0)=0, &&\forall t\in[0,T],\\
&0<\partial_{\xi}\Omega(t,0)<\infty, &&\forall t\in[0,T],\\
&\partial^2_{\xi}\Omega(t,0^+)=-\infty,&&\forall t\in[0,T],\\
\end{aligned}
\end{equation}
along with the positivity and non-decreasing conditions
\begin{equation}\label{pos}
\Omega(t,\xi)>0, \ \partial_{\xi}\Omega(t,\xi)\geq0,\ \forall(t,\xi)\in[0,T]\times\mathbb{R}^+,
\end{equation}
all the while ensuring that $\partial_\xi\Omega(t,\xi^+)\leq\partial_\xi\Omega(t,\xi^-)$ and for which either $\Omega(\cdot,\xi)$ is non-decreasing or $\partial_\xi\Omega(\cdot,0)$ is continuous  (or both) as functions of time. This is precisely the content of Hypothesis \ref{hypoth}, and so if it is true and if in addition
\begin{equation}\label{ICmod}
|u_0(x)-u_0(y)|<\Omega(0,|x-y|),\ \forall  x\neq y,
\end{equation}
then $\tau=T$ and hence $T_*>T$, meaning that the solution velocity vector-field $u$ emanating from such initial data $u_0$ must strictly obey $\Omega$ all the way up to time $T$, and so is smooth on $[0,T]$ and satisfies 
\[
\|\nabla u(t,\cdot)\|_{L^{\infty}}<\partial_\xi\Omega(t,0),\quad \forall t\in[0,T].
\]
Of course, here \eqref{ICmod} implies a smallness assumption $\|\nabla u_0\|_{L^{\infty}}\leq \partial_\xi\Omega(0,0)$. For arbitrary initial data, this may not be true, but can be remedied by exploiting the scale invariance of \eqref{maineq} and Lemma \ref{buildmod}. Indeed, let $u_0\in\mathbb{V}^k$, with $k>d/2+2$, be arbitrary large and let $(u,p)$ be the unique (short-time) solution to \eqref{maineq}. By Lemma \ref{buildmod}, we can chose a sufficiently large $B\geq B_0$ where
\[
B_0:=\frac{2}{\Omega(0,\delta)}\|u_0\|_{L^{\infty}}+\left(\frac{\delta}{\Omega(0,\delta)}\|\nabla u_0\|_{L^{\infty}}\right)^{1/2},
\] 
so that $u_0$ strictly obeys $B\Omega(0,B\xi)$. This means that the solution $(v,q)$ to \eqref{maineq} with initial data $B^{-1}u_0(B^{-1}x)$ is smooth on $[0,T]$ with $v(t,\cdot)$ strictly obeying $\Omega(t,\cdot)$ for every $t\in[0,T]$ and so we have the bound 
\[
\|\nabla v(t,\cdot)\|_{L^{\infty}}< \partial_\xi\Omega(t,0)<\infty,\quad \forall t\in[0,T].
\]
But then by uniqueness of strong solutions, we must have $u(t,x)=Bv(B^2t,Bx)$ and $p(t,x)=B^2q(B^2t,Bx)$ for $t\in[0,T/B^2]$, and so the proof of Theorem \ref{thm1} is complete. The same analysis applies if we are working in the whole space, except we have to take into account the extra requirements on $\Omega$ as explained in Proposition \ref{breakthroughwholespace} and the remarks following it.

A natural question that may pop into the reader's mind is this: do we need $\Omega$ to depend on time? The answer is yes, due to the presence of the integrals, and the fact that we are interested in solutions that are moduli of continuity, meaning they should be non-decreasing in space. To see why, let us ignore the pressure term for a moment (and so we end up with the viscous Burgers equation). The inequality would read
\[
\partial_t\Omega(t,\xi)-4\partial^2_{\xi}\Omega(t,\xi)-\Omega(t,\xi)\partial_\xi\Omega(t,\xi)\geq0.
\]
If we drop time dependence, we end up with the stationary, viscous Burgers ``inequality'': $4\omega''(\xi)+\omega\omega'(\xi)\leq0$. This has a solution $\omega(\xi)=\tanh(\xi)$, and so to satisfy the various requirements in Definition \ref{defmod}, one may define
\begin{equation}\label{statmodpfthm1}
\omega(\sigma):=
\begin{cases}
	2\sigma-\sigma^{3/2}, &\sigma\in[0,\delta_0],\\
	\tanh\left((\sigma-\delta_0)+\mu_0\right),&\sigma>\delta_0,
\end{cases}
\end{equation}
where $\delta_0\in(0,1]$ is a small parameter, and $\mu_0\in(0,1]$ is chosen to ensure continuity at $\delta_0$. This indeed solves $4\omega''+\omega'\omega\leq0$, and is a modulus of continuity provided $\delta_0$ and $\mu_0$ are appropriately chosen. Further, we may rescale it according to Lemma \ref{buildmod} to allow for arbitrary large initial data in a manner that  respects  the scale invariance of the viscous Burgers equation: $\Omega(\xi):=\lambda_0\omega(\lambda_0\xi)$, where $\lambda_0\geq1$ would depend on the size of the given initial data. So for the Burgers equation, we do not need to allow any time-dependence, and we end up with the bound $\|\nabla u(t,\cdot)\|_{L^{\infty}}\leq 2\lambda_0^2$, where $\lambda_0$ depends only on the size of the initial data, but not on the dimension $d$ nor on the size of the period. 

If we consider the perturbation $4\omega''+\omega'\omega+\epsilon\leq0$, where $\epsilon>0$ is small, it would lead us to solving an Airy-type equation (after applying the Cole-Hopf transformation). This is bad news, as such solutions oscillate, and thus are inadequate as moduli of continuity. In some sense, the viscous term is ``saturated'' by the nonlinearity. Hence, if there is any hope for this to work in the case of NSE, one has to rely on the full parabolic operator $\partial_t-\Delta$, not just the elliptic part, $-\Delta$. A natural way of giving some power to the time derivative is to ``dynamically rescale'' $\omega$ as defined in \eqref{statmodpfthm1} in a matter that respects the scale-invariance of Burgers: $\Omega(t,\xi):=\lambda(t)\omega(\xi\lambda(t))$. That way the balance between viscosity and advective nonlinearity is preserved, yet we have some power given to  $\partial_t\Omega$. Now, when  $\xi\lambda\lesssim 1$, the viscous  term will dominate both the advective term and the pressure part, due to the condition $\omega''(0^+)=-\infty$; the difficulty is when $\xi\lambda\gtrsim1$. In this region, the reader is invited to verify that roughly we would need $\lambda'(t)=C_d\lambda^{3}(t)$: we have $\lambda^2$ coming from the product $\Omega^2$, and the third factor comes from the non-integrable singularity. To arrive at this, the interested reader could perform similar analysis to the ones done in \S\ref{pfthm2} below. Regardless of the size of $\lambda_0$, this would always have a solution for short-time. 

Those heuristics are what motivated Theorem \ref{thm2}: we are considering a linear non-local term (which decreases the factor $\lambda^2$ to $\lambda$), and the order of the operator is zero rather than one (decreasing a factor of $\lambda$ to $\log(\lambda)$). This leads to $\lambda'\approx\lambda\log(\lambda)$, which has a global solution. 
\subsection{Proof of Theorem \ref{thm2}}\label{pfthm2}
Deriving the evolution inequality for the modulus of continuity in this case follows essentially in exactly the same way as that done for the NSE. We again assume we are at the breakthrough scenario depicted by Proposition \ref{breakthrough}. That is, we suppose $\tau<T_*$, and so the solution vector-field violates its strict modulus of continuity, giving us a unit vector $e$ and some fixed $x_0\neq y_0$ such that
\begin{equation}\label{brkthru2}
e\cdot \left[u(\tau,x_0)-u_m(\tau,y_0)\right]=\Omega(\tau,|x_0-y_0|).
\end{equation}
Moreover, since $u$ is $L-$periodic in space and $\Omega$ is nondecreasing, we can assume that $|x_0-y_0|\in(0,2L\sqrt{d}]$, for if we can construct an $\Omega$ such that \eqref{brkthru2} is not possible for any $|x-y|\in(0,2L\sqrt{d}]$, then we guarantee that the same thing can be deduced for any $(x,y)\in\rd\times\rd$. Thus, we repeat the exact same argument used in \S\ref{pfthm1}, except we use Lemma \ref{sioholder} instead of \ref{pressest}, to define 
\begin{equation*}\label{contnonloc2}
I(t,\xi):=C_{K,d}\left[\int_{0}^{3\xi}\frac{\Omega(t,\eta)}{\eta}d\eta+\xi^{\rho}\int_{3\xi}^{\infty}\frac{\Omega(t,\eta)}{\eta^{1+\rho}}d\eta\right]
\end{equation*}
and seek an $\Omega$ such that 
\begin{align}
&\partial_t\Omega(t,\xi)-4\partial^2_{\xi}\Omega(t,\xi)-\Omega(t,\xi)\partial_\xi\Omega(t,\xi)-I(t,\xi)>0, &&\forall (t,\xi)\in(0,\infty)\times(0,2L\sqrt{d}],\label{modevol2}\\
&|u_0(x)-u_0(y)|<\Omega(0,|x-y|),&&\forall x\neq y,\label{IC2}
\end{align}
all the while making sure it satisfies Definition \ref{timdepmoddef}:
\begin{align}
&\Omega(t,0)=0, &&\forall t\geq0\label{BCmodthm2},\\
&\Omega(t,\xi_2)\geq\Omega(t,\xi_1), &&\forall \xi_2\geq\xi_1, (t,\xi_1,\xi_2)\in[0,\infty)^3,\label{nondecmodthm2}\\
&\partial_\xi\Omega(t,\xi^+)\leq\partial_\xi\Omega(t,\xi^-),&&\forall (t,\xi)\in(0,\infty)\times(0,\infty),\label{weakconcvthm2}\\
&\partial_t\Omega(t,\xi)\geq0, &&\forall (t,\xi)\in(0,\infty)\times(0,\infty).\label{nondectmodthm2}\\
&\lim_{\xi\rightarrow0^+}\partial^2_\xi\Omega(t,\xi)=-\infty, &&\forall t\geq0,\label{conczerothm2}
\end{align}
Actually, condition \eqref{nondectmodthm2} needs to hold only when $\xi$ is a possible ``breakthrough'' point; in particular, following Remark \ref{imprmk}, we do not need to worry about it when $\xi$ is a point of jump discontinuity of $\partial_\xi\Omega(t,\cdot)$. 

In what follows, $C_0\geq1$ will always denote an absolute universal constant depending on the kernel $K$, the dimension $d$ and the period $L$ and whose value may change from line to line. Inspired by the discussion towards the end of \ref{pfthm1}, we will ``dynamically rescale'' a stationary solution to the one dimensional viscous Burgers equation. We start by defining the stationary modulus of continuity
\begin{equation}\label{statmodpfthm2}
\omega(\sigma):=
\begin{cases}
	2\sigma-\sigma^{3/2}, &\sigma\in[0,\delta],\\
	\tanh\left((\sigma-\delta)+\mu\right),&\sigma>\delta,
\end{cases}
\end{equation}
where $\delta\in(0,1]$ is a small parameter to be determined later, and $\mu\in(0,1]$ is chosen to ensure continuity at $\delta$. It is fairly straightforward to check that $\omega$ is a strong modulus of continuity according to Definition \ref{defmod}: it is positive, piecewise smooth, nondecreasing and satisfies $\omega(0)=0$, $\omega'(0)=2$ and $\omega''(\sigma)=-3\sigma^{-1/2}/4$ for $\sigma\in(0,\delta]$. The last condition we have to worry about is $\omega'(\delta^+)\leq\omega'(\delta^-)$, and this can be guaranteed by choosing $\delta$ small enough, since $\omega'(\delta^+)=\text{sech}^2(\mu)\leq1\leq 2=\omega'(0)$. Next, for any given initial data $u_0$, we chose $\lambda_0>1$ depending on $\delta$ and $\|u_0\|_{W^{1,\infty}}$ such that the vector field $u_0$ has $\lambda_0\omega(\lambda_0\xi)$ as a strict modulus of continuity, which can be done by virtue of Lemma \ref{buildmod}. The penultimate step in our construction is defining $\lambda\in C^1([0,\infty))$ as the solution to 
\begin{equation}\label{deflambda}
	\lambda'=C_{0}\lambda\log(\lambda),\quad \lambda(0)=\lambda_0\geq e,
\end{equation}
given explicitly by $\lambda(t)=\exp\left[\log(\lambda_0)\exp\left(C_{0}t\right)\right]$. Finally, we define 
\begin{equation}\label{defomega1}
\Omega(t,\xi):=\lambda(t)\omega(\lambda(t)\xi),\quad (t,\xi)\in[0,\infty)\times[0,\infty),
\end{equation}
where the fact that $\lambda$ is nondecreasing guarantees that $\Omega$ is a time-dependent strong modulus of continuity according to Definition \ref{timdepmoddef}, and thus satisfies \eqref{BCmodthm2}-\eqref{conczerothm2}. Let us now check that $\Omega$ as defined in \eqref{defomega1} indeed satisfies \eqref{modevol2}. It will be useful to introduce a ``dynamic variable'' $\sigma:=\xi\lambda(t)$, in which case a straightforward calculation and a change in variable in the integral would reduce \eqref{modevol2} to 
\begin{equation}\label{modevol2simp}
\lambda'\omega(\sigma)+\lambda'\sigma\omega'(\sigma)-4\lambda^3\omega''(\sigma)\geq \lambda^3\omega(\sigma)\omega'(\sigma)+C_{K,d}\lambda\left[\int_{0}^{3\sigma}\frac{\omega(\eta)}{\eta}d\eta+\sigma^{\rho}\int_{3\sigma}^{\infty}\frac{\omega(\eta)}{\eta^{1+\rho}}d\eta\right],
\end{equation}
and by periodicity, we would require that \eqref{modevol2simp} to hold true for every $(t,\sigma)\in(0,\infty)\times(0,2L\lambda\sqrt{d}]$. We consider the case when $\sigma\in(0,\delta)$ and $\sigma\in(\delta,2L\lambda\sqrt{d}]$ separately. 
\subsubsection{The case when $\sigma\in(0,\delta)$}
In this region, viscosity dominates and we rely only on the term $\omega''$ to absorb the instabilities arising from the nonlinear term and the nonlocal one. As $\lambda$ and $\omega$ are both non-decreasing and since $\lambda\geq1$, \eqref{modevol2simp} reduces to making sure that
\begin{equation*}\label{c1}
	4\omega''(\sigma)+\omega(\sigma)\omega'(\sigma)+C_{K,d}\left[\int_{0}^{3\sigma}\frac{\omega(\eta)}{\eta}d\eta+\sigma^{\rho}\int_{3\sigma}^{\infty}\frac{\omega(\eta)}{\eta^{1+\rho}}d\eta\right]\leq0.
\end{equation*}
Since $\omega(\sigma)\leq 2\sigma$ for any $\sigma\geq0$, we can estimate the first integral by
\begin{equation*}\label{c2}
	\int_{0}^{3\sigma}\frac{\omega(\eta)}{\eta}d\eta\leq 6\sigma, 
\end{equation*}
while requiring $3\delta\leq1$, we bound the second one by
\begin{equation}\label{c3}
	\sigma^{\rho}\int_{3\sigma}^{\infty}\frac{\omega(\eta)}{\eta^{1+\rho}}d\eta\leq \sigma^{\rho}\int_{3\sigma}^{1}\frac{\omega(\eta)}{\eta^{1+\rho}}d\eta+\sigma^{\rho}\int_{1}^{\infty}\frac{\omega(\eta)}{\eta^{1+\rho}}d\eta
\end{equation}
For the first integral in \eqref{c3}, we use $\omega(\sigma)\leq 2\eta$ and $\eta^{1+\rho}\geq\eta^2$ while for the second we use $\tanh(\zeta)\leq1$ for any $\zeta\geq0$ to get 
\begin{equation*}\label{c6}
\sigma^{\rho}\int_{3\sigma}^{\infty}\frac{\omega(\eta)}{\eta^{1+\rho}}d\eta\leq -2\sigma^{\rho}\log(3\sigma)+\rho^{-1}\sigma^{\rho}\leq -\rho^{-1}3\sigma^{\rho}\log(3\sigma).
\end{equation*}
From the definition of $\omega$ \eqref{statmodpfthm2}, we have 
\[
4\omega''(\sigma)+\omega(\sigma)\omega'(\sigma)\leq -3\sigma^{-1/2}+4\sigma.
\]
Therefore, putting all this together,
\begin{align*}
&4\omega''(\sigma)+\omega(\sigma)\omega'(\sigma)+C_{K,d}\left[\int_{0}^{3\sigma}\frac{\omega(\eta)}{\eta}d\eta+\sigma^{\rho}\int_{3\sigma}^{\infty}\frac{\omega(\eta)}{\eta^{1+\rho}}d\eta\right]\\
&\leq-3\sigma^{-1/2}+4\sigma-C_{K,d}\sigma^{\rho}\log(\sigma),	
\end{align*}
and the right hand side is certainly negative provided $\sigma\in(0,\delta)$, and $\delta$ is small enough.
\subsubsection{The case when $\sigma\in(\delta,2L\lambda\sqrt{d}]$}
Noting that for $\sigma>\delta$, $\omega(\sigma)$ is a stationary solution to the one dimensional viscous Burgers equation, we see that $-4\lambda^3\omega''(\sigma)\geq \lambda^3\omega(\sigma)\omega'(\sigma)$ is always true in this region. Thus, \eqref{modevol2simp} now reduces to making sure that
\begin{equation}\label{c8}
\lambda'\omega(\sigma)+\lambda'\sigma\omega'(\sigma)\geq C_{K,d}\lambda\left[\int_{0}^{3\sigma}\frac{\omega(\eta)}{\eta}d\eta+\sigma^{\rho}\int_{3\sigma}^{\infty}\frac{\omega(\eta)}{\eta^{1+\rho}}d\eta\right].
\end{equation}
We have two positive terms on the left-hand side. Note that the first integral in \eqref{c8} grows logarithmically in $\sigma$, while $\omega(\sigma)$ is bounded, and $\sigma\omega'(\sigma)$ is very small when $\sigma$ is large. This is exactly where we use periodicity in an essential way, and is the main obstacle in obtaining a similar result in the whole space. Further by invoking the change in variables in the integral, we ``transferred'' the singularity from 0 to the upper limit, hence the factor of $\log(\lambda)$ in \eqref{deflambda}. Since $\sigma\leq 2L\lambda\sqrt{d}$, we get
\begin{equation*}\label{c9}
\int_{0}^{3\sigma}\frac{\omega(\eta)}{\eta}d\eta\leq\int_{0}^{\delta}\frac{\omega(\eta)}{\eta}d\eta+\int_{\delta}^{6L\lambda\sqrt{d}}\frac{\omega(\eta)}{\eta}d\eta\leq2\delta+\log(6L\lambda\sqrt{d})-\log(\delta),
\end{equation*}
where we used $\omega(\sigma)\leq2 \sigma$ and $\omega(\sigma)\leq 1$ appropriately. We now also use $\omega(\sigma)\leq 1$ to get
\begin{equation*}\label{c10}
\sigma^{\rho}\int_{3\sigma}^{\infty}\frac{\omega(\eta)}{\eta^{1+\rho}}d\eta\leq\sigma^{\rho}\int_\sigma^{\infty}\eta^{-1-\rho} d\eta=\rho^{-1}.
\end{equation*}
Putting all this together, along with the fact that that both $\lambda$ and $\omega$ are nondecreasing, we get
\begin{align*}
&\lambda'\omega(\sigma)+\lambda'\sigma\omega'(\sigma)-C_{K,d}\lambda\left[\int_{0}^{3\sigma}\frac{\omega(\eta)}{\eta}d\eta+\sigma^{\rho}\int_{3\sigma}^{\infty}\frac{\omega(\eta)}{\eta^{1+\rho}}d\eta\right]\geq\lambda'(t)\omega(\delta)-C_{0}\lambda\log(\lambda)\end{align*}
and the right hand side is certainly positive by choice of $\lambda$ from \eqref{deflambda} with an appropriate $C_0$ depending on $d$, $K$ and $L$.
\subsection{Proof of Theorem \ref{thm3}}\label{pfthm3}
It should be clear at this point that our aim is to construct $\Omega$ such that 
\begin{align}
&\partial_t\Omega(t,\xi)-D_{\alpha}[\Omega](t,\xi)-g(t) \xi^{\beta}\partial_\xi\Omega(t,\xi)\geq0,&&\forall (t,\xi)\in(0,\infty)\times(0,\infty),\label{modevol3}\\
&|u_0(x)-u_0(y)|<\Omega(0,|x-y|),&&\forall x\neq y,\label{IC3}
\end{align}
with $D_1[\Omega](t,\xi)=4\partial^2_{\xi}\Omega(t,\xi)$ and 
\begin{align}\label{fracdifftr}
D_{\alpha}[\Omega](t,\xi)=&C_{\alpha}\int_0^{\xi/2}\frac{\Omega(t,\xi+2\eta)+\Omega(t,\xi-2\eta)-2\Omega(t,\xi)}{\eta^{1+2\alpha}}d\eta\nonumber\\
&+C_{\alpha}\int_{\xi/2}^{\infty}\frac{\Omega(t,\xi+2\eta)-\Omega(t,2\eta-\xi)-2\Omega(t,\xi)}{\eta^{1+2\alpha}}d\eta,
\end{align}
when $\alpha\in(0,1)$, together with the requirements \eqref{BCmodthm2}-\eqref{conczerothm2}. Here, we repeated the same analysis done in \S\ref{pfthm1} and \S\ref{pfthm2}, except that we used the H\"older condition assumed on the drift velocity to take care of the transport part, and extracted local dissipation from $(-\Delta)^{\alpha}$ via \eqref{fraclap}-\eqref{defdalpha} when $\alpha\in(0,1)$. Inspired by the proof of Theorem \ref{thm2}, the idea is to rescale a stationary modulus of continuity in matter that allows the parabolic part to win the tug-of-war against the transport. Let us assume that 
\[
\Omega(t,\xi):=\omega\left(\mu\xi\right),\quad \mu=\mu(t).
\]
As before, we work with the ``dynamic variable'' $\sigma:=\mu\xi$, and we invite the reader to readily verify that $D_{\alpha}[\Omega](t,\xi)=\mu^{2\alpha}D_{\alpha}[\omega](\sigma)$. Hence, we need to choose $\mu$ and $\omega$ such that
\[
\frac{\mu'}{\mu}\sigma\omega'(\sigma)-\mu^{2\alpha}D_{\alpha}[\omega](\sigma)-g\mu^{1-\beta}\sigma^\beta\omega'(\sigma)\geq0.
\]
Now, over small distances (when $\sigma\in(0,1]$), it should be clear that no matter how we choose $\mu$ (even if $\mu'=g\mu^{2-\beta}$), we can only rely on the Laplacian, as $\sigma\leq \sigma^{\beta}$. Thus, we would need to choose $\mu$ such that $\mu^{2\alpha+\beta-1}\approx g$, in order for both terms to be comparable. This makes the most natural choice for $\mu=g^{\gamma}$ where 
\[
\gamma:=\frac{1}{\beta+2\alpha-1},
\]
exactly as was defined in the statement of Theorem \ref{thm3}. This also tells us why the end point $\beta=1-2\alpha$ case is delicate. We still have to take care of the part coming from the time derivative, which now reads
\[
\gamma\frac{g'}{g}\sigma\omega'(\sigma).
\]
If $g$ oscillates, then this term becomes very hard to control in a uniform fashion (even if we consider an $\Omega(t,\xi)=\lambda(t)\omega(\mu\xi)$), which is the main reason behind our non-decreasing assumption on $g$, and somewhat explains why it is more natural to work with the quantity \eqref{altcondbintro} when using this ansatz. 

That being said, let us now proceed to constructing $\omega$. In what follows, $C_{\alpha}>0$ is a positive constant depending only on $\alpha$ and whose value may change from line to line. We let $\delta,r\in(0,1)$ be two parameters to be specified later, and define
\begin{equation}\label{defmodddstat}
\omega(\sigma):=
\begin{cases}
\sigma-\sigma^{2-r}, &\sigma\in[0,\delta],\\
\omega_{\mathcal{R}}(\sigma), &\sigma>\delta,
\end{cases}
\end{equation}
with $\omega_{\mathcal{R}}$ being defined as a solution to
\begin{equation}\label{ddode}
\begin{cases}
C_\alpha\sigma^{2-2\alpha}\omega_{\mathcal{R}}''(\sigma)+\sigma^{\beta}\omega_{\mathcal{R}}'(\sigma)=0,& \sigma\geq\delta,\\
\omega_{\mathcal{R}}(\delta)=\delta-\delta^{2-r},\\
\omega_{\mathcal{R}}'(\delta)=\frac{1}{2}-(2-r)\delta^{1-r}
\end{cases}
\end{equation}
given explicitly by 
\begin{equation*}\label{defmoddd}
\omega_{\mathcal{R}}(\sigma)=\omega_{\mathcal{R}}(\delta)+\omega_{\mathcal{R}}'(\delta)\int_{\delta}^{\sigma}\exp\left(\frac{-\eta^{\beta+2\alpha-1}}{C_\alpha[\beta+2\alpha-1]}\right)d\eta,
\end{equation*}
for some positive $C_{\alpha}$ that depends only on $\alpha$. One arrives at the above solution simply by making the substitution $v=\omega_{\mathcal{R}}'$ in the ODE appearing in \eqref{ddode}. The reader at this point may readily verify that $\omega$ as defined in \eqref{defmodddstat} is a strong modulus of continuity according to Definition \eqref{defmod} that happens to be concave. Since we need to satisfy \eqref{IC3}, we consider
\begin{equation}\label{defmodddtr}
\Omega(t,\xi):=B\omega(Bg^{\gamma}\xi),
\end{equation}
for some $B\geq1$ chosen according to Lemma \ref{buildmod} (recall that we assumed $g\geq1$). Note that $\Omega$ as defined in \eqref{defmodddtr} is indeed a time-dependent modulus of continuity. Furthermore, since solutions to \eqref{classicaldd} satisfy the maximum principle $\|u(t,\cdot)\|_{L^{\infty}}\leq \|u(0,\cdot)\|_{L^{\infty}}$ for all time, the hypothesis of Proposition \ref{breakthroughwholespace} is satisfied even though $\Omega(t,\cdot)$ is bounded (provided $B$ is large enough depending on $\|u_0\|_{L^{\infty}}$). Although we remark that in the periodic setting, one does not need to use the maximum principle. Of course, with such a choice of $\Omega$, we would need to require $g$ to be $C^1$ in time. We can make such an assumption without loss in generality, since the final estimate wouldn't depend on the $C^1$ norm of $g$. That is, we could prove such an estimate with a mollified drift velocity, obtain uniform estimates, then pass to limit using standard compactness arguments.

Let us now proceed to proving that $\Omega$ as defined in \eqref{defmodddtr}  is preserved. If we let $\sigma:=Bg^{\gamma}\xi$ and assume $g$ is non-decreasing, our task reduces to making sure that 
\[
B^{1+2\alpha}g^{2\alpha\gamma}\left[D_{\alpha}[\omega](\sigma)+B^{1-\beta-2\alpha}\sigma^{\beta}\omega'(\sigma)\right]\leq 0.
\]
As we assumed that $\beta+2\alpha-1>0$, $B\geq1$, and $g\geq1$, the problem boils down to
\begin{equation}\label{ineqdd}
D_{\alpha}[\omega](\sigma)+\sigma^{\beta}\omega'(\sigma)\leq 0.
\end{equation} 
Let us now estimate the dissipative part. First of all, notice that since $\omega$ is concave, both integrals appearing in \eqref{fracdifftr} are strictly negative. We will only make use of the first integral, and disregard the second one, i.e.
\[
D_{\alpha}[\omega](\sigma)\leq C_{\alpha}\int_0^{\sigma/2}\frac{\omega(\sigma+2\eta)+\omega(\sigma-2\eta)-2\omega(\sigma)}{\eta^{1+2\alpha}}d\eta.
\] 
To end up with a useful estimate, let us first assume that $\omega$ is smooth (the general case will be taken care of later). The second order Taylor expansion of $\omega$ tells us:
\[
\omega(\sigma\pm2\eta)=\omega(\sigma)\pm2\eta\omega'(\sigma)+2\eta^2\omega''(\sigma)\pm\frac{4}{3}\eta^3\omega'''(\sigma^*),
\]
for some $\sigma^*\in[\sigma-2\eta,\sigma+2\eta]$. It follows that 
\[
\omega(\sigma+2\eta)+\omega(\sigma-2\eta)-2\omega(\sigma)=4\eta^2\omega''(\sigma),
\]
which renders the left hand side of \eqref{ineqdd} controlled by 
\[
D_{\alpha}[\omega](\sigma)+\sigma^{\beta}\omega'(\sigma)\leq C_{\alpha}\sigma^{2-2\alpha}\omega''(\sigma)+\sigma^{\beta}\omega'(\sigma),
\]
with a possibly different $C_\alpha>0$. From \eqref{ddode}, we see that regardless of the value of $\delta, r\in(0,1)$, inequality \eqref{ineqdd} is always satisfied when $\sigma>\delta$. As for the case when $\sigma\in(0,\delta]$, we see that it reduces to making sure that
\[
-(2-r)(1-r)C_{\alpha}\sigma^{2-2\alpha-r}+\sigma^{\beta}\leq0.
\]
Equivalently, we need $\sigma^{\beta+2\alpha+r-2}\leq (2-r)(1-r)C_{\alpha}$, and for that to be possible, we must first guarantee that one can choose an $r\in(0,1)\cap(2-2\alpha-\beta,1)$. This is where the condition $\beta+2\alpha-1>0$ comes into play again, as it would imply that the set $(0,1)\cap(2-2\alpha-\beta,1)$ is not empty. With that in mind, one can choose the midpoint
\begin{equation}\label{defr}
r:=\frac{1}{2}+\frac{2-2\alpha-\beta}{2}.
\end{equation}
Having chosen $r$ as above, we finally chose $\delta$ small enough, depending only on the parameters $\alpha$ and $\beta$ such that $\delta^{\beta+r+2\alpha-2}\leq (2-r)(1-r)C_{\alpha}$. This would conclude the proof, as we have now ruled out the possibility of \eqref{modevol3} happening. This argument hinges upon the estimate $D_{\alpha}[\omega](\sigma)\leq C_{\alpha}\sigma^{2-2\alpha}\omega''(\sigma)$, which in turn relied upon assuming $\omega$ is $C^2$. One way to overcome this is as follows. We first extend $\omega$ in an odd fashion about $\sigma=0$, let $\chi_{\epsilon}$ be a standard (even) mollifier at level $\epsilon>0$, and define $\omega_{\epsilon}:=\omega*\chi_{\epsilon}$. The sequence of functions $\omega_{\epsilon}$ are all smooth on $[0,\infty)$. It follows that by the previous argument
\[
\int_0^{\sigma/2}\frac{\omega_{\epsilon}(\sigma+2\eta)+\omega_{\epsilon}(\sigma-2\eta)-2\omega_{\epsilon}(\sigma)}{\eta^{1+2\alpha}}d\eta= \frac{2^{2\alpha-1}}{1-\alpha}\sigma^{2-2\alpha}\omega_{\epsilon}''(\sigma).
\]
By virtue of Remark \ref{imprmk} and the fact that by construction $\omega'(\delta^+)<\omega'(\delta^-)$, we know $\sigma\neq \delta$, and as $\sigma\neq0$ as well, we may pass to limit $\epsilon\rightarrow0^+$ to get the desired estimate, since $\omega$ is smooth away from $\{0,\delta\}$. Of course, this argument has to be done only when $\alpha\in(0,1)$.

\subsection{Proof of Theorem \ref{thm4}: The Critical Case}\label{pfthm4a}
Let us now move on to adding incompressibility constraints. In this case, the inequality reads  
\begin{align}
&\partial_t\Omega(t,\xi)-D_{\alpha}[\Omega](t,\xi)-g(t) \xi^{\beta}\partial_\xi\Omega(t,\xi)-I(t,\xi)\geq0,&&\forall (t,\xi)\in(0,\infty)\times(0,\infty),\label{modevol4}\\
&|u_0(x)-u_0(y)|<\Omega(0,|x-y|),&&\forall x\neq y,\label{IC4}
\end{align}
where we now use Lemma \ref{pressest} with $\omega_{b}(t,\xi)=g(t)\xi^{\beta}$ to estimate the gradient of the pressure and obtain
\begin{equation*}\label{defI}
I(t,\xi)=C_dg(t)\left[\int_{0}^{\xi}\frac{\Omega(t,\eta)}{\eta^{2-\beta}}d\eta+\xi^{\beta}\int_{\xi}^{\infty}\frac{\Omega(t,\eta)}{\eta^{2}}d\eta+\frac{\xi^{\beta}\Omega(t,\xi)}{(1-\beta)\xi}\right]	.
\end{equation*}
Let us try to use the same ideas as in \S\ref{pfthm3}, and let us focus on the case with classical diffusion ($\alpha=1$). By using the same $\Omega(t,\xi)=\omega(\mu\xi)$ (with $\mu\gtrsim g^{\gamma}$) and performing similar calculation, we realize that we need to guarantee at least 
\begin{equation}\label{nweqpfthm4}
4\omega''(\sigma)+\sigma^{\beta}\omega'(\sigma)+\int_0^{\sigma}\eta^{\beta-2}\omega(\eta)d\eta\leq0.
\end{equation}
While this has a solution, we can no longer guarantee that $\omega$ is non-decreasing when $\sigma$ is large. To see this, if we set 
\[
\kappa:=\int_0^1\eta^{\beta-2}\omega(\eta)d\eta,
\]
then we note that for $\sigma\geq1$, if \eqref{nweqpfthm4} is true, we must have 
\[
4\omega''(\sigma)+\sigma^{\beta}\omega'(\sigma)\leq-\kappa,
\]
rendering
\[
\omega'(\sigma)\leq\left[\omega'(1)\exp\left(\frac{1}{4(\beta+1)}\right)-\frac{\kappa}{4}\int_1^{\sigma}\exp\left(\frac{\eta^{\beta+1}}{4(\beta+1)}\right)d\eta\right]\exp\left(-\frac{\sigma^{\beta+1}}{4(\beta+1)}\right),\quad \sigma\geq1.
\]
This is bad news, as we cannot seem to be able to find a way to make the argument work while relaxing the non-decreasing property required in Definition \ref{defmod}. In particular, Propositions \ref{breakthrough} and \ref{breakthroughwholespace} may no longer be true. 

As the reader may have realized by now, the general idea is that over small distances the  diffusive term is very strong, and can absorb almost any instabilities that may arise from the transport part. The difficulty lies in ruling out the breakthrough scenario over large distances. As the above heuristics suggest, we simply cannot rely on the dissipative part of the equation to take care of this over large distances when taking into account something besides the transport, while still maintaining the properties of being a modulus of continuity. Recall that we ran into a similar difficulty when proving Theorems \ref{thm1} and \ref{thm2}. Thus, if there is any hope in closing the argument here, one has to rely on the full parabolic operator, not just the elliptic one. That is, we need to consider a modulus of continuity of the form
\begin{equation}\label{deftransmodnse}
\Omega(t,\xi):=\lambda\omega(g^{\gamma}\xi),\ \lambda=\lambda(t),\ g=g(t).
\end{equation}
As the pressure scales like the transport term, it is natural to use the same choice of 
\begin{equation}\label{defgammaproof}
\gamma:=\frac{1}{\beta+2\alpha-1}.
\end{equation}
Further, as was done in the pure transport case, whenever $\sigma=g^\gamma\xi$ is of order 1 (that is, $\xi\in(0,g^{-\gamma})$), the Laplacian should be strong enough to balance everything out. It's power when $\sigma\gtrsim1$ can at most balance out the transport, as the above argument suggest. The pressure needs to be balanced out by the term $\lambda'(t)\omega(\sigma)$  when $\sigma\gtrsim 1$. That is 
\[
\lambda'\omega(\sigma)\gtrsim g\xi^{\beta-1}\lambda\omega(\sigma)=g^{1-\gamma(\beta-1)}\lambda\sigma^{\beta-1}\omega(\sigma),
\]
and from \eqref{defgammaproof}, we see that 
\begin{equation}\label{defpstar}
\lambda'(t)\approx g^{p*}(t)\lambda (t),\quad p^*:=\frac{2\alpha}{\beta+2\alpha-1},
\end{equation}
hence the reason why this particular construction would require a critical assumption on the drift.

Let us now make this rigorous. The easy choice would be to use the exact same stationary modulus from the previous section: as the equation is linear, multiplying it by a factor of $\lambda$ will not mess up the balance between transport and diffusion, while giving power to the time derivative. And this would work, we can use it to guarantee that \eqref{modevol4} is always satisfied. However, we remind the reader that the stationary modulus of continuity constructed in \S\ref{pfthm3} given by \eqref{defmodddstat}-\eqref{ddode} is bounded. This was not an issue there, since we had a maximum principle. Such a maximum principle is not readily available when adding incompressibility constraints, and so this construction would only work in the periodic setting. As we already saw that in order to balance out the instabilities that may arise from the pressure by moduli of continuity of the form \eqref{deftransmodnse} will downgrade our assumptions from supercritical to critical, we no longer have to rely purely on the dissipative term to balance out the transport over all $\sigma\in(0,\infty)$, provided we take care of it via the time derivative. This means we are free to use any concave modulus of continuity when $\sigma\geq1$, in particular, we can use an unbounded function, which always satisfy the hypothesis of Proposition \ref{breakthroughwholespace}.

We start by defining our $\Omega$ and then proceed to show that it must be preserved. As was done previously, given $\alpha\in(0,1]$ and $\beta\in (0,1)\cap(1-2\alpha,1)$, we set $r$ as in \eqref{defr}. Having done that, we let $\delta\in(0,1)$ be a small parameter to be defined later and we set
\begin{equation}\label{defstatmodnse}
\omega(\sigma):=
\begin{cases}
2\sigma-\sigma^{2-r}, &\sigma\in[0,\delta],\\
\delta\log(\sigma/\delta)+2\delta-\delta^{2-r}, &\sigma\in(\delta,\infty).
\end{cases}
\end{equation}
Next, we choose a $B\geq1$ such that the initial data strictly obeys $\omega(B\xi)$ according to Lemma \ref{buildmod} (since $\omega$ is unbounded, we do not need to consider $B\omega(B\xi)$). Nothing is special about the logarithmic function, one could choose for instance something that grows like $\sigma^{\kappa}$, for some $\kappa\in(0,1)$. This could potentially yield a sharper constant $B$, but we find working with logarithms to be somewhat easier. The proof for the critical case and super-critical case will require different choices of $\lambda$ and $\mu$, but is otherwise identical. Thus, rather than perform the same calculations twice, let us for now assume $\lambda=\lambda(t)\geq1$, $\mu=\mu(t)\geq g^{\gamma}(t)$ and define $\Omega$ as 
\begin{equation}\label{defomegatranscrit}
\Omega(t,\xi):=\lambda\omega(B\mu\xi).
\end{equation}
We point out that $\Omega$ as defined above satisfies \eqref{IC4}, since we assumed $g(t)\geq1$ for (almost) every $t$. Further, by virtue of remark \ref{imprmk}, we need not to worry about $\sigma=\delta$. As was done in \S\ref{pfthm3}, if we set  $\sigma:=B\mu\xi$, we know that 
\begin{equation*}\label{disspnse}
-D_{\alpha}[\Omega](t,\xi)=-(B\mu)^{2\alpha}\lambda D_{\alpha}[\omega](\sigma)\geq -C_{\alpha}(B\mu)^{2\alpha}\lambda \sigma^{2-2\alpha}\omega''(\sigma).
\end{equation*}
Let us now look at the transport term: 
\begin{equation*}\label{transportnse}
g(t)\xi^{\beta}\partial_\xi\Omega(t,\xi)=B\mu g\lambda \xi^{\beta}\omega'(\sigma)=(B\mu)^{1-\beta}g\lambda\sigma^{\beta}\omega'(\sigma).
\end{equation*}
Analogously, we have 
\begin{equation*}\label{singnse}
g\xi^{\beta-1}\Omega(t,\xi)=(B\mu)^{1-\beta}g\lambda\sigma^{\beta-1}\omega(\sigma),
\end{equation*}
and invoking a change of variable in the integrals (with a slight abuse of notation)
\begin{equation*}\label{intnse}
g(t)\left[\int_{0}^{\xi}\frac{\Omega(t,\eta)}{\eta^{2-\beta}}d\eta+\xi^{\beta}\int_{\xi}^{\infty}\frac{\Omega(t,\eta)}{\eta^{2}}d\eta\right]=(B\mu)^{1-\beta}g\lambda\left[\int_{0}^{\sigma}\frac{\omega(\eta)}{\eta^{2-\beta}}d\eta+\sigma^{\beta}\int_{\sigma}^{\infty}\frac{\omega(\eta)}{\eta^{2}}d\eta\right].
\end{equation*}
Thus, \eqref{modevol4} now reduces to the requirement
\begin{align*}
\lambda'\omega(\sigma)+&\frac{\mu'}{\mu}\lambda\sigma\omega'(\sigma)-C_{\alpha}(B\mu)^{2\alpha}\lambda \sigma^{2-2\alpha}\omega''(\sigma)\nonumber\\
&-C_d(B\mu)^{1-\beta}g\lambda\left[\sigma^{\beta}\omega'(\sigma)+\frac{1}{1-\beta}\sigma^{\beta-1}\omega(\sigma)+\int_{0}^{\sigma}\frac{\omega(\eta)}{\eta^{2-\beta}}d\eta+\sigma^{\beta}\int_{\sigma}^{\infty}\frac{\omega(\eta)}{\eta^{2}}d\eta\right]\geq0.
\end{align*}
This can be simplified via an integration by parts:
\[
\frac{1}{1-\beta}\sigma^{\beta-1}\omega(\sigma)+\int_{0}^{\sigma}\frac{\omega(\eta)}{\eta^{2-\beta}}d\eta=\frac{1}{1-\beta}\int_{0}^{\sigma}\frac{\omega'(\eta)}{\eta^{1-\beta}}d\eta,
\]
to get 
\begin{align}\label{modevolnse}
\lambda'\omega(\sigma)+&\frac{\mu'}{\mu}\lambda\sigma\omega'(\sigma)-C_{\alpha} (B\mu)^{2\alpha}\lambda\sigma^{2-2\alpha}\omega''(\sigma)\nonumber\\
&-C_d(B\mu)^{1-\beta}g\lambda\left[\sigma^{\beta}\omega'(\sigma)+\frac{1}{1-\beta}\int_{0}^{\sigma}\frac{\omega'(\eta)}{\eta^{1-\beta}}d\eta+\sigma^{\beta}\int_{\sigma}^{\infty}\frac{\omega(\eta)}{\eta^{2}}d\eta\right]\geq0.
\end{align}
As was discussed in the case of the transport-diffusion, the ``natural'' choice for $\mu$ is 
\[
\mu(t):=g^{\gamma}(t),\quad \gamma=\frac{1}{2\alpha+\beta-1},
\]
and from the heuristic argument \eqref{defpstar}, we choose 
\[
\lambda'(t)=C_{d,\beta,\alpha}B^{1-\beta}g^{p^*}(t)\lambda(t),\quad \lambda(0)=1,
\]
for some constant $C_{d,\beta,\alpha}>0$ to be determined and 
\[
p^*:=\frac{2\alpha}{2\alpha+\beta-1}.
\]
With such a choice, and the assumption that $g$ is non-decreasing, \eqref{modevolnse} now reduces to making sure that
\begin{align}\label{modevolnsecrit}
\lambda'\omega(\sigma)&-C_{\alpha}\lambda B^{2\alpha}g^{p^*}\sigma^{2-2\alpha}\omega''(\sigma)\nonumber\\
&-C_dB^{1-\beta}g^{p^*}\lambda\left[\sigma^{\beta}\omega'(\sigma)+\frac{1}{1-\beta}\int_{0}^{\sigma}\frac{\omega'(\eta)}{\eta^{1-\beta}}d\eta+\sigma^{\beta}\int_{\sigma}^{\infty}\frac{\omega(\eta)}{\eta^{2}}d\eta\right]\geq0.
\end{align}
We analyze the case when $\sigma\in(0,\delta)$ and $\sigma\in(\delta,\infty)$ separately.
\subsubsection{The case when $\sigma\in(0,\delta)$}\label{seccritnsesmalldist}
From $\omega'(\sigma)\leq2$ we see that
\[
\sigma^{\beta}\omega'(\sigma)+\frac{1}{1-\beta}\int_{0}^{\sigma}\frac{\omega'(\eta)}{\eta^{1-\beta}}d\eta\leq C_{\beta}\sigma^{\beta}.
\]
For the other integral, we use $\omega(\eta)\leq2\eta$ on $[\sigma,\delta]$ and $\omega(\eta)\leq\delta\log(\eta/\delta)+2\delta$ on $(\delta,\infty)$ to get
\[
\int_{\sigma}^{\infty}\frac{\omega(\eta)}{\eta^{2}}d\eta=\int_{\sigma}^{\delta}\frac{\omega(\eta)}{\eta^{2}}d\eta+\int_{\delta}^{\infty}\frac{\omega(\eta)}{\eta^{2}}d\eta\leq \log(\delta/\sigma)+\delta\int_{\delta}^{\infty}\frac{\log(\eta/\delta)+2}{\eta^{2}}d\eta\leq 3-2\log(\sigma),
\]
from which we get, for $\sigma\in(0,\delta)\subset(0,1)$,
\begin{equation}\label{destabsmalldist}
\sigma^{\beta}\omega'(\sigma)+\frac{1}{1-\beta}\int_{0}^{\sigma}\frac{\omega'(\eta)}{\eta^{1-\beta}}d\eta+\sigma^{\beta}\int_{\sigma}^{\infty}\frac{\omega(\eta)}{\eta^{2}}d\eta\leq -C_{\beta}\sigma^{\beta}\log(\sigma).
\end{equation}
As $B\geq1$, $\beta+2\alpha-1>0$ and $\lambda'\geq0$, if we plug this into the left-hand side of \eqref{modevolnsecrit} we get a lower bound
\[
-C_\alpha B^{2\alpha}g^{p^*}\lambda\left[\sigma^{2-2\alpha}\omega''(\sigma)-C_{d,\alpha,\beta}\sigma^{\beta}\log(\sigma)\right],
\]
and as $\omega''(\sigma)=-(2-r)(1-r)\sigma^{-r}$, we now need 
\[
-(2-r)(1-r)-C_{d,\alpha,\beta}\sigma^{r+2\alpha+\beta-2}\log(\sigma)\leq0.
\]
Our choice of $r$ from \eqref{defr} ensures that $r+2\alpha+\beta-2>0$, and so one can chose a small enough $\delta=\delta_{d,\alpha,\beta}$ such that the above is true whenever $\sigma\in(0,\delta)$. 

\subsubsection{The case when $\sigma\in(\delta,\infty)$}\label{seccritnselargedist}
Let us start by using $\omega'(\sigma)\leq 2$ always and $\omega'(\sigma)=\delta\sigma^{-1}$ on $(\delta,\infty)$ to get 
\[
\int_{0}^{\sigma}\frac{\omega'(\eta)}{\eta^{1-\beta}}d\eta=\int_{0}^{\delta}\frac{\omega'(\eta)}{\eta^{1-\beta}}d\eta+\int_{\delta}^{\sigma}\frac{\omega'(\eta)}{\eta^{1-\beta}}d\eta\leq2\beta^{-1}\delta^{\beta}+\delta\int_\delta^{\sigma}\eta^{\beta-2}d\eta\leq C_\beta\delta^\beta.
\]
As for the second integral, we may integrate by parts to get 
\[
\sigma^{\beta}\int_{\sigma}^{\infty}\frac{\omega(\eta)}{\eta^{2}}d\eta=\sigma^{\beta}\int_\sigma^{\infty}\frac{\omega'(\eta)}{\eta}d\eta+\sigma^{\beta-1}\omega(\sigma)=\delta\sigma^{\beta-1}+\sigma^{\beta-1}\omega(\sigma), 
\]
As $\omega'(\sigma)=\delta\sigma^{-1}$, we now have the bound 
\[
\sigma^{\beta}\omega'(\sigma)+\frac{1}{1-\beta}\int_{0}^{\sigma}\frac{\omega'(\eta)}{\eta^{1-\beta}}d\eta+\sigma^{\beta}\int_{\sigma}^{\infty}\frac{\omega(\eta)}{\eta^{2}}d\eta\leq 2\delta\sigma^{\beta-1}+C_{\beta}\delta^{\beta}+\sigma^{\beta-1}\omega(\sigma),
\]
which upon noting that $\omega(\sigma)\geq\delta$ and $\sigma^{\beta-1}\leq\delta^{\beta-1}$ whenever $\sigma\in(\delta,\infty)$ we end up with 
\begin{equation}\label{destablargedist}
\sigma^{\beta}\omega'(\sigma)+\frac{1}{1-\beta}\int_{0}^{\sigma}\frac{\omega'(\eta)}{\eta^{1-\beta}}d\eta+\sigma^{\beta}\int_{\sigma}^{\infty}\frac{\omega(\eta)}{\eta^{2}}d\eta\leq C_{\beta}\delta^{\beta-1}\omega(\sigma).
\end{equation}
Plugging into the left hand side of \eqref{modevolnsecrit} and using the concavity of $\omega$ we see that our requirement now is 
\[
\lambda'\omega(\sigma)-C_{d,\beta}(B/\delta)^{1-\beta}g^{p^*}\lambda\omega(\sigma)\geq0,
\]
which is satisfied provided we define $C_{d,\alpha,\beta}:=C_{d,\beta}\delta^{\beta-1}$ and set
\[
\lambda'(t)=C_{d,\alpha,\beta}B^{1-\beta}g^{p^*}(t)\lambda(t),\quad \lambda(0)=1.
\]

\subsection{Proof of Theorem \ref{thm4}: The Super-Critical Case}\label{pfthm4b}
Let us now turn to proving the partial regularity result. We start by recalling that when dynamically rescaling a stationary modulus of continuity, $\Omega(t,\xi)=\lambda\omega(\mu \xi)$, in order to balance out the transport part by the diffusive term over small distances, we roughly needed $\mu\gtrsim g^{\gamma}$, while to take care of the pressure term over large distances (for instance the term $\xi^{\beta-1}g\Omega$) we need
\begin{equation}\label{grwthlmbapf}
\lambda'(t)\approx g\mu^{1-\beta}\lambda.
\end{equation}
So the question is, what are our options for choosing $\mu$? Our aim here is to try and dip below the critical level, so that we would like to choose $\mu$ to meet those ends. As we saw in \S\ref{pfthm4a}, choosing the $\mu$ to be ``right'' power of $g$ would lead to a critical assumption, so we need something else. Since $\gamma\leq p^*$ when $\alpha\in[1/2,1]$, and $\gamma>p^*$ when $\alpha\in(0,1/2)$ (where $p^*$ is the critical exponnent), we would need to consider the two cases separately in order to obtain optimal results.
\subsubsection{The case when $\alpha\in[1/2,1]$}
Our target here is to try and deal with the pressure term by only assuming that at most $g\in L^{\gamma}(0,T)$ (and is non-decreasing), the same supercritical assumptions that lead to regularity in the classical drift-diffusion problem. From \eqref{grwthlmbapf}, if we are to go below $L^1(0,T)$ (say when $\alpha=1$), we need to make use of the factor $\mu^{-\beta}$ to bring down the power of $g$, all the while avoiding a Riccati equation for $\lambda$. Thus, a natural choice is $\mu=\log(\lambda)$. With such a choice for $\mu$, we would need to make sure that $\mu=\log (\lambda)\gtrsim g^{\gamma}$ in order to make sure dissipation prevails over short distances. If this condition is satisfied, then we would have reduced the power of $g$ by a factor of $\mu^{-\beta}\lesssim g^{-\gamma\beta}$, and so we would need $\lambda$ to grow as $g^{1-\gamma\beta}\lambda\log(\lambda)$. Notice that with our definition of $\gamma$, we have $1-\gamma\beta=\gamma(2\alpha-1)$, which in fact is less than $\gamma$ (better than we hoped for), but is still non-negative since $\alpha\in[1/2,1]$.

 To make the above heuristics rigorous, we proceed in a manner similar to what was done previously. Namely we will let $\Omega$ be as in \eqref{defomegatranscrit} (with the same choice of $B$, $r$ and $\delta$), except we now let $\lambda$ be any solution to 
\begin{align}
&\lambda'(t)\geq C_{d,\beta,\alpha}B^{1-\beta}g^{1-\gamma\beta}\lambda\log(\lambda),\label{defodelambdapfthm4b}\\
&\log(\lambda)\geq g^{\gamma}\label{condloglambda},
\end{align}
where $C_{d,\beta,\alpha}$ is some constant, and set $\mu(t):=\log(\lambda(t))$. For instance, since we are assuming $g$ is non-decreasing, one may choose
\[
\log(\lambda(t)):=g^{\gamma}(t)\exp\left(C_{d,\alpha,\beta}B^{1-\beta}\int_0^tg^{1-\gamma\beta}(s)ds\right).
\]
With such choices of $\lambda$ and $\mu$, we invite the reader to readily verify that $\Omega$ will now satisfy \eqref{modevolnse} by repeating the exact same calculations as in \S\ref{pfthm4a}: for $\sigma\in(0,\delta)$, \eqref{destabsmalldist} and the fact that $\lambda$ and $\mu$ are both non-decreasing tell us that the left hand side of \eqref{modevolnse} is bounded from below by 
\[
-C_{\alpha}(B\mu)^{2\alpha}\lambda\left[\sigma^{2-2\alpha}\omega''(\sigma)-C_{d,\alpha,\beta}(B\mu)^{1-2\alpha-\beta}g\sigma^{\beta}\log(\sigma)\right].
\]
From $B\geq1$, $\beta+2\alpha-1>0$ and $\mu=\log(\lambda)\geq g^{\gamma}$, we arrive at requiring 
\[
\sigma^{2-2\alpha}\omega''(\sigma)-C_{d,\alpha,\beta}\sigma^{\beta}\log(\sigma)\leq0,
\]
thus the same choice of $r\in(0,1)$ and $\delta\in(0,1)$ from \S\ref{seccritnsesmalldist} would work. For $\sigma>\delta$, we would use \eqref{destablargedist} and concavity of $\omega$ to arrive at the following lower bound for the left-hand side of \eqref{modevolnse}:
\begin{equation}\label{calc1supercritpfa}
\lambda'(t)\omega(\sigma)-C_{d,\alpha,\beta}B^{1-\beta}\mu^{1-\beta}g\lambda\omega(\sigma)=\left[\lambda'-C_{d,\alpha,\beta}\mu^{-\beta}gB^{1-\beta}\mu\lambda\right]\omega(\sigma).
\end{equation}
From our choice of $\mu=\log(\lambda)\geq g^{\gamma}$ we would get 
\[
\mu^{-\beta}g\leq\frac{g}{\log^{\beta}(\lambda)}\leq \frac{g}{g^{\gamma\beta}}=g^{1-\gamma\beta},
\]
and utilizing \eqref{defodelambdapfthm4b}, we therefore are guaranteed that the right-hand side of \eqref{calc1supercritpfa} is non-negative. This give us the first part of estimate \eqref{supercritbdnsethm}.
\subsubsection{The case when $\alpha\in(0,1/2)$}
Let us recall that we need $\lambda$ to grow as $\lambda'(t)\approx g\mu^{1-\beta}\lambda$, and that $\mu\gtrsim g^{\gamma}$. If we make the same choices as before, we end up with $\lambda'(t)\approx g^{\gamma(2\alpha-1)}\lambda\log\lambda$. When $\alpha\in(0,1/2)$, we need not to worry about integrability of the term $g^{\gamma(2\alpha-1)}$, since it becomes bounded uniformly by 1. On the other hand, the estimate that we will get will be of the form 
\[
\|\nabla u(t,\cdot)\|_{L^{\infty}}\lesssim \lambda(t)\log(\lambda(t)),\quad \log(\lambda)\approx g^{\gamma}(t),
\]
which renders 
\[
\log\left(\|\nabla u(t,\cdot)\|_{L^{\infty}}\right)\lesssim \log(\lambda(t))+\log\log(\lambda(t)),\quad \log(\lambda)\approx g^{\gamma}(t),
\]
and thus to get the partial regularity result would still require the same subcritical assumption when $\alpha\in(0,1/2)$ we had when proving regularity (which is pointless). This is exactly why we need to consider both cases separately. In particular, since $1<p^*$, we could choose $\mu=\log^{\kappa}\lambda$, some $\kappa>1$ in this case, and use the term $\mu^{-\beta}$ to avoid solving a ``logarithmic'' Riccati equation, rather than decrease the exponent of $g$ (which is already at the supercritical level). We would then use this extra power of $\log$ to increase the dissipative power, since our requirement would read $\log(\lambda)\approx g^{\gamma/\kappa}$, and so 
\[
\log\left(\|\nabla u(t,\cdot)\|_{L^{\infty}}\right)\lesssim \log(\lambda(t))+\kappa\log\log(\lambda(t)),\quad \log(\lambda)\approx g^{\gamma/\kappa}(t).
\]
The question is how large can we choose $\kappa$. To avoid solving $\lambda'=g\lambda\log^{1+\varrho}(\lambda)$, some $\varrho>0$, the best we could do is set $\kappa:=1/(1-\beta)$. With such a choice, we get 
\[
\frac{\gamma}{\kappa}=\frac{1-\beta}{\beta+2\alpha-1}<\frac{2\alpha}{\beta+2\alpha-1}=p^*,
\]
since we assumed that $\beta>1-2\alpha$. Thus, we would end up with a supercritical assumption, namely $g\in L^p(0,T)$, where $p:=\max\{1,\kappa/\gamma\}<p^*$, as desired.

With those remarks in mind, we set $\kappa:=1/(1-\beta)$, define $\Omega$ as in \eqref{defomegatranscrit} (with the same choice of $B$, $r$ and $\delta$), except we now let $\lambda$ be any solution to 
\begin{align*}
&\lambda'(t)\geq C_{d,\beta,\alpha}B^{1-\beta}g\lambda\log(\lambda),\\
&\log(\lambda)\geq g^{\gamma/\kappa},
\end{align*}
where $C_{d,\beta,\alpha}$ is the same constant as before and set $\mu(t):=\log^{\kappa}(\lambda(t))$. For instance, we could choose 
\[
\log(\lambda(t)):=g^{\gamma/\kappa}(t)\exp\left(C_{d,\alpha,\beta}B^{1-\beta}\int_0^tg(s)ds\right).
\]
The reader may now repeat the same calculations to show that such an $\Omega$ does indeed satisfy \eqref{modevolnse}, Definition \ref{timdepmoddef} and the hypothesis of Proposition \ref{breakthroughwholespace}. Thus, the second part of estimate \eqref{supercritbdnsethm} holds true.
\section*{Acknowledgments}
The author would like to thank Edriss Titi and Tarek Elgindi for several useful discussions as well as Peter Constantin and Theodore Drivas for helpful remarks on an earlier version of this manuscript. Moreover, the author thanks Titi for support and the research group of Rupert Klein at the Freie Univerit\"{a}t Berlin for their kind hospitality. Finally, my gratitude extends to Peter Kuchment for useful editorial remarks.
\bibliographystyle{abbrv}
\bibliography{mybib}
\end{document}